\documentclass[reqno]{amsart}

\usepackage[centertags]{amsmath}
\usepackage[colorlinks]{hyperref}
\usepackage{amsfonts}
\usepackage{MnSymbol}
\usepackage{amsthm}
\usepackage{newlfont}
\usepackage{amscd}
\usepackage{amsmath}
\usepackage{enumerate}
\usepackage[all,2cell]{xy}
\UseAllTwocells
\input xy
\xyoption{2cell}
\xyoption{all}
\usepackage{verbatim}
\usepackage{eucal}
\usepackage{xpatch}
\usepackage{graphicx}

\newcommand{\pbc}[1][dr]{\save*!/#1-1.8pc/#1:(-1,1)@^{|-}\restore}

\newcommand{\NN}{\mathbb{N}}
\newcommand{\GG}{\mathfrak{G}}

\newcommand{\ZZ}{\mathbb{Z}}
\newcommand{\LL}{\mathbb{L}}
\newcommand{\RR}{\mathbb{R}}
\newcommand{\BB}{\mathbb{B}}
\newcommand{\EE}{\mathbb{E}}

\def\A{\mathcal{A}}
\def\B{\mathcal{B}}

\def\C{\mathcal{C}}
\def\D{\mathcal{D}}
\def\G{\mathcal{G}}

\def\F{\mathcal{F}}
\def\SS{\mathcal{S}}
\def\M{\mathcal{M}}
\def\N{\mathcal{N}}

\def\I{\mathcal{I}}

\def\T{\mathcal{T}}
\def\P{\mathcal{P}}

\def\R{\mathcal{R}}
\renewcommand{\L}{\mathcal{L}}
\def\U{\mathcal{U}}

\def\O{\mathcal{O}}

\def\wtl{\widetilde}

\def\E{\mathcal{E}}
\def\bS{\mathbf{S}}
\DeclareMathOperator\bbS{\mathbb{S}}
\def\fC{\mathfrak{C}}
\def\rN{\mathrm{N}}
\def\rH{\mathrm{H}}
\def\K{\mathcal{K}}

\def\alp{{\alpha}}

\def\gam{{\gamma}}

\def\eps{{\varepsilon}}

\def\lam{{\lambda}}

\def\Om{{\Omega}}

\def\Del{{\Delta}}
\def\Sig{{\Sigma}}

\def\vphi{\varphi}

\newcommand{\HSwarrow}{\kern0.05ex\vcenter{\hbox{\Huge\ensuremath{\Swarrow}}}\kern0.05ex}
\newcommand{\hSwarrow}{\kern0.05ex\vcenter{\hbox{\huge\ensuremath{\Swarrow}}}\kern0.05ex}
\newcommand{\LLSwarrow}{\kern0.05ex\vcenter{\hbox{\LARGE\ensuremath{\Swarrow}}}\kern0.05ex}
\newcommand{\LSwarrow}{\kern0.05ex\vcenter{\hbox{\Large\ensuremath{\Swarrow}}}\kern0.05ex}

\newcommand{\HSearrow}{\kern0.05ex\vcenter{\hbox{\Huge\ensuremath{\Searrow}}}\kern0.05ex}
\newcommand{\hSearrow}{\kern0.05ex\vcenter{\hbox{\huge\ensuremath{\Searrow}}}\kern0.05ex}
\newcommand{\LLSearrow}{\kern0.05ex\vcenter{\hbox{\LARGE\ensuremath{\Searrow}}}\kern0.05ex}
\newcommand{\LSearrow}{\kern0.05ex\vcenter{\hbox{\Large\ensuremath{\Searrow}}}\kern0.05ex}

\newcommand{\HDownarrow}{\kern0.05ex\vcenter{\hbox{\Huge\ensuremath{\Downarrow}}}\kern0.05ex}
\newcommand{\hDownarrow}{\kern0.05ex\vcenter{\hbox{\huge\ensuremath{\Downarrow}}}\kern0.05ex}
\newcommand{\LLDownarrow}{\kern0.05ex\vcenter{\hbox{\LARGE\ensuremath{\Downarrow}}}\kern0.05ex}
\newcommand{\LDownarrow}{\kern0.05ex\vcenter{\hbox{\Large\ensuremath{\Downarrow}}}\kern0.05ex}

\newcommand{\HUparrow}{\kern0.05ex\vcenter{\hbox{\Huge\ensuremath{\Uparrow}}}\kern0.05ex}
\newcommand{\hUparrow}{\kern0.05ex\vcenter{\hbox{\huge\ensuremath{\Uparrow}}}\kern0.05ex}
\newcommand{\LLUparrow}{\kern0.05ex\vcenter{\hbox{\LARGE\ensuremath{\Uparrow}}}\kern0.05ex}
\newcommand{\LUparrow}{\kern0.05ex\vcenter{\hbox{\Large\ensuremath{\Uparrow}}}\kern0.05ex}

\newtheorem{thm}{Theorem}[subsection]
\newtheorem{cor}[thm]{Corollary}
\newtheorem{lem}[thm]{Lemma}
\newtheorem{pro}[thm]{Proposition}

\makeatletter
\@addtoreset{thm}{section}
\makeatother

\numberwithin{thm}{subsection}
\numberwithin{equation}{subsection}

\theoremstyle{definition}
\newtheorem{define}[thm]{Definition}

\newtheorem{example}[thm]{Example}

\newtheorem{defn}[thm]{Definition}

\theoremstyle{remark}
\newtheorem{rem}[thm]{Remark}

\DeclareMathOperator{\holim}{holim}

\DeclareMathOperator{\colim}{colim}
\DeclareMathOperator{\Ker}{Ker}

\DeclareMathOperator{\Id}{Id}
\DeclareMathOperator{\id}{id}
\DeclareFontFamily{OT1}{pzc}{}
\DeclareFontShape{OT1}{pzc}{m}{it}{<-> s * [1.10] pzcmi7t}{}
\DeclareMathAlphabet{\mathpzc}{OT1}{pzc}{m}{it}
\DeclareMathOperator{\cof}{cof}
\DeclareMathOperator{\fib}{fib}
\DeclareMathOperator{\ad}{ad}

\DeclareMathOperator{\Alg}{Alg}

\DeclareMathOperator{\LMod}{LMod}

\DeclareMathOperator{\ModCat}{ModCat}
\DeclareMathOperator{\sGr}{sGr}

\DeclareMathOperator{\diag}{diag}
\DeclareMathOperator{\op}{op}
\DeclareMathOperator{\Map}{Map}
\DeclareMathOperator{\red}{red}

\DeclareMathOperator{\df}{def}
\DeclareMathOperator{\Set}{Set}

\DeclareMathOperator{\Cat}{Cat}

\DeclareMathOperator{\CAlg}{CAlg}
\DeclareMathOperator{\Mod}{Mod}

\DeclareMathOperator{\Fun}{Fun}

\DeclareMathOperator{\Un}{Un}
\DeclareMathOperator{\Ob}{Ob}
\DeclareMathOperator{\Sp}{Sp}
\DeclareMathOperator{\Ho}{Ho}

\DeclareMathOperator{\ev}{ev}

\DeclareMathOperator{\dom}{dom}

\DeclareMathOperator{\Ab}{Ab}

\DeclareMathOperator{\Tw}{Tw}
\DeclareMathOperator{\der}{h}
\DeclareMathOperator{\BiMod}{BiMod}
\DeclareMathOperator{\aug}{aug}

\DeclareMathOperator{\St}{St}
\DeclareMathOperator{\Spectra}{Spectra}
\DeclareMathOperator{\Reedy}{Reedy}

\DeclareMathOperator{\cosk}{cosk}

\DeclareMathOperator{\Idem}{Idem}
\DeclareMathOperator{\cone}{cone}
\DeclareMathOperator{\hocofib}{hocofib}
\DeclareMathOperator{\cov}{cov}
\DeclareMathOperator{\cod}{cod}

\DeclareMathOperator{\cls}{cl}

\def\x{\overset}

\def\Hom{\textrm{Hom}}
\def\End{\textrm{End}}

\def\inv{\textup{inv}}

\newcommand{\tgpd}{\kern0.05ex\vcenter{\hbox{\footnotesize\ensuremath{2}}}\kern0.05ex\mathcal{G}pd} 

\def\lrar{\longrightarrow}
\def\llar{\longleftarrow}
\def\hrar{\hookrightarrow}

\newcommand{\adj}{\mathrel{\substack{\longrightarrow \\[-.6ex] \x{\upvdash}{\longleftarrow}}}}

\def\ovl{\overline}

\def\bar{\overline}
\def\uline{\underline}

\title{The abstract cotangent complex and Quillen cohomology of enriched categories}

\author{Yonatan Harpaz}
\author{Joost Nuiten}
\author{Matan Prasma}

\date{}
\begin{document}

\begin{abstract}
In his fundamental work, Quillen developed the theory of the cotangent complex as a universal abelian derived invariant, and used it to define and study a canonical form of cohomology, encompassing many known cohomology theories. Additional cohomology theories, such as generalized cohomology of spaces and topological Andr\'e-Quillen cohomology, can be accommodated by considering a spectral version of the cotangent complex. Recent work of Lurie established a comprehensive $\infty$-categorical analogue of the cotangent complex formalism using stabilization of $\infty$-categories. In this paper we study the spectral cotangent complex while working in Quillen's model categorical setting. Our main result gives new and explicit computations of the cotangent complex and Quillen cohomology of enriched categories. For this we make essential use of previous work, which identifies the tangent categories of operadic algebras in unstable model categories. In particular, we present the cotangent complex of an $\infty$-category as a spectrum valued functor on its twisted arrow category, and consider the associated obstruction theory in some examples of interest.
\end{abstract}

\maketitle

\tableofcontents

\section{Introduction}
The use of cohomological methods to study objects and maps pervades many areas in mathematics. Although cohomology groups come in various forms, they share many abstract properties and structures, which can be organized by means of abstract homological algebra. In his seminal book~\cite{Qui67}, Quillen pioneered a way of not just reorganizing homological algebra, but also extending it outside the additive realm. This allows one, in particular, to study cohomology groups, the objects one takes cohomology of, and even the \textbf{operation} of taking cohomology itself, all on a single footing.
In Quillen's approach, the homology of an object is obtained by \textbf{deriving} its abelianization.  
More precisely, given a model category $\M$, one may consider the category $\Ab(\M)$ of \textbf{abelian group objects} in $\M$, namely, objects $M \in \M$ equipped with maps $u: \ast_{\M} \lrar M$, $m: M \times M \lrar M$ and $\inv: M \lrar M$ satisfying (diagramatically) all the axioms of an abelian group. Under suitable conditions, the category $\Ab(\M)$ carries a model structure so that the free-forgetful adjunction
$$ \F: \M \adj \Ab(\M): \U $$
is a Quillen adjunction. Given an abelian group object $M \in \Ab(\M)$ and an integer $n\geq 0$, Quillen defined the $n$'th cohomology group of $X$ with coefficients in $M$ by the formula
$$ \rH^n_Q(X,M) \x{\df}{=} \pi_0\Map^{\der}_{\Ab(\M)}(\LL\F(X),\Sig^nM). $$
In particular, all Quillen cohomology groups of $X$ are simple invariants of the derived abelianization $\LL\F(X)$ of $X$: one just takes homotopy classes of maps into the linear object $M$. 

The universality of the constructions $X \mapsto \F(X)$ is best understood on the categorical level.  
Recall that a locally presentable category $\C$ is called \textbf{additive} if it is tensored over the category $\Ab$ of abelian groups. In this case the tensoring is essentially unique and induces a natural enrichment of $\C$ in $\Ab$. If $\D$ is a locally presentable category then the category $\Ab(\D)$ of abelian group objects in $\D$ is additive and the free abelian group functor $\D \lrar \Ab(\D)$ exhibits $\Ab(\D)$ as universal among the additive categories receiving a colimit preserving functor from $\D$. 
In this sense the free abelian group functor (and consequently the classical notion of Quillen cohomology which is based on it) is uniquely determined by our notion of what an additive category is, which in turn is completely determined by its universal example - the category of abelian groups.

In the case of simplicial sets, $\LL\F(X)$ is given by the free simplicial abelian group $\ZZ X$ generated from $X$. The classical Dold-Thom theorem then shows that Quillen cohomology reproduces ordinary cohomology and more generally, ordinary cohomology with coefficients in a simplicial abelian group.
The quest for more refined invariants has led to the axiomatization of \textbf{generalized} homology and cohomology theories, and to their classification via the notion of \textbf{spectra}. The passage from ordinary cohomology to generalized cohomology therefore highlights spectra as a natural extension of the notion of ``linearity'' provided by simplicial abelian groups. 
Indeed, all generalized cohomology groups of a space $X$ are now determined not by $\ZZ X$, but by the free spectrum $\Sig^{\infty}_+X$ generated from $X$, also known as the \textbf{suspension spectrum} of $X$. 

The passage from simplicial abelian groups to spectra is a substantial one, even from a homotopy-theoretic point of view. Indeed, in homotopy theory the notion of a spectrum is preceded by the natural notion of an \textbf{$\EE_\infty$-group}, obtained by interpreting the axioms of an abelian group not strictly, but up to coherent homotopy. It then turns out that specifying an $\EE_\infty$-group structure on a given space $X_0$ is equivalent to specifying, for every $n \geq 1$, an $(n-1)$-connected space $X_n$, together with a weak equivalence $X_{n-1} \x{\simeq}{\lrar} \Om X_n$. Such a datum is also known as a \textbf{connective spectrum}, and naturally extends to the general notion of a spectrum by removing the connectivity conditions on $X_n$. This passage from connective spectra (or $\EE_\infty$-groups) to spectra should be thought of as an extra linearization step that is possible in a homotopical setting, turning additivity into \textbf{stability}. Using stability as the fundamental form of linearity is the starting point for the theory of \textbf{Goodwillie calculus}, which extends the notion of stability to give meaningful analogues to higher order approximations, derivatives and Taylor series for functors between $\infty$-categories. 

The universality of the constructions $X \mapsto \Sig^{\infty}_+X$ can be understood in a way analogous to that of $X \mapsto \ZZ X$  
by working in a higher categorical setting. To this end we may organize the collection of spectra into an $\infty$-category $\Sp$, which is presentable and symmetric monoidal. Applying the above logic we may replace the notion of an additive category by that of a presentable $\infty$-category which is tensored over $\Sp$, which coincides with the notion of being a \textbf{stable presentable $\infty$-category}. Given a presentable $\infty$-category $\D$, there exists a universal stable presentable $\infty$-category $\Sp(\D)$ receiving a colimit preserving functor $\Sig^{\infty}_+:\D \lrar \Sp(\D)$. One may realize $\Sp(\D)$ as the $\infty$-category of \textbf{spectrum objects} in $\D$, also known as the \textbf{stabilization} of $\D$.


A key insight of stable homotopy theory is that the various favorable formal properties of cohomology theories are essentially a consequence of the stability of the $\infty$-category of spectra.
It is hence worthwhile, when studying a particular presentable $\infty$-category $\D$, to look for invariants 
which take values in a stable $\infty$-category, i.e.\ ``linear invariants''. If we restrict attention to linear invariants that preserve colimits (such as generalized cohomology theories), then there is a universal such invariant. This invariant takes values in the stabilization $\Sp(\D)$, and associates to an object $X \in \D$ its suspension spectrum $\Sig^{\infty}_+X \in \Sp(\D)$. Classical Quillen cohomology can then be replaced by an analogous construction where $\F(X)$ is replaced by $\Sig^\infty_+X$ and the coefficients are taken in spectrum objects of $\D$. In the case of spaces, this definition captures all generalized cohomology theories.

One may use the above construction to obtain even more refined invariants. Given an object $X \in \D$, if $X$ admits a map $f: X \lrar Y$ then it can naturally be considered as an object in $\D_{/Y}$. We may then obtain a more refined linear invariant by sending $X$ to its corresponding suspension spectrum in $\Sp(\D_{/Y})$. If we just have the object $X$ itself, there is a universal choice for $f$, namely the identity map $\Id: X \lrar X$. The corresponding suspension spectrum $L_X := \Sig^{\infty}_+ X \in \Sp(\D_{/X})$ is a linear invariant which determines all the others mentioned above. This invariant is called the \textbf{cotangent complex} of $X$. Given the cotangent complex $L_X$, we may obtain invariants which live in the $\infty$-category of spectra by taking a coefficient object $M \in \Sp(\D_{/X})$ and considering the mapping spectrum $\Map(L_X,M) \in \Sp$. The homotopy groups of these spectra form a natural generalization of classical Quillen cohomology groups. When $X$ is a space, this generalization corresponds to \textbf{twisted generalized cohomology} (see e.g.~\cite[\S 20]{MS06}), i.e.\ allowing coefficients in a local system of spectra on $X$. When $X$ is an $\EE_\infty$-ring spectrum this form of cohomology is also known as \textbf{topological Andr\'e-Quillen cohomology} (see~\cite{BM}), and plays a key role in deformation theory (see~\cite{Lur16}).

An abstract cotangent complex formalism was developed in the $\infty$-categorical context in~\cite[\S 7.3]{Lur14}. With a geometric analogy in mind, if we consider objects $Z \lrar X$ of $\D_{/X}$ as paths in $\D$, then we may consider spectrum objects in $\D_{/X}$ as ``infinitesimal paths'', or ``tangent vectors'' at $X$. As in~\cite{Lur14}, we will consequently refer to $\Sp(\D_{/X})$ as the \textbf{tangent $\infty$-category} at $X$, and denote it by $\T_X\D$. Just like the tangent space is a linear object, we may consider $\T_X\D$ as a linear categorical object, being a stable $\infty$-category.  
This analogy is helpful in many of the contexts in which linearization plays a significant role. Furthermore, it is often useful to assemble the various tangent categories into a global object, which is known as the \textbf{tangent bundle} of $\D$. The collection of functors $\Sig^{\infty}_+: \D_{/X} \lrar \T_X\D$ can then be assembled into a single functor $\Sig^{\infty}_{\int}: \D \lrar \T\D$, yielding convenient setting for studying the Quillen cohomology of several objects simultaneously.

This is the second in a series of papers dedicated to the cotangent complex formalism and its applications.  
In the previous papers~\cite{part0} and~\cite{part1} we studied tangent categories and tangent bundles in a model categorical setting. 
The main result of~\cite{part1} was a comparison theorem, which identified the tangent model categories of operadic algebras with tangent model categories of modules. When the algebras take values in a stable model category this reproduces the comparison appearing in~\cite[Theorem 7.3.4.13]{Lur14}.

The goal of the present paper is two-fold. Our first objective is to develop more of the cotangent complex formalism in the model categorical setting. This is done in the first part of the paper and includes, in particular, the relative counterparts of the cotangent complex and Quillen cohomology and their role in \textbf{obstruction theory}. We consider three classical cases - spaces, simplicial groups, and algebras over dg-operads - and analyze them with these tools. In particular, we obtain a description of the cotangent complex of a simplicial group, which, to the knowledge of the authors, does not appear in the literature. We finish the first part by formulating a general \textbf{Hurewicz principle}, which ties together various cases where necessary and sufficient conditions for a map to be an equivalence can be formulated using cohomology. The second part of the present paper is dedicated to the study of the cotangent complex and Quillen cohomology of \textbf{enriched categories}, and makes an essential use of the unstable comparison theorem of~\cite{part1}. Our main result can be formulated as follows:
\begin{thm}[{\ref{c:comp-lifts},~\ref{c:conceptual}}]\label{t:main-intro}
Let $\bS$ be a sufficiently nice symmetric monoidal model category and let $\T\bS \lrar \bS$ be the tangent bundle of $\bS$. Let $\Cat_{\bS}$ be the model category of $\bS$-enriched categories. Then for every fibrant $\bS$-enriched category $\C$ we have:
\begin{enumerate}
\item
The tangent model category $\T_{\C}\Cat_{\bS}$ is naturally Quillen equivalent to the model category $\Fun^{\bS}_{/\bS}(\C^{\op} \otimes \C,\T\bS)$ consisting of the $\bS$-enriched functors $\C^{\op} \otimes \C \lrar \T\bS$ which sit above the mapping space functor $\Map_{\C}:\C^{\op} \otimes \C \lrar \bS$.
\item
Under this equivalence, the cotangent complex of $\C$ corresponds to the desuspension of the composite functor $\Sig^{\infty}_{\int} \circ \Map_{\C}:\C^{\op} \otimes \C \lrar \bS \lrar \T\bS$.
\end{enumerate}
\end{thm}

When $\bS$ is stable the situation becomes simpler:
\begin{cor}[{\ref{c:comp-stable},~\ref{c:cotangent-stable}}]\label{c:comp-stable-intro}
Let $\bS$ be as in Theorem~\ref{t:main-intro} and assume in addition that $\bS$ is stable. Then for every fibrant $\bS$-enriched category $\C$ the tangent model category $\T_{\C}\Cat_{\bS}$ is naturally Quillen equivalent to the model category of enriched functors $\C^{\op} \otimes \C \lrar \bS$.
Under this equivalence, the cotangent complex of $\C$ corresponds to the desuspension of the mapping space functor $\Map_{\C}:\C^{\op} \otimes \C \lrar \bS $.
\end{cor}
Applying Corollary~\ref{c:comp-stable-intro} to the case where $\bS$ is the category of chain complexes over a field one obtains an identification of Quillen cohomology of a dg-category $\C$ with the corresponding \textbf{Hochschild cohomology}, up to a shift. To the knowledge of the authors this precise identification is new (see \S\ref{s:previous} for a more detailed discussion of the relation to existing results).

When $\bS$ is the category of simplicial sets, $\Cat_{\bS}$ is a model for the theory of $\infty$-categories. In this case, the computation of Theorem~\ref{t:main-intro} simplifies in a different way:
\begin{thm}[{see Corollary~\ref{c:twisted-arrow}}]\label{t:main-intro-2}
Let $\C$ be an $\infty$-category. Then the tangent $\infty$-category $\T_{\C}\Cat_{\infty}$ is equivalent to the $\infty$-category of functors
$$ \Tw(\C) \lrar \Sp $$
from the twisted arrow category of $\C$ to spectra. Under this equivalence, the cotangent complex $L_{\C}$ corresponds to the constant functor whose value is the desuspension of the sphere spectrum. More generally, if $f:\C \lrar \D$ is a map of $\infty$-categories then the functor $\Tw(\D) \lrar \Sp$ corresponding to $\Sig^{\infty}_+f \in \T_\D\Cat_{\infty}$ can be identified with $\Tw(f)_!\mathbb{S}[-1]$, i.e. with the left Kan extension of the shifted constant sphere spectrum diagram from $\Tw(\C)$ to $\Tw(\D)$.
\end{thm}

\begin{cor}[{see Corollary~\ref{c:twisted-arrow-quillen}}]\label{c:twisted-arrow-quillen-intro}
Let $\F: \Tw(\C) \lrar \Spectra$ be a functor and let $M_{\F} \in \T_\C\Cat_\infty$ be the corresponding spectrum object under the equivalence of Theorem~\ref{t:main-intro-2}. Then the Quillen cohomology group $\rH^n_Q(\C;M_{\F})$ is naturally isomorphic to the $(-n-1)$'th homotopy group of the spectrum $\lim\F$. 
\end{cor}

In particular, if $\C$ is a discrete category and $\F$ is a diagram of Eilenberg-MacLane spectra corresponding to a functor $\F': \Tw(\C) \lrar \Ab$, then the Quillen cohomology group $\rH^n_Q(\C;M_{\F})$ is naturally isomorphic to the $(n+1)$'th derived functor $\lim^{n+1}\F'$.

Recall that a map of simplicial sets is said to be \textbf{coinitial} if its opposite is cofinal. Theorem~\ref{t:main-intro-2} then yields the following sufficient condition for a map of $\infty$-categories to induce an isomorphism on Quillen cohomology:
if $f:\C \lrar \D$ is a map of $\infty$-categories such that the induced map $\Tw(\C) \lrar \Tw(\D)$ is coinitial, then $f$ induces an equivalence on Quillen cohomology with arbitrary coefficients.
Combined with the Hurewicz principle discussed in \S\ref{s:white}, This corollary 
implies that if $f: \C \lrar \D$ is a map of $\infty$-categories such that $\Tw(\C) \lrar \Tw(\D)$ is coinitial then $f$ is an equivalence if and only if it induces an equivalence on homotopy $(2,1)$-categories $\Ho_{\leq 2}(\C) \lrar \Ho_{\leq 2}(\D)$. Examples and applications to classical questions such as detecting equivalences and splitting homotopy idempotents are described in~\S\ref{s:simp-categories}. 

The results of this paper are used in subsequent work~\cite{part3} to study the Quillen cohomology of $(\infty,2)$-categories. This relies in an essential way on the fact that Theorem~\ref{t:main-intro} can handle rather general forms of enrichment. In particular, we obtain an analogous result to~\ref{t:main-intro-2} by replacing the twisted arrow category with a suitable $(\infty,2)$-categorical analogue, which we call the \textbf{twisted $2$-cell category}. This can be used to establish an obstruction theory as described in \S\ref{s:white} in the setting of $(\infty,2)$-categories.


\subsection{Relation to other work}\label{s:previous}
Using its classical definition via abelian group objects (see \S\ref{s:quillen}), Quillen cohomology of simplicial categories was studied by Dwyer-Kan-Smith (\cite{DKS86}) and Dwyer-Kan (\cite{DK88}). In those papers, the authors work in the category $\Cat^O_\Del$ of simplicial categories with a \textbf{fixed set of objects $O$}. In particular, they identify abelian group objects in $(\Cat^O_\Del)_{/\C}$ in explicit terms, which, in the language of the current paper can be reformulated as an equivalence between abelian group in $(\Cat^O_\Del)_{/\C}$ and functors from the twisted arrow category of $\rN(\C)$ to simplicial abelian groups. When $\C$ is a discrete category the cohomology theory associated to functors on the twisted arrow category is also known as \textbf{Baues-Wirsching cohomology}, see~\cite{BW85}. The relation between this cohomology and the (classical) Quillen cohomology of $\C$ as a category with a fixed set of objects (\cite{DKS86}) was identified by Baues and Blanc in~\cite{BB11} in the context of obstructions to realizations of $\Pi$-algebras. Similar results in the case of algebraic theories with a fixed set of objects were established in~\cite{JP05}.

However, in contrast to Corollary~\ref{c:twisted-arrow-quillen-intro}, in the above setting the classical Quillen cohomology of $\C$ with coefficients in such a diagram of simplicial abelian groups will \textbf{not} be given by the corresponding homotopy limit of $\F$. This is due to the fact that the (classical) cotangent complex of $\C \in \Cat^O_\Del$ is not a constant diagram. Instead, there is a long exact sequence (cf. Corollary~\ref{c:famous-cofiber}) relating the classical Quillen cohomology of $\C$ as an object of $\Cat^O_\Del$ and the homotopy groups of the homotopy limit of $\F$. As can be deduced from Corollary~\ref{c:twisted-arrow-quillen-intro} and Proposition~\ref{p:rigid}, the latter are in fact the classical Quillen cohomology groups of $\C$ as an object of $\Cat_\Del$. 

We note that the analogue of Theorem~\ref{t:main-intro-2} in the setting of abelian group objects is \textbf{false} when one does not fix the set of objects (see also Remark~\ref{r:DK}).
In particular, the clean identification given in Theorem~\ref{t:main-intro-2} actually \textbf{requires} working with the spectral version of Quillen cohomology, and to the knowledge of the authors, was not known before. We also note that varying the set of objects is essential in many contexts. For example, as described in \S\ref{s:white}, Quillen cohomology is specifically suited to support an obstruction theory against a certain class of maps, known as small extensions. In the case of $\infty$-categories, many small extensions of interest require changing the set of objects, see, e.g., Example~\ref{e:joost-2}.

A similar phenomenon holds in the case of dg-categories. In this case the associated Hochschild cohomology was studied by many authors (see~\cite{To07},~\cite{Kel}), and is related to (classical) Quillen cohomology once one fixes the set of objects. 
However, for the same reason as above the relation between Hochschild cohomology and the corresponding classical Quillen cohomology is not an identification, but rather a long exact sequence (see~\cite{Tab90}). A direct identification between Hochschild cohomology and Quillen cohomology as steaming from Corollary~\ref{c:comp-stable-intro} only holds when one works with \textbf{spectral} Quillen cohomology and when one does not fix the set of objects. To the knowledge of the authors, this result is new.

\subsection*{Acknowledgements}
The authors would like to thank Tomer Schlank for useful discussions and for suggesting the idea of using the twisted arrow category. While working on this paper the first author was supported by the FSMP. The second author was supported by NWO. The third author was supported by an ERC grant 669655. 

\section{The abstract cotangent complex formalism}

The goal of this section is to study the cotangent complex formalism and Quillen cohomology in the model categorical framework described in~\cite{part0}. We begin by recalling and elaborating this basic setup in~\S\ref{s:tangent}. In \S\ref{s:quillen} we consider the spectral version of Quillen cohomology and study its basic properties. The next three sections are devoted to classical examples, namely spaces, simplicial groups and algebras over dg-operads. In particular, we obtain a description of the cotangent complex of a simplicial group which appears to be new, though most likely well-known to experts. For algebras over dg-operads we unwind the connection between Quillen cohomology and operadic K\"ahler differentials, and briefly discuss the relation with deformation theory. Finally, in \S\ref{s:white} we propose a general Hurewicz principle based on the cotangent complex, which ties together various results on cohomological criteria for a map to be an equivalence.

\subsection{Tangent categories and the cotangent complex}\label{s:tangent}
In this section we recall from~\cite{part0} the formation of tangent bundles and the cotangent complex in the model-categorical setting. Recall that in order to study tangent categories of model categories, one needs a method for associating to a model category a universal stable model category, usually referred to as its \textbf{stabilization}. Given a pointed model category $\M$, there are various ways to realize its stablization as a certain model category of spectrum objects in $\M$ (see~\cite{Hov}). However, most of these constructions require $\M$ to come equipped with a point-set model for the suspension-loop adjunction (in the form of a Quillen adjunction), which is lacking in some cases of interest, e.g., the case of enriched categories which we will study in \S\ref{s:enriched}. As an alternative, the following model category of spectrum objects was developed in \cite{part0}, based on ideas of Heller (\cite{Hel}) and Lurie (\cite{Lur06}): for a nice pointed model category $\M$ we consider the left Bousfield localization $\Sp(\M)$ of the category of $(\NN\times \NN)$-diagrams in $\M$ whose fibrant objects are those diagrams $X: \NN\times\NN\lrar \M$ for which $X_{m, n}$ is weakly contractible when $m\neq n$ and for which each diagonal square
\begin{equation}\label{e:square_nn}
\xymatrix@R=1.5pc@C=1.5pc{
X_{n, n}\ar[r]\ar[d] & X_{n, n+1}\ar[d]\\
X_{n+1, n}\ar[r] & X_{n+1, n+1}
}
\end{equation}
is homotopy Cartesian. In this case the diagonal squares determine equivalences $X_{n,n} \x{\simeq}{\lrar} \Om X_{n+1,n+1}$, and so we may view fibrant objects of $\Sp(\M)$ as $\Om$-spectrum objects. The existence of this left Bousfield localization requires some assumptions on $\M$, for example, being combinatorial and left proper. In this case there is a canonical Quillen adjunction
$$ \Sig^{\infty}: \M \adj \Sp(\M): \Om^{\infty} $$
where $\Om^{\infty}$ sends an $(\NN \times \NN)$-diagram $X_{\bullet\bullet}$ to $X_{0,0}$ and $\Sig^{\infty}$ sends an object $X$ to the constant $(\NN \times \NN)$-diagram with value $X$. While $\Sig^{\infty}X$ may not resemble the classical notion of a suspension spectrum, it can be replaced by one in an essentially unique way, up to a stable equivalence (see~\cite[Remark 2.3.4]{part0}). The flexibility of not having to choose a specific suspension spectrum model has its advantages and will be exploited, for example, in~\S\ref{s:cotangent}. 

When $\M$ is not pointed one stabilizes $\M$ by first forming its \textbf{pointification} $\M_{\ast} := \M_{\ast/}$, endowed with its induced model structure, 
and then forming the above mentioned model category of spectrum objects in $\M_{\ast}$. We then denote by $\Sig^{\infty}_+: \M \adj \Sp(\M_{\ast}): \Om^{\infty}_+$ the composition of Quillen adjunctions
$$ \xymatrix{
\Sig^{\infty}_+:\M \ar@<1ex>[r]^-{(-) \coprod \ast} & \M_{\ast} \ar@<1ex>[l]^-{\U} \ar@<1ex>[r]^-{\Sig^{\infty}} & \Sp(\M_{\ast}) \ar@<1ex>[l]^-{\Om^{\infty}}: \Om^{\infty}_+ \\
}.$$ 
When $\M$ is a left proper combinatorial model category and $A\in \M$ is an object, the pointification of $\M_{/A}$  
is given by the (combinatorial, left proper) model category $\M_{A//A}:=\left(\M_{/A}\right)_{\id_A/}$ of objects in $\M$ over-under $A$, endowed with its induced model structure. The stabilization of $\M_{/A}$ is then formed by taking the model category of spectrum objects in $\M_{A//A}$ as above. 
\begin{define}
Let $\M$ be a left proper combinatorial model category. We will denote the resulting stabilization of $\M_{/A}$ by
$$ \T_A\M \x{\df}{=} \Sp(\M_{A//A}) $$
and refer to its as the \textbf{tangent model category} to $\M$ at $A$.
\end{define}
\begin{define}[{cf.~\cite[\S 7.3]{Lur14}}]
Let $\M$ be a left proper combinatorial model category. We will denote by
$$ L_A = \LL\Sig^{\infty}_+(A) \in \T_A\M $$
the derived suspension spectrum of $A$ and will refer to $L_A$ as the \textbf{cotangent complex} of $A$. Given a map $f:A \lrar B$ we will denote by 
$$ L_{B/A} = \hocofib\left[\LL\Sig^{\infty}_+(f) \lrar L_B\right] $$ 
the homotopy cofiber of the induced map $\LL\Sig^{\infty}_+(f) \lrar L_B$ in $\T_B\M$. The object $L_{B/A}$ is known as the \textbf{relative cotangent complex} of the map $f$. We note that when $f$ is an equivalence the relative cotangent complex is a weak zero object, while if $A$ is weakly initial we have $L_{B/A} \simeq L_B$.
\end{define}
Recall that any model category $\M$ (and in fact any relative category) has a canonically associated $\infty$-category $\M_\infty$, obtained by formally inverting the weak equivalences of $\M$. By~\cite[Proposition 3.3.2]{part0} the $\infty$-category associated to the model category $\T_A\M$ is equivalent to the \textbf{tangent $\infty$-category} $\T_A(\M_\infty)$, in the sense of~\cite[\S 7.3]{Lur14}, at least if $A$ is fibrant or if $\M$ is right proper (so that $\M_{/A}$ models the slice $\infty$-category $(\M_\infty)_{/A}$). In this case, the cotangent complex defined above agrees with the one studied in~\cite[\S 7.3]{Lur14}.

If $\C$ is a presentable $\infty$-category, then the functor $\C \lrar \Cat_{\infty}$ sending $A \in \C$ to $\T_A\C$ classifies a (co)Cartesian fibration $\T\C\lrar \C$ known as the \textbf{tangent bundle} of $\C$. 
A simple variation of the above model-categorical constructions can be used to give a model for the tangent bundle of a model category $\M$ as well, which furthermore enjoys the type of favorable formal properties one might expect (see \cite{part0}). 
More precisely, if $(\NN \times \NN)_\ast$ denotes the category obtained from $\NN \times \NN$ by \textbf{freely adding a zero object} and $\M$ is a left proper combinatorial model category, then one can define $\T\M$ as a left Bousfield localization of the Reedy model category $\M^{(\NN \times \NN)_{\ast}}_{\Reedy}$, where a Reedy fibrant object $X: (\NN \times \NN)_\ast\lrar \M$ is fibrant in $\T\M$ if and only if the map $X_{n,m} \lrar X_{\ast}$ is a weak equivalence for every $n \neq m$ and the square~\eqref{e:square_nn} is homotopy Cartesian for every $n \geq 0$. 

The projection $\ev_*: \T\M \lrar \M$ is then a (co)Cartesian fibration which exhibits $\T\M$ as a \textbf{relative model category} over $\M$ in the sense of~\cite{HP}: $\T\M$ has relative limits and colimits over $\M$ and factorization (resp.\ lifting) problems in $\T\M$ with a solution in $\M$ admit a compatible solution in $\T\M$. In particular, it follows that the projection is a left and right Quillen functor and that each fiber is a model category. When $A\in \M$ is a fibrant object, the fiber $(\T\M)_A$ can be identified with the tangent model category $\T_A\M$. Furthermore, the underlying map of $\infty$-categories $\T\M_\infty \lrar \M_{\infty}$ exhibits $\T\M_{\infty}$ as the tangent bundle of $\M_\infty$ (see~\cite[Proposition 3.25]{part0}). 
We have a natural Quillen adjunction over $\M$ of the form
$$\xymatrix{\M^{[1]}\ar@<1ex>[rr]^{\Sigma^{\infty}_{\M}}\ar_{\cod}[dr] && \T\M\ar@<1ex>[ll]_{\upvdash}^{\Om^{\infty}_{\M}}\ar^{\pi}[dl]\\ & \M &}$$ 
where $\Om^{\infty}_\M$ is the restriction along the arrow $(0,0) \lrar \ast$ of $(\NN \times \NN)_\ast$. Composing with the Quillen adjunction $\xymatrix{\M \ar@<1ex>[r]^{\diag} & \M^{[1]} \ar@<1ex>[l]^{\dom}_{\upvdash}}$ given by the diagonal and domain functors yields a Quillen adjunction 
$$\xymatrix{\Sig^{\infty}_{\int}:\M\ar@<1ex>[r] & \T\M\ar@<1ex>[l]_-{\upvdash}}:\Om^{\infty}_{\int} .$$ 
When $A$ is fibrant we may identify $\Sig^{\infty}_{\int}(A)$ with the cotangent complex of $A$ lying in the fiber $\T_{A}\M \subseteq \T\M$. We refer to the derived functor of $\Sig^{\infty}_{\int}$ as the \textbf{global cotangent complex} functor of $\M$.
%
%

It is often the case that $\M$ is tensored and cotensored over a symmetric monoidal (SM) model category $\bS$. In favorable cases (e.g., when $\bS$ is tractable, see~\cite[\S 3.2]{part0}), the tangent bundle $\T\M$ inherits this structure. This is particularly useful when discussing tangent bundles of functor categories. Indeed, if $\I$ is a small $\bS$-enriched category then we may consider the category of $\bS$-enriched functors $\Fun^{\bS}(\I,\M)$ with the projective model structure. The tangent bundle of $\Fun^{\bS}(\I,\M)$ can then be identified as follows:

\begin{pro}[{\cite[Proposition 2.2.1]{part1}}]\label{c:tangent-functor}
Let $\M$ be a left proper combinatorial model category. There is a natural equivalence of categories
\begin{equation}\label{e:tangent}
\vcenter{\xymatrix{
\T\Fun^{\bS}(\I,\M) \ar^{\simeq}[rr]\ar[rd] && \Fun^{\bS}(\I,\T\M)\ar[dl] \\
& \Fun^{\bS}(\I,\M) & \\
}}
\end{equation}
identifying the tangent bundle model structure on the left with the projective model structure on the right. Here the left diagonal functor is the tangent projection of $\Fun^{\bS}(\I,\M)$ and the right diagonal functor is given by post-composing with the tangent projection of $\M$. 
\end{pro}

\begin{rem}\label{r:global-tensored}
Since the equivalence~\ref{e:tangent} is an equivalence over $\Fun^{\bS}(\I,\M)$, associated to every $\F: \I \lrar \M$ is an equivalence of categories
\begin{equation}\label{e:equiv-functor}
\T_\F\Fun^{\bS}(\I,\M)) \x{\cong}{\lrar} \Fun^{\bS}_{/\M}(\I,\T\M)
\end{equation}
where $\Fun^{\bS}_{/\M}(\I,\T\M)$ denotes the category of $\bS$-enriched lifts
$$ \xymatrix{
& \T\M \ar^{\pi}[d] \\
\I \ar@{-->}[ur] \ar_-{\F}[r] & \M. \\
}$$
By transport of structure one obtains a natural model structure on $\Fun^{\bS}_{/\M}(\I,\T\M)$, which coincides with the corresponding projective model structure (i.e., where weak equivalences and fibrations are defined objectwise).
\end{rem}

\begin{rem}\label{r:functor-cotangent}
Under the identification of Corollary~\ref{c:tangent-functor} the global cotangent functor
$$ \Sig^{\infty}_{\int}: \Fun^{\bS}(\I,\M) \lrar \T\Fun^{\bS}(\I,\M) \cong \Fun^{\bS}(\I,\T\M) $$
is simply given by post-composing with the global cotangent complex of $\M$. In particular, the cotangent complex $L_\F$ of a functor $\F: \I \lrar \M$, when considered as an object of $\Fun^{\bS}_{/\M}(\I,\T\M)$ is simply given by the composition $\I \x{\F}{\lrar} \M \x{\Sig^{\infty}_{\int}}{\lrar} \T\M$.
\end{rem}

When $\M$ is stable the situation becomes even simpler:
\begin{cor}[{\cite[Corollary 2.2.4, Corollary 2.2.6]{part1}}]\label{c:stable-functor-tangent}
Let $\M$ be a left proper combinatorial stable model category tensored over a tractable SM model category $\bS$ and let $\F: \I \lrar \M$ be an $\bS$-enriched functor. Assume either that $\M$ is right proper or that $\F$ is fibrant. Then the tangent model category $\T_\F\Fun^{\bS}(\I,\M)$ is Quillen equivalent to $\Fun^{\bS}(\I,\M)$. Under this equivalence the cotangent complex $L_\F \in \T_\F\Fun^{\bS}(\I,\M)$ of $\F$ maps to $\F$ itself.
\end{cor}

\subsection{Spectral Quillen cohomology}\label{s:quillen}
The classical work of Quillen \cite{Qui67} gives a way to define cohomology groups for objects in general model categories, using a derived construction of abelianization. These groups are now known as \textbf{Quillen cohomology groups}. In this section we consider Quillen cohomology and its basic properties in the setting of spectrum objects, rather than abelian group objects, and compare the two approaches. We will follow the ideas of \cite[\S 7.3] {Lur14}, using model categories instead of $\infty$-categories. In particular, our description will use the model for spectrum objects discussed in the previous section.

Let us begin by recalling the original construction of cohomology groups due to Quillen. Given a model category $\M$ and an object $X$, consider the category $\Ab(\M_{/X})$ of \textbf{abelian group objects} in $\M_{/X}$. In favorable cases, the slice model structure on $\M_{/X}$ can be transferred to $\Ab(\M_{/X})$ along the free-forgetful adjunction
$$ \F: \M_{/X} \adj \Ab(\M_{/X}): \U. $$
Given an abelian group object $M \in \Ab(\M_{/X})$, and a (possibly negative) integer $n$, Quillen defines the $n$'th cohomology group of $X$ with coefficients in $M$ by the formula
$$ \rH^n_{\cls,Q}(X;M) \x{\df}{=} \pi_0\Map^{\der}_{\Ab(\M_{/X})}(\Sig^k\LL\F(X),\Sig^mM) $$
where $k,m$ are non-negative integers such that $m-k = n$. To make sure that these cohomology groups are well-defined and well-behaved, Quillen imposes certain homotopical conditions on the model category $\Ab(\M_{/X})$, including, in particular, the assumption that for every abelian group object $M \in \Ab(\M_{/X})$ the canonical map $M\lrar \Om\Sig M$ is a weak equivalence (this property of $\Ab(\M_{/X})$ is referred to as being \textbf{linear} in~\cite{Sch97}). Equivalently, this means that the suspension functor associated to $\Ab(\M_{/X})$ is derived fully-faithful.
 
If we consider stabilization as a refined form of abelianization, we may attempt to define Quillen cohomology using the stabilization $\Sp(\M_{X//X})$ instead of the abelianization $\Ab(\M_{/X}) = \Ab(\M_{X//X})$. Several arguments can be made in favor of this choice:
\begin{enumerate}[(1)]
\item
Except in special cases, it is not easy to check whether the induced model structure on abelian group objects exists. Furthermore, even in cases where the transferred model structure does exist, the association $\M \mapsto \Ab(\M)$ is not invariant under Quillen equivalences, and may generate unpredictable results. For example, if $\Cat$ is the category of small categories equipped with the Thomason model structure (\cite{Tho80}) then $\Cat$ is Quillen equivalent to the model category of simplicial sets with the Kan-Quillen model structure. However, the underlying category of any abelian group object in $\Cat$ is automatically a groupoid, and hence its underlying space is $2$-truncated, while abelian group objects in simplicial sets can have non-trivial homotopy groups in any dimension.
\item
As we will see below, under the conditions assumed by Quillen, classical Quillen cohomology can always be recovered as a special case of the definition below, by restricting to coefficients of a particular type. The resulting type of coefficients is natural in some cases, but is very unnatural in others. 
\item
In the case of the category $\SS$ of simplicial sets, classical Quillen cohomology corresponds to taking ordinary cohomology with local coefficients, while spectral Quillen cohomology allows for an arbitrary local system of spectra as coefficients. In particular, spectral Quillen cohomology of spaces subsumes all \textbf{generalized} cohomology theories as well as their twisted versions, and can be considered as a universal cohomology theory for spaces.
\end{enumerate}

Motivated by the above considerations, we now come to the main definition:
\begin{defn}\label{d:Quillen-coh}
Let $\M$ be a left proper combinatorial model category and let $X$ be a fibrant object. For $n \in \ZZ$ we define the \textbf{$n$'th (spectral) Quillen cohomology} group of $X$ with coefficients in $M \in \T_X\M = \Sp(\M_{X//X})$ by the formula
$$ \rH_Q^n(X;M):=\pi_0 \Map^{\der}(L_X,\Sig^n M)$$
where $\Map^{\der}(L_X, \Sig^nM)$ is the (derived) mapping space. Similarly, if $f:A \lrar X$ is a map in $\M$, we define the \textbf{relative $n$'th Quillen cohomology} group of $X$ with coefficients in $M\in \T_X\M$ by the formula
$$ \rH_Q^n(X,A;M):=\pi_0 \Map^{\der}(L_{X/A},\Sig^n M) .$$  
where $L_{X/A}$ is the relative cotangent complex of the map $f$.
\end{defn} 

The formation of Quillen cohomology is contravariantly functorial. More explicitly, given a map $f: X \lrar Y$ in $\M$, we have a commutative square of right Quillen functors
$$ \xymatrix{
\Sp(\M_{Y//Y}) \ar_{\Om^{\infty}_+}[d]\ar^{f^*}[r] & \Sp(\M_{X//X}) \ar^{\Om^{\infty}_+}[d] \\
\M_{/Y} \ar^{f^*}[r] & \M_{/X} \\
}$$
For a fibrant $\Om$-spectrum $M \in \Sp(\M_{Y//Y})$ we then get an induced map
$$ \Map^{\der}_{\M_{/Y}}(Y,\Om^{\infty}_+(M[n]))\lrar \Map^{\der}_{\M_{/X}}(X,f^*\Om^{\infty}_+(M[n])) =  \Map^{\der}_{\M_{/X}}(X,\Om^{\infty}_+(f^*M[n])) $$
and hence a map
$$ f^*: \rH^n_Q(Y;M) \lrar \rH^n_Q(X;f^*M) $$
on Quillen cohomology groups.

Let us now explain the relation between Definition~\ref{d:Quillen-coh} and Quillen's classical definition described above. Let $\M$ and $X$ be such that the transferred model structure on $\Ab(\M_{/X})$ exists and assume in addition that the stable model structures on $\Sp(\M_{X//X})$ and $\Sp(\Ab(\M_{/X}))$ exist, so that we have a commutative diagram of Quillen adjunctions
\begin{equation} \label{e:EM}\vcenter{\xymatrix{
\M_{/X} \ar@<1ex>^-{\F}[r]\ar@<1ex>^-{\Sig^{\infty}_+}[d] & \Ab(\M_{/X}) \ar@<1ex>^-{\Sig^{\infty}_+}[d]\ar@<1ex>^-{\U}_-{\upvdash}[l] \\
\Sp(\M_{X//X}) \ar@<1ex>^-{\F_{\Sp}}[r] \ar@<1ex>^-{\Om^{\infty}_{+}}_-{\vdash}[u] & \Sp(\Ab(\M_{/X})). \ar@<1ex>^-{\U_{\Sp}}_-{\upvdash}[l] \ar@<1ex>^-{\Om^{\infty}_{+}}_-{\vdash}[u]
}}\end{equation}
\begin{defn}\label{d:EM}
Let $M \in \Ab(\M_{/X})$ be an abelian group object. We will denote by $\rH M := \U_{\Sp}\LL\Sig^{\infty}_+M$ the image of the suspension spectrum of $\M$ in $\Sp(\M_{X//X})$ under the forgetful functor $\U_{\Sp}: \Sp(\Ab(\M_{/X})) \lrar \Sp(\M_{X//X})$. We will refer to $\rH M$ as the \textbf{Eilenberg-MacLane spectrum} associated to $M$.
\end{defn}

\begin{pro}\label{p:rigid}
Under the assumptions above, if for each $M$ in $\Ab(\M_{/X})$, the unit map $M\lrar \Om\Sig M$ is an equivalence (see \cite[5.2]{Qui67}), then for every object $M \in \Ab(\M_{/X})$ there is a canonical isomorphism of groups
$$
\rH_{\cls, Q}^n(X, M) \cong \rH_Q^n(X, \rH M)
$$

\end{pro}
\begin{proof}
Since the suspension functor on $\Ab(\M_{/X})$ is derived fully faithful, it follows that $\Sig^{\infty}_+: \Ab(\M_{/X}) \lrar \Sp(\Ab(\M_{/X}))$ is derived fully-faithful. Given an object $M \in \Ab(\M_{/X})$ we obtain a weak equivalence
\begin{align*}
 \Map^{\der}_{\Ab(\M_{/X})}(\LL\F(X),M) &\simeq \Map^{\der}_{\Sp(\Ab(\M_{/X}))}(\Sig^{\infty}_+\LL\F(X),\Sig^{\infty}_+M) \\
 & \simeq \Map^{\der}_{\Sp(\Ab(\M_{/X}))}(\LL \F_{\Sp}(L_X),\Sig^{\infty}_+M) \simeq \Map^{\der}_{\Sp(\M_{X//X})}(L_X,\rH M) 
\end{align*}
as desired.
\end{proof}

\begin{rem}
In light of the above result, we will abuse terminology and henceforth refer to spectral Quillen cohomology simply as Quillen cohomology.
\end{rem}



Quillen cohomology is a natural host for obstructions to the existence of lifts against a certain class of maps, known as \textbf{small extensions}. We recall here the main definitions in the setting of spectral Quillen cohomology. A more elaborate discussion of the obstruction theoretic aspect will be worked out in \S\ref{s:white}.

Let $\M$ be a left proper combinatorial model category and let $X$ be a fibrant-cofibrant object. Let $M \in \Sp(\M_{X//X})$ be an $\Om$-spectrum object over $X$. By adjunction, we may represent classes in $\rH^n_Q(X,M)$ by maps $\alp: X \lrar \Om^{\infty}_+(M[n])$ over $X$, and two such maps represent the same element in $\rH^n_Q(X,M)$ if and only if there exists a homotopy between them in $\M_{/X}$. The trivial class in $\rH^n_Q(X,M)$ is represented by the image $s_0: X=\Om^\infty_+(0)\lrar \Om^\infty_+(M[n])$ of the zero map $0 \lrar M[n]$ under $\Om^\infty_+$. We shall therefore refer to $s_0$ as the \textbf{$0$-section} of $\Om^{\infty}_+(M[n]) \lrar X$. 
In practice it is often useful to work with a homotopical variant of the notion of a $0$-section. 

\begin{define}
If $f: 0'\lrar M[n]$ is a map in $\Sp(\M_{X//X})$ whose domain is a weak zero object, we will call the induced map $s_0'=\Om^\infty_+(f): \Om^\infty_+(0')\lrar \Om^\infty_+(M[n])$ a \textbf{weak $0$-section} of $\Om^{\infty}_+(M[n]) \lrar X$.
\end{define}
In particular, taking $0'\lrar M[n]$ to be a fibration, we can always work with weak $0$-sections that are fibrations.

\begin{define}\label{d:small}
Let $\M$ be a left proper combinatorial model category, $X \in \M$ a fibrant object and $M \in \T_X\M = \Sp(\M_{X//X})$ a fibrant $\Om$-spectrum object. For any $f: Y\lrar X$ in $\M_{/X}$ and any map $\alp: Y \lrar \Om^{\infty}_+(M[1])$ in $\M_{/X}$, we will say that a square in $\M_{/X}$ of the form
\begin{equation}\label{e:small}\vcenter{
\xymatrix{
Y_\alp \ar[r]\ar_{p_\alp}[d] & \Om^\infty_+(0') \ar^{s_0'}[d] \\
Y \ar^-{\alp}[r] & \Om^{\infty}_+(M[1]) \\
}}
\end{equation}
exhibits $Y_\alpha$ as a \textbf{small extension} of $Y$ by $M$ if it is homotopy Cartesian and $s_0'$ is a weak $0$-section. In this case we will also say that $p_\alp$ is the small extension associated to $\alp$.
\end{define}
The map $\alpha: Y\lrar \Om^\infty_+M[1]$ over $X$ gives rise to an element $[\alpha]$ in the group $\pi_0\Map^{\der}_{/X}(Y, \Om^\infty_+M[1])$, which can be identified with
\begin{align*}
\pi_0\Map^{\der}_{/X}(Y, \Om^\infty_+M[1]) &\cong \pi_0\Map^{\der}_{\T_X\M}(f_!L_Y, M[1]) \\
& \cong \pi_0\Map^{\der}_{\T_Y\M}(L_Y, f^*M[1]) = \rH^1_Q(Y; f^*M).
\end{align*}
We will say that the small extension $p_\alp$ is \textbf{classified} by the resulting element $[\alp]$ in the first Quillen cohomology group $\rH^1_Q(Y;f^*M)$ of $Y$ with values in the base change of $M$ along $f: Y\lrar X$. The case where the map $f: Y\lrar X$ is a weak equivalence with cofibrant domain (e.g. $f=\id_X$ when $X$ is fibrant-cofibrant) is of particular importance. In that case, all classes in $\rH^1_Q(X;M)$ can be realized by maps $\alpha: Y\lrar \Om^\infty_+M[1]$ over $X$, and therefore classify small extensions of $Y$ by $M$.

We consider the small extension $p_\alp$ as a geometric incarnation of the Quillen cohomology class $[\alp]$. However, one should note that in general the class $[\alp]$ cannot be reconstructed from the map $p_\alp$ alone, with the following notable exception:

\begin{example}\label{e:trivial-1}
If $\alp: Y \lrar \Om^{\infty}_+(M[1])$ factors (over $X$) through a weak $0$-section then $Y_\alp$ is weakly equivalent to the homotopy fiber product $Y \times^h_{\Om^{\infty}_+(f^*M[1])} Y \simeq \Om\Om^{\infty}_+(f^*M[1]) \simeq \Om^{\infty}_+(f^*M)$ and $p_\alp$ is equivalent to the canonical map $\Om^{\infty}_+(f^*M) \lrar Y$. In this case we will say that $p_\alp$ is a \textbf{split small extension}. In particular, a split small extension admits a section up to homotopy. This is actually a necessary and sufficient condition: if $p_\alp$ admits a section up to homotopy then $\alp$ factors up to homotopy through $s_0'$, and hence factors honestly through some weak $0$-section.
\end{example}

\begin{rem}\label{r:rigid}
The above definition of small extension can also be formulated in the classical setup of Quillen cohomology. Suppose $[\alpha]\in \rH^1_{\cls, Q}(X, M)$ is a class in the first classical Quillen cohomology group of $X$ with coefficients in $M\in \Ab(\M_{/X})$. The corresponding class in spectral Quillen cohomology classifies an extension of (a suitable cofibrant model of) $X$ by $\rH M$, where $\rH M$ is the Eilenberg-MacLane spectrum of $M$ (see Definition~\ref{d:EM}). Since the map $M\lrar \Om\Sig M$ is an equivalence, the object $\Om^\infty \rH M$ is equivalent to the underlying object of $M$ in $\M$. The corresponding small extension then coincides with the small extension that is associated to the class $[\alpha]$ in classical Quillen cohomology. 
\end{rem}

\begin{rem}\label{r:base-change}
In the setting of Definition~\ref{d:small}, assume that $Y$, $Y_\alp$ and $\Om^\infty_+(0')$ are fibrant in $\M_{/X}$ and that the square~\ref{e:small} exhibits $Y_\alp$ as a small extension of $Y$ by the fibrant $\Om$-spectrum $M \in \T_X\M = \Sp(\M_{X//X})$. If $g:Z \lrar X$ is a map then the square
\begin{equation}\label{e:small-2}
\vcenter{\xymatrix{
Z'_\alp := Z \times_X Y_\alp \ar[r]\ar[d] & \hspace{0pt} Z \times_X \Om^\infty_+(0') = \Om^\infty_+(g^*0')\ar@<-1ex>[d]   \\
Z':=Z\times_X Y \ar^-{g^*\alp}[r] & \Om^{\infty}_+(g^*M[1]) \\
}}
\end{equation}
exhibits $Z'_\alp$ as the small extension of $Z'$ by $g^*M$ which is classified by $g^*[\alp] \in \rH^1_Q(Z;g^*f^*M)$.
\end{rem}

We finish this subsection with the following result, identifying relative Quillen cohomology as the absolute Quillen cohomology in coslice categories. Given a map $f:A \lrar X$, we have an equivalence of categories $(\M_{A/})_{f//f}\cong \M_{X//X}$ identifying the slice-coslice model structures on both sides. Both the cotangent complex of $X$, viewed as an object in $\M$, and of $X$, viewed as an object in $\M_{A/}$, can therefore be viewed as objects of the same model category $\T_X\M$.
\begin{pro}\label{p:relative}
Let $\M$ be a left proper combinatorial model category and let $f: A \lrar X$ be a map in $\M$. Then the relative cotangent complex $L_{X/A} \in \T_X\M$ is naturally weakly equivalent to the (absolute) cotangent complex of $X$, considered as an object of the coslice model category $\M_{A/}$.
\end{pro}
\begin{proof}
Since $\M$ is left proper we may assume without loss of generality that $f$ is a cofibration. We consider the pushout square 
\begin{equation}\label{e:AX}
\vcenter{\xymatrix{
A \coprod X \ar[r]\ar[d] & X \coprod X \ar[d] \\
X \ar[r] & X \coprod_{A} X \\
}}
\end{equation}
as a pushout square in $\M_{X//X}$, where the maps to $X$ are the obvious ones and the map $X \lrar A \coprod X$ is the inclusion of the second factor. Since we assumed $f$ to be a cofibration it follows that the top horizontal map is a cofibration and since $\M$ is left proper we may conclude that~\eqref{e:AX} is a homotopy pushout square as well. Applying the functor $\Sig^{\infty}: \M_{X//X} \lrar \T_X\M$ and using the fact that $X$ is a zero object in $\M_{X//X}$ we obtain a homotopy cofiber sequence in $\T_X\M$ of the form
$$ \Sig^{\infty}_+(f) \lrar \Sig^{\infty}_+(X) \lrar \Sig^{\infty}(X\textstyle\coprod_A X). $$
Since the relative cotangent complex $L_{X/A}$ is defined to be the homotopy cofiber of the map $\Sig^{\infty}_+(f) \lrar \Sig^{\infty}_+(X)$ we now obtain a natural equivalence
$ L_{X/A} \simeq \Sig^\infty(X\coprod_A X)$. Since $X \coprod_A X$ can be identified with the coproduct, in $\M_{A/}$, of $X$ with itself we may consider $\Sig^\infty(X\coprod_A X)$ as a model for the cotangent complex of $X \in \M_{A/}$, and so the desired result follows.
\end{proof}

\begin{cor}\label{c:relative}
The relative Quillen cohomology of $X$ over $A$ with coefficients in $M \in \T_X\M$ is isomorphic to the Quillen cohomology of $X\in \M_{A/}$ with coefficients in $M$.
\end{cor}

\subsection{Spaces and parametrized spectra}\label{s:spaces}

In this section we will discuss the notions of the cotangent complex and Quillen cohomology when $\M = \SS$ is the category of simplicial sets, endowed with the Kan-Quillen model structure. The underlying $\infty$-category $\SS_{\infty}$ of $\SS$ is a model for the $\infty$-category of spaces, and we will consequently refer to objects in $\SS$ simply as spaces. Given a space $X$, the slice model category $\SS_{/X}$ is Quillen equivalent to the category of simplicial functors $\fC[X] \lrar \SS$ out of its associated simplicial category, endowed with the projective model structure (see e.g.,~\cite[\S 2]{Lur09}). 
By Corollary~\ref{c:stable-functor-tangent} it follows that $\Sp(\SS_{/X})$ is Quillen equivalent to $\Fun(\fC[X],\Sp(\SS_\ast))$, whose underlying $\infty$-category is equivalent to the $\infty$-category $\Fun(X,\Sp(\SS_\ast)_{\infty})$ of functors from $X$ to spectra by~\cite{Lur09}. We may thus consider $\T_X\SS$ as a model for the theory of \textbf{parametrized spectra over $X$} (see~\cite{MS06},~\cite{ABG11}). More concretely, an $\Om$-spectrum $X \lrar Z_{\bullet\bullet} \lrar X$ in $\Sp(\SS_{X//X})$ encodes the data of a family of $\Om$-spectra parametrized by $X$, where the $\Om$-spectrum associated to the point $x \in X$ is the $\Om$-spectrum $Z_{\bullet\bullet} \times_X \{x\}$. By~\cite[Remark 2.4.7]{part0} the cotangent complex $\Sigma^\infty_+X = \Sigma^{\infty}\left(X \coprod X\right)$ is stably equivalent to the $\Om$-spectrum $Z_{\bullet\bullet}$ given by 
$$ Z_{n,n} \simeq \colim_j\Om^j\Sig^{j+n}\left(X \coprod X\right) \simeq \colim_j \Om^j(X\times S^{j+n}) \simeq X \times \colim_j\Om^jS^{j+n}\simeq X\times \bbS_{n, n} $$
and thus corresponds to the \textbf{constant sphere spectrum} over $X$.

Given a parametrized spectrum $M \in \T_X\SS$, the $n$'th Quillen cohomology 
$$ \rH^n_Q(X;M) = \pi_0\Map^{\der}_{\T_X\SS}(L_X,M) \simeq \pi_0\Map^{\der}_{\SS_{/X}}(X,\Om^{\infty}_+(\Sig^n M)) $$ 
is the set of homotopy classes of sections of the fibration $\Om^{\infty}_+(\Sig^nM) \lrar X$. In particular, for negative $n$'s we may identify the corresponding Quillen cohomology groups with the homotopy groups of the space of sections $\Map^{\der}_{\SS_{/X}}(X,\Om^{\infty}_+M)$. 

For every fibrant $\Om$-spectrum $M$ the object $\Om^{\infty}_+M \in \SS_{/X}$ carries the structure of an \textbf{$\EE_{\infty}$-group object} in $\SS_{/X}$. The notion of a \textbf{small extension} $X_{\alp} \lrar X$ of $X$ by $M$  (see Definition~\ref{d:small}) then corresponds to the notion of a \textbf{torsor} under $\Om^{\infty}_+M$. 
When $M$ is a constant parametrized spectrum with value $M_0 \in \Sp(\SS_\ast)$ the associated $\EE_{\infty}$-group object splits as a product $\Om^{\infty}_+M =X\times\Om^{\infty}_+M_0$. In this case we may identify small extensions by $M$ with \textbf{principal fibrations} $X_\alp \lrar X$ with structure group $\Om^{\infty}_+M_0$ and identify the corresponding Quillen cohomology class with the classifying map $X \lrar B\Om^{\infty}_+M_0 \simeq \Om^{\infty}_+(M_0[1])$.

A case of special interest of Quillen cohomology and small extensions is the case where $M$ comes from a local system of abelian groups. More precisely, let $A: \Pi_1(X)\lrar \Ab$ be a local system of abelian groups on $X$ and let $n \geq 0$ be a non-negative integer. The association $x \mapsto K(A(x),n)$ determines a functor from the fundamental groupoid $\Pi_1(X)$ of $X$ to the category $\Ab(\SS)$ of simplicial abelian groups. Alternatively, we may consider $K(A(-),n)$ as a (fibrant) abelian group object in the functor category $\SS^{\Pi_1(X)}$ (endowed with the projective model structure). Applying the relative nerve construction of~\cite[Definition 3.2.5]{Lur09} (which is a right Quillen functor) we obtain an abelian group object $\ovl{K}_{\Pi_1(X)}(A,n)$ in the category $\SS_{/\rN(\Pi_1(X))}$ of simplicial sets over the nerve of $\Pi_1(X)$. Pulling back along the map $X \lrar \rN\Pi_1(X)$ we obtain an abelian group object
$$ \ovl{K}_X(A,n) \in \Ab(\SS_{/X}) .$$
We may now associate to $\ovl{K}_X(A,n)$ its Eilenberg-Maclane spectrum $H\ovl{K}_{X}(A,n) \in \T_X\SS$ (see Definition~\ref{d:EM}), which, as a family of spectra, is the family which associates to $x \in X$ the Eilenberg-Maclane spectrum corresponding to $K(A(x),n)$. The (spectral) Quillen cohomology groups of this parametrized spectrum coincide with the classical Quillen cohomology groups of the abelian group object $\ovl{K}_X(A,n)$ by Proposition~\ref{p:rigid}, and are given by the ordinary cohomology groups of $X$ with coefficients in $A$ (with a degree shift by $n$).

We also note that if $p:Y \lrar X$ is a fibration whose homotopy fibers have non-trivial homotopy groups only in a single dimension $n \geq 2$ then $Y$ is a small extension of $X$ by the Eilenberg-MacLane object $\rH \ovl{K}(A,n) \in \T_X\SS$, where $A$ is the local system of abelian groups on $X$ associating to a point $x \in X$ the $n$-th homotopy group of the homotopy fiber of $p$ over $x$. 

\subsection{Simplicial groups and equivariant spectra}\label{s:groups}

In this section we will discuss the notions of the cotangent complex and Quillen cohomology when $\M = \sGr$ is the category of simplicial groups, endowed with the model structure transferred from $\SS$ along the free-forgetful adjunction. To begin, given a group $G$, we would like to describe the tangent model category $\T_G(\sGr) = \Sp(\sGr_{G//G})$ in reasonably concrete terms. For this, it will be convenient to use the Quillen equivalence between simplicial groups and \textbf{reduced simplicial sets}. Recall that a reduced simplicial set is simply a simplicial set with a single vertex. The category $\SS^0$ of reduced simplicial sets can be endowed with a model structure in which cofibrations are the monomorphisms and weak equivalences are the weak equivalences of the underlying simplicial sets (see~\cite[\textrm{VI}.6.2]{GJ}). One then has a Quillen equivalence (see~\cite[$\mathrm{V}.6.3$]{GJ})
$$ \xymatrix@=13pt{\SS^0\ar[rr]<1ex>^(0.5){\GG} &&  \sGr \ar[ll]<1ex>_(0.5){\upvdash}^(0.5){\ovl{W}}} $$
where $\GG$ is the Kan loop group functor and $\ovl{W}G$ is a suitable reduced model for the \textbf{classiying space} of $G$. Furthermore, for each simplicial group $G$ there exists a natural Quillen equivalence between the slice model structure on $\SS_{/\ovl{W}G}$ and the functor category $\SS^{\BB G}$ from the one object simplicial groupoid $\BB G$ with automorphism group $G$ into $\SS$ (\cite{DDK}).

Since $\SS^0$ and $\sGr$ are both left proper model categories (see~\cite[Lemma 2.7]{Ber08}) and every object in $\sGr$ is fibrant the adjunction
$$ \SS^0_{\bar{W}G//\bar{W}G} \adj \sGr_{G//G}  $$
is a Quillen equivalence 
and thus induces a Quillen equivalence on model categories of spectrum objects. To compute the stabilization of $\sGr_{G//G}$, it will therefore suffice to compute the stabilization of $\SS^0_{\bar{W}G//\bar{W}G}$. 
This, in turn, can be done by comparing reduced spaces with pointed spaces. Indeed, we have a Quillen adjunction
$\L:\SS^0 \adj \SS_*:\R$
where $\L$ associates to a reduced simplicial set itself endowed with its unique base point and $\R$ is the \textbf{reduction} functor which associates to a pointed simplicial set the sub-simplicial set spanned by the $n$-simplices supported on $\ast$.

\begin{lem}\label{l:loop}
For any reduced space $X\in\SS^0$, a pointed fibrant space $Y\in\SS_*$ and a map $\L X\lrar Y$ of pointed spaces, the counit map $\L\R \lrar \id$ induces an equivalence of homotopy pullbacks $$\L X\times_{\L\R Y}^h \L X\lrar \L X\times^h_{Y} \L X.$$
\end{lem}
\begin{proof}
Let $Y_0\lrar Y$ be the inclusion of the path component of the basepoint of $Y$. Since $\L X$ is a connected space, it follows immediately that the map
$\L X\times^h_{Y_0} \L X\lrar \L X\times_Y^h \L X$ is a weak equivalence. The result now follows from the fact that the counit map factors as $\L\R Y\lrar Y_0\lrar Y$, where the first map is a weak equivalence.
\end{proof}

Let us now consider the Quillen adjunction
$$ (-)_+ : \SS \adj \SS_{\ast}: \U $$
where $X_+ = X \coprod \ast$ is the free pointed space generated from $X$ and 
$\U$ is the functor which forgets the base point. In particular, the composition $\iota = \U \circ \L: \SS^0 \lrar \SS$ is the natural inclusion. Given a simplicial group $G$ let us denote by $BG = \iota\bar{W}G$ and $BG_\ast = \L\bar{W}G$, so that $BG_{\ast}$ is a pointed version of the classifying space of $G$ and $BG$ is obtained from $BG_{\ast}$ by forgetting the base point. The following corollary provides a computation of the tangent category of $\sGr$:

\begin{cor}\label{c:compute}
Let $G$ and $X$ be as above. Then we have a diagram of Quillen adjunctions  
\begin{equation}\label{e:nice}
\vcenter{\xymatrix{
\T_G\sGr  \ar@<-1ex>[r]\ar@<1ex>[d]_-\dashv & \T_{\bar{W}G}\SS^0 \ar@<1ex>[r]^-{\simeq}\ar@<1ex>[d]_-\dashv \ar@<-1ex>[l]^-\upvdash_-{\simeq} & \T_{BG_{\ast}}\SS_* \ar@<1ex>[l]_-\upvdash\ar@<-1ex>[r]^-\upvdash\ar@<1ex>[d]_-\dashv & \T_{BG}\SS\ar@<1ex>[d]_-\dashv\ar@<-1ex>[l]_-\simeq \ar@<1ex>[r]^-\simeq & \T_{\BB G}\SS_{\ast} \ar@<1ex>[l]_-\upvdash\ar@<1ex>[d]\\
\sGr_{/G} \ar@<-1ex>[r]^-\upvdash \ar@<1ex>[u]^-{\Sigma^\infty_+}& \SS^0_{/\bar{W}G} \ar@<1ex>[r]^-\L \ar@<1ex>[u]^-{\Sigma^\infty_+} \ar@<-1ex>[l]_-{\simeq} & (\SS_*)_{/BG_{\ast}} \ar@<1ex>[l]_-\upvdash^-\R \ar@<-1ex>[r]^-\upvdash_-{\U}\ar@<1ex>[u]^-{\Sigma^\infty_+} & \SS_{/BG} \ar@<1ex>[u]^-{\Sigma^\infty_+}\ar@<-1ex>[l]_-{(-)\coprod \ast} \ar@<1ex>[r]^-\simeq & (\SS^{\BB G})_{\ast}\ar@<1ex>[u]^-{\Sigma^\infty_+}_-\dashv \ar@<1ex>[l]_-\upvdash
}}
\end{equation}
in which the adjunctions marked with $\simeq$ are Quillen equivalences.
\end{cor}
\begin{proof}
The two extreme Quillen equivalences are the ones discussed above, where we have identified $\Sp((\SS^{\BB G})_{\ast}) \cong (\Sp(\SS_\ast))^{\BB G}$ using Remark~\ref{r:global-tensored}. The second top Quillen pair is a Quillen equivalence by Lemma~\ref{l:loop} and~\cite[Corollary 2.4.9]{part0}. Finally, the third top Quillen pair is an equivalence of categories, since the adjunction $(\SS_*)_{/BG_{\ast}}\adj \SS_{/BG}$ induces an equivalence of categories upon pointification. 
\end{proof}

Composing all the equivalences of $\infty$-categories arising from Corollary~\ref{c:compute}, one may identify the tangent $\infty$-category $\T_G(\sGr_\infty) = \Sp(\sGr_{G//G})_\infty$ with the $\infty$-category of functors from (the coherent nerve of) $\BB G$ to spectra, i.e., with (naive) \textbf{$G$-equivariant spectra}, or just $G$-spectra for brevity.

\begin{rem}\label{r:central}
Let $M \in \T_G\sGr$ be an $\Om$-spectrum object whose image in $\Sp(\SS_{\ast})^{\BB G}$ under the equivalences of Corollary \ref{c:compute} is a $G$-spectrum $M' \in \Sp(\SS_\ast)$. Unwinding the definitions, we see that the simplicial group $\Om^{\infty}_+(M)$ is the semi-direct product $H \rtimes G$, where $H$ is the Kan loop group of $\Om^{\infty}_+(M') \in \SS_{\ast}$, with $G$-action induced from the $G$-action on $\Om^\infty_+(M')$. We note that $H$ is not necessarily an abelian group, but can be promoted from a group to an $\EE_\infty$-group. We also note that \textbf{small extensions} of $G$ by $M$ correspond to not-necessarily-split simplicial group extensions $\ovl{G} \lrar G$ with kernel $H$. When $M'$ is a trivial $G$-spectrum, this can be considered as an $\EE_\infty$-analogue of the notion of a \textbf{central extension}. 
\end{rem}

Given a simplicial group $G$, we would like to compute the $G$-spectrum corresponding to the cotangent complex $L_G = \LL \Sig^{\infty}_+(G)$. Since $\sGr_{/G} \adj \SS^0_{/\bar{W}G}$ is a Quillen equivalence it follows that the image of $L_G$ in $\T_{\bar{W}G}\SS^0$ is equivalent to the cotangent complex of $\bar{W}G \in \SS^0_{/\bar{W}G}$. By naturality of $\Sig^{\infty}_+$, we find that the image of $L_G$ in $\T_{BG_{\ast}}\SS_{\ast}$ is weakly equivalent to the cotangent complex of the pointed space $BG_{\ast} = \L\bar{W}G$. Let $BG = \iota\bar{W}G$ be, as above, the underlying space of $BG_{\ast}$. Identifying $\T_{BG_{\ast}}\SS_\ast \cong \T_{BG}\SS$ we may consider the cotangent complex of $X_{\ast}$ as an object of $\T_{BG}\SS$. Proposition~\ref{p:relative} then yields a homotopy cofiber sequence $\T_{BG}\SS$ of the form 
\begin{equation}\label{e:cofiber-2}
\Sig^\infty_+(x_0)\lrar \Sig^\infty_+(BG)\lrar \Sig^\infty_+(BG_{\ast})
\end{equation}
where $x_0$ denotes the corresponding object $x_0: \ast \lrar BG$ of $\SS_{/BG}$. 
In order to translate~\ref{e:cofiber-2} to a homotopy cofiber sequence of naive $G$-spectra consider the triangle of Quillen adjunctions 
$$\xymatrix@R=3.5em@C=3.5em{
\SS_{/BG}\ar@<0.9ex>[rr]\ar@<0.6ex>[dr]^(0.5){i^*} & & \SS^{\mathbb{B}G}\ar@<0.9ex>[ll]_-{\upvdash}\ar@<0.6ex>[dl]^-{i^*}\\ &\SS\ar@<0.6ex>[ur]^(0.5){i_{!}}\ar@<0.6ex>[ul]^-{i_{!}}}$$
in which the top horizontal Quillen pair is a Quillen equivalence, the right copy of $i_!$ sends a space $K$ to the free $G$-space on $K$, while the left copy of $i_!$ sends it to $K\lrar \ast \x{x_0}{\lrar} BG$. Passing to stabilizations, we obtain a triangle of Quillen adjunctions
$$\xymatrix@R=2.5em@C=2.5em{
\T_{BG}\SS\ar@<0.9ex>[rr]^-{}\ar@<0.7ex>[dr]^(0.65){i^*_{\Sp}} & & \Sp(\SS_*)^{\mathbb{B}G}\ar@<0.9ex>[ll]_-{\upvdash}^-{}\ar@<0.7ex>[dl]^-{i^*_{\Sp}}\\ &\Sp(\SS_*)\ar@<0.7ex>[ur]^(0.35){i_{!}^{\Sp}}\ar@<0.7ex>[ul]^-{i_{!}^{\Sp}}}$$
in which the horizontal adjunction is the Quillen equivalence between naive $G$-spectra and spectra parametrized by $X$ appearing in Corollary~\ref{c:compute}. By abuse of notation let us write $L_G$ for the image of $L_G$ in $\Sp(\SS_*)^{\mathbb{B}G}$. Sending the sequence~\eqref{e:cofiber-2} to its image in $\Sp(\SS_*)^{\mathbb{B}G}$ we now obtain a homotopy cofiber sequence
\begin{equation}\label{e:cofiber-3}
\xymatrix{
\mathbb{S}[G] \ar[r] & \mathbb{S}\ar[r] & L_G
}\end{equation}
where $\mathbb{S}$ is the sphere spectrum with trivial $G$-action and $\mathbb{S}[G]$ is the free $G$-spectrum on the sphere spectrum.
Indeed, $\Sigma^\infty_+(BG)$ is simply given by the constant sphere spectrum over $BG$, which maps to the sphere spectrum $\mathbb{S}$ with trivial $G$-action, while the commutativity of the above diagram shows that the object $\Sigma^\infty_+(x_0) = \Sig^{\infty}_+(i_!(\ast)) \in \T_{BG}\SS$ maps to $i_!^{\Sp}(\Sig^{\infty}_+(\ast)) = i_!^{\Sp}(\mathbb{S})$, which is the free $G$-spectrum on $\mathbb{S}$.

\begin{rem}
Let $G$ be a discrete group and $M$ a $G$-module. We may then associate to $M$ its corresponding Eilenberg-Maclane $G$-spectrum $\rH M$, and consider the Quillen cohomology groups of $G$ with coefficients in $M$. The homotopy cofiber sequence~\eqref{e:cofiber-3} shows that these Quillen cohomology groups are related to ordinary group cohomology via a long exact sequence. More precisely, since $\Sp(\SS_{\ast})^{\BB G}$ is naturally enriched in $\Sp(\SS_{\ast})$, mapping the cofiber sequence \eqref{e:cofiber-3} into $\rH M$ yields a homotopy fiber sequence of \textbf{spectra}
$$
\uline{\Map}^{\der}(L_G, \rH M)\lrar \uline{\Map}^{\der}(\mathbb{S}, \rH M) \lrar \uline{\Map}^{\der}(\mathbb{S}[G], \rH M)\simeq \rH M
$$
where $\uline{\Map}^{\der}$ denotes the derived $\Sp(\SS_{\ast})$-enriched mapping space functor. We may identify the homotopy groups of the middle spectrum with the cohomology groups of the space $BG$ with values in the local system $M$, while the right most map is induced by restriction along $\{x_0\} \hrar BG$. 
In particular, while the usual group cohomology $\rH^\bullet(G, M)$ is just the cohomology $\rH^\bullet(BG, M)$ of the classifying space $BG$, the Quillen cohomology groups $\rH^\bullet_Q(G, M)$ correspond to the \textbf{reduced} cohomology groups $\tilde{\rH}^\bullet(BG, M)$ of $BG$ as a pointed space.
\end{rem}

\subsection{Algebras over dg-operads and their modules}

Let $\C(\ZZ)$ denote the model category of unbounded chain complexes over the integers, endowed with the projective model structure. The tensor product of complexes endows $\C(\ZZ)$ with a symmetric monoidal structure which is compatible with the model structure. In particular, $\C(\ZZ)$ is a symmetric monoidal (SM) model category.

\begin{defn}
A \textbf{dg-model category} is a model category $\M$ which is tensored and cotensored over $\C(\ZZ)$.
\end{defn}

\begin{example}
Let $\A$ be a Grothendieck abelian category $\A$. Then the model category
$\C(\A)$ of unbounded chain complexes in $\A$ equipped with the injective model structure is a dg-model category.
\end{example}

\begin{example}
Let $\D$ be small dg-category (i.e., a category enriched in $\C(\ZZ)$). Then the functor category $\C(\ZZ)^{\D}$ endowed with the projective model structure is a dg-model category. Similarly, any left $\C(\ZZ)$-enriched left Bousfield localization of $\C(\ZZ)^{\D}$ is a dg-model category.
\end{example}

\begin{rem}\label{r:semi-additive}
Let $\M$ be a dg-model category. Then the underlying model category of $\M$ is tensored over $\ZZ$. It is then a classical fact that $\M$ is semi-additive, i.e., $\M$ has a zero object and the ``identity matrix'' map
$$ \I_{X_1,...,X_n}: X_1 \coprod ... \coprod X_n \lrar X_1 \times ... \times X_1 $$
is an isomorphism for every $X_1,...,X_n \in \M$. In this case, we will denote the zero object by $0$ and the (co)product by $\oplus$. 
\end{rem}

\begin{lem}
Let $\M$ be a dg-model category. Then $\M$ is stable.
\end{lem}
\begin{proof}
By Remark~\ref{r:semi-additive}, $\M$ is strictly pointed. It therefore suffices to show that the suspension functor is an equivalence.
For each $n \in \ZZ$, let $\ovl{\ZZ}[n] \in \C(\ZZ)$ be the chain complex which is the group $\ZZ$ concentrated in degree $n$. Then $\ovl{\ZZ}[n]$ is cofibrant and the functor $\ovl{\ZZ}[n] \otimes (-): \M \lrar \M$ is a left Quillen functor. Since $\ovl{\ZZ}[n] \otimes \ovl{\ZZ}[-n] \simeq \ovl{\ZZ}[0]$ is the unit of $\C(\ZZ)$ it follows that tensoring with each $\ovl{\ZZ}[n]$ is a left Quillen equivalence. We now claim that tensoring with $\ovl{\ZZ}[1]$ is a model for the suspension functor. Indeed, there is a natural cofiber sequence $\ovl{\ZZ}[0] \lrar \ovl{\ZZ}[0,1] \lrar \ovl{\ZZ}[1]$ where $\ovl{\ZZ}[0,1]$ is the complex $[\ZZ \lrar \ZZ]$ concentrated in degrees $1$ and $0$. In particular, the map $\ovl{\ZZ}[0] \lrar \ovl{\ZZ}[0,1]$ is a cofibration between cofibrant objects and the cofiber sequence is a homotopy cofiber sequence. It then follows that for each cofibrant $X \in \M$ the induced sequence $X \lrar \ovl{\ZZ}[0,1] \otimes X \lrar \ovl{\ZZ}[1]\otimes X$ is a homotopy cofiber sequence and $\ovl{\ZZ}[0,1] \otimes X \cong 0 \otimes X$ is a weak $0$-object (see Remark~\ref{r:semi-additive}). This homotopy cofiber sequence then naturally exhibits $\ovl{\ZZ}[1] \otimes X$ as the suspension of $X$ and so the desired result follows.
\end{proof}

\begin{rem}
Although every dg-model category is stable, there are stable model categories that do not admit a Quillen equivalent dg-model category. Indeed, the homotopy category of any dg-model category is tensored over $\Ho(\C(\ZZ))$ as a triangulated category, while an argument of Schwede shows that the homotopy category of \textbf{spectra} does not admit such a tensor structure (see~\cite[Proposition 1, Proposition 4]{Sch08}).
\end{rem}

Let us now consider a combinatorial dg-model category $\M$, an admissible and $\Sig$-cofibrant colored (symmetric) operad $\P$ and a fibrant-cofibrant $\P$-algebra $A \in \Alg_\P := \Alg_\P(\M)$. If the stable model structure on $\T_A\Alg_\P = \Sp((\Alg_\P)_{A//A})$ exists, then by~\cite[Corollary 4.2.6]{part1} and~\cite[Remark 4.2.3]{part1} it is Quillen equivalent to the model category $\Mod^\P_A$ of $A$-modules in $\M$. Even when the stable model structure on $\T_A\Alg_\P$ does not exist, one may still identify $\Mod^\P_A$ as a model for the $\infty$-categorical tangent category $\T_A((\Alg_\P)_\infty)$ (see~\cite[\S 4.3]{part1}). Consequently, one may attempt to define the cotangent complex and Quillen cohomology of a $\P$-algebra while working directly with $\Mod^\P_A$, without explicit reference to the stabilization process. Classically, the notions of the cotangent complex and the associated cohomology theory were indeed developed in this way (typically when $\M$ is the category of chain complexes over a field) using suitable operadic analogues of \textbf{K\"ahler differentials} and \textbf{square-zero extensions} (see~\cite{Hin97},~\cite{GH00},~\cite{Mil11}). In this section we will unwind the relation between the abstract definition of the cotangent complex and the concrete one using K\"ahler differentials in the setting of $\P$-algebras taking values in a dg-model category $\M$. 

\begin{defn}
Let $\M$ be a symmetric monoidal dg-model category and let $\P$ be a (symmetric, colored) operad in $\M$. The \textbf{reduction} $\P_{\red}$ of $\P$ is the operad (with the same set of colors) which agrees with $\P$ in all arities $\geq 1$ and has 0-objects in arity 0. In particular, $\P_{\red}$-algebras are just \textbf{non-unital} $\P$-algebras. We denote by $\eta: \P_{\red} \lrar \P$ the natural inclusion.
\end{defn}

\begin{lem}\label{l:reduced}
Let $\M$ be a symmetric monoidal dg-model category and let $\P$ be an admissible operad in $\M$. Then the map $\eta: \P_{\red}\lrar \P$ induces an equivalence between categories of augmented algebras
$$\xymatrix{
\eta_!^{\aug}: \Alg_{\P_{\red}} = \Alg_{\P_{\red}}^{\aug}\ar@<1ex>[r] & 
\Alg_\P^{\aug} \ar@<1ex>[l]_-{\upvdash} : {\eta_{\aug}^*}
}$$
which identifies the model structures on both sides.
\end{lem}
\begin{proof}
The functor $\eta^*_{\aug}$ sends an augmented $\P$-algebra $\P_{0}\lrar A\lrar \P_{0}$ to $A\times_{\P_{0}} 0$ while $\eta^{\aug}_!$ sends a $\P_{\red}$-algebra $B$ to $\P_{0}\lrar \P_{0}\oplus B\lrar \P_{0}$. The fact that $\M$ is semi-additive (see Remark~\ref{r:semi-additive}) implies that these two functors are mutual inverses. The description of $\eta^{\aug}_!$ and the fact that $\oplus$ is the product in $\M$ show that $\eta^{\aug}_!$ preserves (trivial) fibrations, which are just (trivial) fibrations in the base category $\M$. Since $\eta_{\aug}^*$ clearly preserves (trivial) fibrations, it follows that the model structures on both sides are identified (in particular, $\P_{\red}$ is admissible as well).
\end{proof}

If $\P$ is an operad in $\M$ and $A$ is a $\P$-algebra, then there exists an operad $\P^A$ (with the same set of colors), called the \textbf{enveloping operad} of $\P$, such that $\P^A$-algebras are the same as $\P$-algebras under $A$ (see~e.g.~\cite{BM07}, \cite[\S 3.1]{part1}). We recall as well the operad $\P^A_1$ whose $1$-ary operations are the same as those of $\P^A$ and which has 0-objects in all other arities. In particular, algebras over $\P^A_1$ are the same as $A$-modules in $\M$.

When $\M$ is a dg-model category we may consider the natural inclusion $\eta:\P^A_{\red} = (\P^A)_{\red} \lrar \P^A$, as well as a collapse map $\theta: \P^A_{\red}\lrar \P^A_1$ sending all non-unary operations to zero. While $\eta_!: \Alg^{\P^A_{\red}} \lrar \Alg^{\P^A}_{\aug}$ is generically a left Quillen functor, Lemma~\ref{l:reduced} implies that when $\M$ is a dg-model category $\eta_!$ is furthermore a right Quillen functor. This enables one to make the following definition:
\begin{define}\label{d:sqz}
Let $\M$ be a symmetric monoidal dg-model category, $\P$ an admissible operad in $\M$ and $A$ a $\P$-algebra such that $\P^A_1$ is admissible. We define the \textbf{square-zero extension functor} $(-)\rtimes A: \Mod^{\P}_A \lrar (\Alg_{\P})_{/A}$ to be the composition of right Quillen functors
$$ (-)\rtimes A:\Mod^{\P}_A =\Alg_{\P^A_1}\x{\theta^*}{\lrar} \Alg_{\P^A_{\red}} \x{\eta^{\aug}_!}{\lrar} (\Alg_{\P})_{A//A} \lrar (\Alg_{\P})_{/A} .$$
We will refer to its left adjoint $\Om^{/A}: (\Alg_{\P})_{/A} \lrar \Mod^{\P}_A$ as the \textbf{functor of K\"ahler differentials} and to $\Om_A \x{\df}{=} \Om^{/A}(\id_A)$ as the \textbf{module of K\"ahler differentials of $A$}. 
\end{define}

Unwinding the definitions, we can identify $M\rtimes A \in \Alg^{\P}_{/A}$ with the $\P$-algebra $M\oplus A$ equipped with the \textbf{square-zero} $\P$-algebra structure, which is determined by the property that for every collection of elements $w_1,...,w_n,w_\ast$ in the color set $W$, the structure map
\begin{equation}\label{e:sqz}\vcenter{\xymatrix@C=1.2pc{
\P(w_1,...,w_n;w_\ast)\otimes (M(w_1) \oplus A(w_1))\otimes ... \otimes (M(w_n) \oplus A(w_n)) 
 \ar[r] & M(w_\ast) \oplus A(w_{\ast})
}}\end{equation}
vanish on summands of the left hand side which have more than one factor from $M$, and are given by the $A$-module structure of $\M$ on the other components.

We are now in a position to relate the abstract cotangent complex to the notions of K\"ahler differentials and square-zero extensions: 
\begin{pro}\label{p:cotangentvskaehler}
Let $\M$ be a symmetric monoidal combinatorial dg-model category and $\P$ a $\Sig$-cofibrant operad in $\M$. Let $A \in \Alg_\P$ be a fibrant-cofibrant algebra. Assuming that all the involved model structures exist, there is a commuting diagram of right Quillen functors
\begin{equation}\label{d:operadiccotangentadditive}\vcenter{\xymatrix@C=3pc{
\T_A\Alg_{\P}\ar[d]_{\Omega_+^\infty} & \T_0\Mod^{\P}_A\ar[l]^-\sim_-{\Sp((-)\rtimes A)}\ar[d]_\sim^{\Omega_+^\infty}\\
(\Alg_\P)_{/A} & \Mod^{\P}_A\ar[l]^-{(-)\rtimes A}
}}\end{equation}
in which the right adjoints marked by $\sim$ are right Quillen equivalences.
\end{pro}
\begin{proof}
The existence of the commuting square follows from the naturality of $\Sp(-)$, see~\cite{part0}, and the left vertical right Quillen functor is a right Quillen equivalence since $\Mod^\P_A$ is stable, see~\cite[Corollary 3.3.3]{part0}. To show that the top horizontal right Quillen functor is an equivalence let $\K: (\Alg_{\P})_{A//A} \lrar \Mod^{\P}_A$ be the composition of the forgetful functor $(\Alg_{\P})_{A//A} \lrar (\Mod^\P_A)_{A//A}$ and the kernel functor $\ker:(\Mod^\P_A)_{A//A} \lrar \Mod^\P_A$. Since $\P$ is $\Sig$-cofibrant and $A$ is cofibrant it follows that $\P^A$ is $\Sig$-cofibrant (see~\cite[Proposition 2.3]{BM09}), and since $\M$ is stable~\cite[Corollary 4.2.6]{part1} implies that $\K$ induces a right Quillen equivalence 
$$ \K_{\Sp}: \T_A\Alg_\P = \Sp((\Alg_{\P})_{A//A}) \lrar \Sp(\Mod^\P_A) = \T_0\Mod^\P_A .$$ 
Since the square-zero extension functor $(-) \rtimes A: \Mod^{\P}_A\lrar (\Alg_{\P})_{A//A}$ is right inverse to $\K$ it follows that the $\Sp((-)\rtimes A)$ is a right Quillen equivalence, as desired.
\end{proof}

\begin{cor}\label{c:kahler}
Let $\M, \P$ and $A$ be as in Proposition~\ref{p:cotangentvskaehler}. Then the image in $\Mod^{\P}_A$ of the cotangent complex $L_A = \LL\Sig^{\infty}_+(A) \in \T_A\Alg_{\P}$ of $A$ under the equivalence $\T_A\Alg_{\P} \simeq \T_0\Mod^{\P}_A \simeq \Mod^{\P}_A$ is weakly equivalent to the derived module of K\"ahler differentials $\LL\Om^{/A}(A) \in \Mod^{\P}_A$. 
\end{cor}

\begin{rem}
If $M$ is an $A$-module, then Corollary~\ref{c:kahler} implies that the Quillen cohomology groups of $A$ with coefficients in $M$ are given by
$$ \rH^n_Q(A;M) = \pi_0\Map^{\der}_{\Mod^{\P}_A}(\LL\Om^{/A}(A), M[n]) .$$
These groups are also known as \textbf{operadic Andr\'e-Quillen cohomology groups}.
\end{rem}

Let us now discuss the functor $\Om^{/A}$ in further detail. By definition, this functor is given in general by the composition of left adjoints
$$ (\Alg_{\P})_{/A} \x{\coprod A}{\lrar} (\Alg_{\P})_{A//A} \x{\eta^*_{\aug}}{\lrar} \Alg_{\P^A_{\red}} \x{\theta_!}{\lrar} \Mod^{\P}_A $$
which is somewhat inexplicit. Given a map $f: B \lrar A$ of $\P$-algebras the commutative diagram of left Quillen functors
$$ \xymatrix@C=2pc{
(\Alg_{\P})_{/B} \ar^{f_!}[r]\ar_{\Om^{/B}}[d] & (\Alg_{\P})_{/A} \ar^{\Om^{/A}}[d] \\
\Mod^\P_B \ar_{(-) \otimes_B A}[r] & \Mod^{\P}_A \\
}$$
yields a natural isomorphism $\Om^{/A}(B) \cong \Om^{/B}(B) \otimes_B A$. It is hence sufficient, in principle, to compute the module of K\"ahler differentials $\Om_A= \Om^{/A}(A)$ for a $\P$-algebra $A$. 

Let $\F: \M \lrar \Alg_\P$ be the free algebra functor. When $B = \F(V)$ is a free $\P$-algebra and the map $\F(V) \lrar A$ is induced by a map $V \lrar A$ in $\M$, then for every $A$-module $M$ we have natural identifications
$$
\Hom_{\Mod^{\P}_A}(\Om^{/A}(B), M)\cong \Hom_{(\Alg_\P)_{/A}}(B,M \rtimes A) \cong \Hom_{\M_{/A}}(V,M \oplus A) \cong \Hom_{\M}(V,M).
$$
It follows that $\Om^{/A}(B) = \Om^{/A}(\F(V))$ is the free $A$-module generated by $V$. This allows one to compute $\Om_A$ by applying the functor $\Om^{/A}$ to the coequalizer diagram

\begin{equation}\label{e:coeq}
\xymatrix{
\F(\F(A))\ar@<0.5ex>[r] \ar@<-0.5ex>[r] & 
\F(A) \ar[r] & A \\
}
\end{equation}
yielding a description of $\Om_A$ by generators and relations. Similarly, to compute the derived counterpart $\LL\Om^{/A}(A)$ one may often represent $A$ as the geometric realization of a simplicial diagram consisting of free algebras (using the comonadic bar resolution, for example), yielding a description of $\Om_A$ as a geometric realization of free $A$-modules. We refer the reader to~\cite[\S 1.2]{Mil11} for further details in the case where $\M$ is the category of chain complexes over a field of characteristic $0$.

\begin{example}
Let $\P$ be a single colored operad in abelian groups and let $A$ be a $\P$-algebra in $\Ab$ (here we may consider $\Ab$ as equipped with the trivial model structure). Applying $\Om^{/A}$ to the coequalizer sequence~\ref{e:coeq} and unwinding the definitions we obtain a description of $\Om_A$ as the free $A$-module generated by the underlying abelian group of $A$ modulo the \textbf{Leibniz relations}.  
More explicitly, we may identify $\Om_{A}$ with the $A$-module generated by the formal elements $d(a)$ for $a \in A$, modulo the relations
$$ d(a + b) = d(a) + d(b) $$
$$ d(f(a_1,...,a_n))= \sum_{i=1}^{n}f(a_1,...,a_{i-1},d(a_i),a_{i+1},...,a_n) $$
for every $a_1,...,a_n \in A$ and $n$-ary operation $f \in \P(n)$.
\end{example}

\begin{example}
Let $\M$ be the model category of chain complexes over a field of characteristic $0$ and let $\P$ be the \textbf{commutative} operad. In this case, the module $\Om_A$ of K\"ahler differentials of a $\P$-algebra $A$ can be explicitly tracked down as follows. First note that the image of $\Id_A \in (\Alg_{\P})_{/A}$ in $(\Alg_{\P})_{A//A}$ is given by $A \lrar A \otimes A \lrar A$. The functor $\eta^*_{\aug}$ sends this object to the kernel $I \subseteq A \otimes A$ of the multiplication map $A \otimes A \lrar A$, considered as a non-unital commutative dg-algebra. Finally, the functor $\theta_!$ sends $I$ to $I/I^2$, yielding the classical description of the module of K\"ahler differentials.
\end{example}

We shall now restrict our attention to the case where $\M = \C(k)$ is the dg-model category of unbounded chain complexes over a field $k$ of characteristic $0$, equipped with the projective model structure. Let $\P$ be a cofibrant single-colored operad in $\C(k)$. In this case $\Alg_\P := \Alg_\P(\C(k))$ is a left proper model category by the results of~\cite{fresse} (see~\cite[Remark 4.1.2]{part1}) and so we may form its tangent bundle $\T\Alg_\P$ as in \S\ref{s:tangent}. Furthermore, since $\Alg_\P$ is also right proper, \cite[Theorem 3.1.9]{part0} implies that the projection $\T\Alg_\P \lrar \Alg_{\P}$ is a \textbf{model fibration} in the sense of~\cite{HP}. On the other hand, by Corollary~\ref{c:kahler} the fiber of this model fibration over a $\P$-algebra $A$ is Quillen equivalent to $\Mod^\P_A$. This suggests that we might model the tangent bundle of $\Alg_\P$ by the Grothendieck construction of the various categories of modules.
In what follows we will explain how this can be done using the machinery of~\cite{HP}. This will also allow one to present the global cotangent complex functor (see~\S\ref{s:tangent}) in terms of K\"ahler differentials. 

Recall that a map $f\colon A\lrar B$ of $\P$-algebras induces a Quillen adjunction
$$\xymatrix{
f_!\colon \Mod^{\P}_A \ar@<1ex>[r] & 
\Mod^{\P}_B \ar@<1ex>[l]_-{\upvdash} \colon f^*.
}$$
The association $A \mapsto \Mod^{\P}_A$ thus determines a functor $\Mod^\P:\Alg_{\P} \lrar \ModCat$ from $\P$-algebras to model categories. The Grothendieck construction of this functor is the category $\int_A \Mod^{\P}_A$ whose objects are pairs $(A, M)$ where $A$ is a $\P$-algebra and $M$ is an $A$-module. Morphisms are given by pairs $(f,\vphi)$ where $f\colon A \lrar B$ is a map of $\P$-algebras and $\vphi: M \lrar f^*N$ is a map of $A$-modules (equivalently, the latter map can be encoded in terms of its adjoint $\vphi^{\ad}: f_!M\lrar N$).

The resulting category of $\P$-algebras and modules over them admits a simple description in terms of operads: to the operad $\P$ one may associate a 2-colored operad $\M\P$ whose category of algebras $\Alg_{\M\P}(\C(k))$ is naturally isomorphic to $\int_A\Mod^{\P}_A$ (see, e.g. \cite[\S 5]{hin-rectification}). By~\cite[Theorem 2.6.1]{hin-rectification} and~\cite[Example 2.5.4]{hin-rectification} the 2-colored operad $\M\P$ is admissible over $\C(k)$. In particular, in the notation above, we have that a map $(f,\vphi): (A,M) \lrar (B,N)$ is a (trivial) fibration if and only if $f$ is a (trivial) fibration and $\vphi^{\ad}: M \lrar f^*N$ is a (trivial) fibration.  

\begin{pro}
Let $\P$ be a cofibrant single-colored operad in $\C(k)$. Then the canonical projection $\int_{A \in \Alg_\P} \Mod^\P_A \lrar \Alg_\P$ is a \textbf{model fibration} in the sense of~\cite[\S 5]{HP}.
\end{pro}
\begin{proof}
We first show that the functor $\Mod^\P:\Alg_{\P} \lrar \ModCat$ which sends $A$ to $\Mod^{\P}_A$ is proper and relative in the sense of~\cite[\S 3]{HP}.
To see that $\Mod^\P$ is relative (i.e., that $\Mod^\P$ sends weak equivalences to Quillen equivalences), note that we may identify it with the composite
$$\xymatrix{
\Alg_{\P} \ar[r]^{\P^{(-)}(1)} & \Alg \ar[r]^-{\LMod} & \ModCat
}$$
where $\P^{(-)}(1)$ sends a $\P$-algebra to its enveloping algebra (i.e.~the algebra of $1$-ary operations of its enveloping operad), and $\LMod$ sends an associative algebra to its category of left modules. Since all objects in $\C(k)$ are cofibrant, it follows from \cite[17.4.1.B]{fresse} that the functor $\P^{(-)}(1)$ is relative. As stated in~\cite[Theorem 6.3.10]{HP}, the functor $\LMod:\Alg \lrar \ModCat$ is relative as well and so we may conclude that $\Mod^\P$ is relative. Now the projection $\Alg_{\M\P} = \int_A\Mod^{\P}_A \lrar \Alg_\P$ is right Quillen with respect to the transferred model structure, and hence~\cite[Corollary 5.0.13]{HP} implies that the functor $\Mod^\P$ is left proper, thus proper, since all object of $\int_{A\in\Alg^\P}\Mod^{\P}_A$ are fibrant in their respective fibers. The main theorem of~\cite[\S 5]{HP} now endows $\int_{A \in \Alg_\P}\Mod^\P_A$ with a model structure such that the projection $\int_{A \in \Alg_\P}\Mod^\P_A \lrar \Alg_\P$ is a model fibration. An inspection of the fibrations and trivial fibrations of the resulting model structure reveals that they are the same as those of the transferred model structure, and hence the two model structures must coincide.
\end{proof}

The right Quillen functors from Diagram \eqref{d:operadiccotangentadditive} are natural in $A$ and assemble to form a square of Quillen morphisms in the sense of~\cite{HP}. This results in a square of Quillen adjunctions over $\Alg_{\P}$
$$\xymatrix{
\T\Alg_\P\ar@<1ex>[d] \ar@<1ex>[r] & \int_{A\in \Alg_\P} \Sp(\Mod^{\P}_A)\ar@<1ex>[l]_-\sim \ar@<1ex>[d]\\
\Alg_{\P}^{[1]} \ar@<1ex>[r]^-{\Omega^\P} \ar@<1ex>[u]_{\dashv} & \int_{A\in \Alg_{\P}}\Mod^{\P}_A \ar@<1ex>[u]_\sim \ar@<1ex>[l]^-{\ltimes}_-{\upvdash}.
}$$
where the adjunctions marked by $\sim$ are Quillen equivalences.
We may now compose the bottom Quillen pair with the diagonal adjunction $\Alg^{\P} \adj (\Alg^{\P})^{[1]}$ to produce a Quillen pair
$$\xymatrix{
\Om^{\P}_{\int}:\Alg_{\P}\ar@<1ex>[r]^-{\Delta} & \Alg_{\P}^{[1]} \ar@<1ex>[r]^-{\Omega^\P}\ar@<1ex>[l]^-{\dom}_-{\upvdash} & \int_{A\in \Alg_{\P}}\Mod^{\P}_A  \ar@<1ex>[l]^-{\ltimes}_-{\upvdash}
}: \rtimes. $$ 
The value of $\Omega^\P_{\int}$ on a $\P$-algebra $A$ is simply the pair $(A,\Om_A) \in \int_{A\in \Alg_{\P}}\Mod^{\P}_A$, where $\Om_A$ is the $A$-module of K\"ahler differentials discussed above.
\begin{cor}
The image of the global cotangent complex $\LL\Sig^{\infty}_{\int}(A)$ of $A$ in $\int_{A\in \Alg_{\P}}\Mod^{\P}_A$ is weakly equivalent to $\LL\Om^\P_{\int}(A)$.
\end{cor}

The cotangent complex classically appears in the setting of commutative (dg--)algebras, where it plays an important role in deformation theory. We will finish this section by sketching the relation between first order deformations of (dg-)algebras and small extensions, which in turn are controlled by the abstract cotangent complex and Quillen cohomology (as discussed in~\S\ref{s:quillen}).

Denote by $\CAlg=\CAlg(\C(k))$ the model category of commutative dg-algebras over $k$. Let $k[\eps]$ be the dg-algebra obtained from $k$ by adjoining an element $\eps$ of degree $0$ satisfying $\eps^2 =  0$ and $d(\eps) = 0$. Then $k[\eps] \cong k \rtimes k$ is a split square-zero extension of $k$ (see Definition~\eqref{d:sqz} and Example \ref{e:trivial-1}), and we may identify it with $\Om^{\infty}_+k \in \CAlg_{/k}$ where $k \in \Sp(\CAlg_{k//k}) \simeq \C(k)$ is considered as a chain complex concentrated in degree $0$ (see Proposition~\ref{p:cotangentvskaehler}). Indeed, the projection $k[\eps]\lrar k$ fits into a pullback square of commutative dg-algebras (over $k$)
\begin{equation}\label{e:split}\vcenter{\xymatrix{
k[\eps]= k \rtimes k \ar[r]\ar@<3ex>[d] & k  \ar[d]\\
k[0, 1]\rtimes k \ar[r] & k[1]\rtimes k\\
}}\end{equation}
where $k[0, 1]=\cone(k) = [k \lrar k]$ is the cone of $k$, consisting of two copies of $k$ in degrees $0$ and $1$ respectively. Note that 
this square is in fact homotopy Cartesian, because the bottom horizontal map is a fibration. 

Now let $R \in \CAlg^{\geq 0}$ be a commutative dg-algebra concentrated in non-negative degrees. A \textbf{first order deformation} of $R$ is a cofibrant commutative $k[\eps]$-dg-algebra $\ovl{R}$, which is concentrated in non-negative degrees, together with a weak equivalence $\eta:\ovl{R} \otimes_{k[\eps]}k \x{\simeq}{\lrar} R$. Let us denote $R' = \ovl{R} \otimes_{k[\eps]}k$. 

We shall now explain how one can consider first order deformations as small extensions. Since homotopy limits of dg-algebras are created by the projection to chain complexes and tensoring with $\ovl{R}$ preserves finite homotopy limits, it follows that the functor $(-) \otimes_{k[\eps]} \ovl{R}: \CAlg_{k[\eps]//k} \lrar \CAlg_{/R'}$ preserves homotopy Cartesian squares. Applying this functor to~\eqref{e:split} and using the weak equivalence $\eta: R' \lrar R$ we obtain a homotopy Cartesian square 
$$ \xymatrix{
\ovl{R} \ar[r]\ar[d] &  R \ar[d] \\
R'' \ar[r] & R[1] \rtimes R\\
}$$
in $\Alg_{/R}$, where we have denoted $R'' = (k[0,1]\rtimes k) \otimes_{k[\eps]} \ovl{R}$. Identifying the underlying $k[\eps]$-module of $k[0,1] \rtimes k$ with the cone of the $k[\eps]$-module map $k \lrar k[\eps]$ sending $1$ to $\eps$, we obtain an identification of the underlying $\ovl{R}$-module of $R''$ with the cone the $\ovl{R}$-module map $R' \lrar \ovl{R}$ sending $1$ to $\eps$. Identifying $R[1] \rtimes R \simeq \Om^{\infty}_+(R[1]) \in \CAlg_{/R}$, the right vertical map becomes the $0$-section of the structure map $\Om^{\infty}_+(R[1])\lrar R$ and the square exhibits $\ovl{R}$ as a small extension of $R''$ by the $R$-module $R$. We now observe that the map $R'' \lrar R$ is a weak equivalence in $\CAlg_{/R'}$ and hence we may in principle consider this data as a small extension of $R$ by itself. We note that when $R$ is cofibrant we may choose a homotopy inverse weak equivalence $R \lrar R''$ and use it to obtain an honest small extension of $R$, but of course in general we do not expect the corresponding Quillen cohomology class to be realizable as an actual map out of $R$.

The above discussion shows that any first order deformation of $R$ can be viewed as a small extension of $R$ by $R$. A fundamental theorem of deformation theory refines this result, asserting that first order deformations of $R$ in the above sense, considered up to a suitable notion of equivalence, are in fact in bijection with equivalence classes of small extensions of $R$ by $R$, and are hence classified by the Quillen cohomology group $\rH^1_Q(R,R)$. More generally, for every chain complex $M$ one may consider first order deformations of $R$ over the square-zero extension $\Om^{\infty}M = M \rtimes k$. These correspond to small extensions of $R$ by $R \otimes M$, and are classified by the Quillen cohomology group $\rH^1_Q(R,R \otimes M)$.

\subsection{The Hurewicz principle}\label{s:white}

The classical Hurewicz theorem asserts that a map $f: X \lrar Y$ of simply connected spaces is a weak equivalence if and only if it induces an isomorphism on cohomology groups with all possible coefficients . One can extend this theorem to non-simply connected spaces by including cohomology groups with \textbf{local coefficients}. In this case one obtains that a map $f: X \lrar Y$ is a weak equivalence if and only if it induces an equivalence on fundamental groupoids and an isomorphism on cohomology groups with all local coefficients. One way to prove this theorem is by applying obstruction theory along the Postnikov tower of $Y$. In this section we will study this procedure in the general context of spectral Quillen cohomology. We will identify sufficient conditions for a Hurewicz type theorem to hold in a general model category, and discuss three notable examples, namely spaces, simplicial algebras and simplicial categories.

Throughout this section let us fix a left proper combinatorial model category $\M$. Let $X$ be a fibrant object of $\M$ and let $M \in \Sp(\M_{X//X})$ be a fibrant $\Om$-spectrum object. 
Consider a lifting problem in a diagram in $\M_{/X}$ of the form
\begin{equation}\label{e:square-2}
\vcenter{\xymatrix{
A \ar_{f}[d]\ar[r] & X_{\alp} \ar^{p_\alp}[d] \ar[r] & \Om^\infty_+(0') \ar^{s_0'}[d] \\
B \ar@{-->}[ur]\ar^{g}[r] & X \ar^-{\alp}[r] & \Om^{\infty}_+(M[1]) \\
}}
\end{equation}
where $f: A \lrar B$ is a cofibration, $X_\alp \lrar X$ is a fibration and the right square exhibits $X_\alp$ as a small extension of $X$ by $M$ (see \S\ref{s:quillen}). Factorizing the map $0'\lrar M[1]$ as a weak equivalence, followed by a fibration $0''\lrar M[1]$, we find that the map $X_\alpha\lrar X$ is weakly equivalent (over $X$) to the (homotopy) pullback $X\times_{\Om^\infty_+(M[1])} \Om^\infty(0'')\lrar X$. Replacing $0'$ by $0''$ and $X_\alpha$ by this pullback, we may (and will) therefore assume that the map $0'\lrar M[1]$ is a fibration and that the right square is Cartesian (hence homotopy Cartesian). In that case, finding the desired lift $B\lrar X_\alpha$ is equivalent to finding a diagonal lift in the square
$$\vcenter{\xymatrix{
A \ar_{f}[d]\ar[r] & \Om^\infty_+(0') \ar^{s_0'}[d] \\
B \ar@{-->}[ur]\ar_-{\alpha g}[r] & \Om^{\infty}_+(M[1]).
}}$$
The above square in $\M_{/X}$ is equivalent by adjunction to a square in $\Sp(\M_{X//X})$ of the form
$$\xymatrix{
\Sig^\infty_+(A)\ar[d]_{\Sig^\infty_+(f)}\ar[r] & 0'\ar[d]\\
\Sig^\infty_+(B)\ar[r]\ar@{-->}[ru] & M[1].
}$$
In particular, we obtain a map from the pushout $\Sig^\infty_+(B)\coprod_{\Sig^\infty_+(B)}0'$ into $M[1]$. Since the left vertical map is a cofibration, $0'$ is a weak zero object and $\Sp(\M_{X//X})$ is left proper, this pushout is a model for the homotopy cofiber of the map $\Sig^\infty_+(A)\lrar \Sig^\infty_+(B)$. By definition, this is equivalent to the image $g_!L_{B/A}$ of the relative cotangent complex under cobase change along the map $g: B\lrar X$. Using again the left properness of $\Sp(\M_{X//X})$ we may conclude that the above lifting problem is equivalent to a lifting problem of the form
$$\xymatrix{
& 0'\ar[d]\\
g_!L_{B/A}\ar[r]_\beta \ar@{-->}[ru] & M[1]
}$$
We now observe that the bottom map $\beta$ (or, more precisely, its adjoint map $L_{B/A}\lrar g^*M[1]$) determines a class $[\beta]\in \rH_Q^1(B, A; M)$ in the \textbf{relative Quillen cohomology} of $B$ under $A$ with coefficients in $g^*M$. This element is trivial if and only if the map $g_!L_{B/A}\lrar M[1]$ is null-homotopic, i.e.\ if and only if there exists a dotted lift in the above diagram (equivalently, in Diagram~\eqref{e:square-2}). Furthermore, the entire \textbf{derived space of lifts} of the original diagram \eqref{e:square-2} can be identified with the derived space of sections of the fibration $0'\lrar M[1]$ over the map $\beta: g_!L_{B/A}\lrar M[1]$, i.e.\ the space of null-homotopies of the map $\beta$. In other words, the space of derived lifts is equivalent to the homotopy fiber
$$
\Map^{\der}(g_!L_{B/A}, 0')\times^h_{\Map^{\der}(g!L_{B/A}, M[1])}\{\beta\} \simeq \{0\}\times^h_{\Map(g_!L_{B/A}, M[1])} \{\beta\}.
$$
which can be identified as the space of paths from $\beta$ to the 0-map in the space $\Map^{\der}(g_!L_{B/A}, M[1])$. This space of paths is empty if $\beta$ and $0$ live in different path components (i.e. $[\beta]\neq 0$) and when $[\beta]=0$, it is a torsor over the loop space of $\Map^{\der}(g_!L_{B/A}, M[1])$ at the zero map. To sum up, the \textbf{obstruction} to a lift against a small extension is a certain natural class $[\beta] \in \rH^1_Q(B,A;g^*M)$ in the relative Quillen cohomology of $B$ under $A$. When $[\beta]=0$, a choice of null-homotopy for $\beta$ identifies the space of derived lifts is with the space 
$$
\Om\Map^{\der}(g_!L_{B/A}, M[1])\simeq \Map^{\der}(g_!L_{B/A}, M)\simeq \Map^{\der}(L_{B/A}, g^*M)
$$ 
whose $n$'th homotopy group is isomorphic to the $(-n)$'th relative Quillen cohomology group $\rH_Q^{-n}(B,A;g^*M)$. 
\begin{cor}\label{c:etale}
Let $\M$ be a left proper combinatorial model category and let $f: A \lrar B$ be a map whose relative cotangent complex vanishes. Let $p: Y \lrar X$ be a map which is an inverse homotopy limit of a tower of small extensions. Then the square
$$ \xymatrix{
\Map^{\der}(B,Y) \ar[r]\ar[d] & \Map^{\der}(A,Y) \ar[d] \\
\Map^{\der}(B,X) \ar[r] & \Map^{\der}(A,X) \\
}$$
is homotopy Cartesian.
\end{cor}
\begin{proof}
Clearly it suffices to prove this for the case where $p$ is a small extension. We may assume that $p$ is a fibration between fibrant objects and $f$ is a cofibration.The desired result now follows from the fact that the square~\ref{e:square-2} admits a contractible space of lifts when the relative cotangent complex of $f$ vanishes, as explained above.
\end{proof}

\begin{cor}[The Hurewicz principle]\label{c:white}
Let $\L:\M \adj \N:\R$ be a Quillen adjunction with $\M$ left proper combinatorial, such that the following property holds: for every cofibrant $X \in \M$ the derived unit map $X \lrar \RR \R\L(X)$ is the homotopy limit of a tower of small extensions. Let $f: A \lrar B$ be a map. Then the following are equivalent:
\begin{enumerate}[(1)]
\item
$f$ is a weak equivalence.
\item
$\LL\L(f)$ is a weak equivalence and the relative cotangent complex $L_{B/A}$ of $f$ is a weak zero object, i.e.\ $f$ induces an isomorphism $\rH_Q^\bullet(B; M)\overset{\cong}{\lrar} \rH^\bullet_Q(A; f^*M)$ on Quillen cohomology with coefficients in any object $M\in \T_B\M$.
\end{enumerate}
\end{cor}
\begin{rem}\label{r:refined}
More generally, let $\E$ be a class of objects in $\T\M$ and suppose that for any $X$ the unit map $X\lrar \RR\R\L(X)$ is the homotopy limit of a tower of small extensions, each of which is the base change of a map $\Om^\infty_+(0)\lrar \Om^\infty_+(E)$ with $E\in \E$. Then a map $f: A\lrar B$ is a weak equivalence if its image $\LL\L(f)$ is a weak equivalence and if $f$ induces isomorphisms on Quillen cohomology groups with coefficients from $\E$. 
\end{rem}
\begin{proof}
It is clear that (1) implies (2), so assume that (2) holds. To show that $f$ is a weak equivalence it will suffice to show that for any fibrant-cofibrant object $X \in \M$ the induced map
$$ \Map^{\der}_\M(B,X) \lrar \Map^{\der}_\M(A,X) $$
is a weak equivalence. Consider the commutative square
\begin{equation}\label{e:mapping-space}
\vcenter{\xymatrix{
\Map^{\der}(B,X) \ar[r]\ar[d] & \Map^{\der}(A,X) \ar[d] \\
\Map^{\der}(B, \RR \R\L(X)) \ar[r] & \Map^{\der}(A, \RR \R\L(X)). \\
}}
\end{equation}
The bottom horizontal map is a weak equivalence by adjunction, since $\LL \L(A) \lrar \LL\L(B)$ is assumed to be a weak equivalence. By Corollary~\ref{c:etale} the square is homotopy Cartesian, and hence the top horizontal map is a weak equivalence as well.
\end{proof}

\begin{example}[(The classical Hurewicz theorem)]\label{e:white-spaces}
Let $\M = \SS$ be the category of simplicial sets endowed with the Kan-Quillen model structure and let $L_1\M$ be the left Bousfield localization of $\M$ whose fibrant objects are the $1$-truncated Kan complexes. The unit map of the canonical adjunction $\M \adj L_1\M$ can be identified with the first Postnikov piece map $X \lrar P_1(X)$.
This map can be factored as an inverse limit of the Postnikov tower 
$$  \cdots \lrar P_n(X) \lrar \cdots \lrar P_2(X) \lrar P_1(X) .$$
Furthermore, for every $n \geq 1$ the map $P_{n+1}(X) \lrar P_n(X)$ is naturally a small extension of $P_n(X)$ by the Eilenberg-McLane spectrum object $\rH \ovl{K}(\pi_{n+1}(X),n+1) \in \Sp(\SS_{P_n(X)//P_n(X)})$ (see~\S\ref{s:spaces}), where $\pi_{n+1}(X)$ is considered as a local system of abelian groups on $P_n(X)$ .

It follows that the adjunction $\M \adj L_1\M$ satisfies the conditions of Corollary~\ref{c:white}, and so we may conclude that a map $f: X \lrar Y$ of spaces is an equivalence if and only if it induces an equivalence on first Postnikov pieces and an isomorphism on Quillen cohomology with arbitrary coefficients. By Remark \ref{r:refined}, it is in fact sufficient to require $f$ to induce an isomorphism on cohomology with local coefficients of abelian groups.
\end{example}

\begin{example}[(Hurewicz theorem for simplicial algebras)]\label{e:white-algebras}
Let $k$ be a field of characteristic zero and let $\M$ be the category of simplicial $k$-modules with its model structure transferred along the free-forgetful adjunction with simplicial sets. We note that $\M$ carries a natural symmetric monoidal structure given by levelwise tensor product of $k$-modules. Let $\P$ be a cofibrant operad in $\M$. In~\cite{GH00} it is shown that any $\P$-algebra $A$ admits a Postnikov tower 
$$  \cdots \lrar P_n(A) \lrar \cdots \lrar P_2(A) \lrar P_1(A) \lrar P_0(A) $$
in $\P$-algebras such that for any $n \geq 0$ the map $P_{n+1}(A) \lrar P_{n}(A)$ is a small extension. Let $L_0\Alg_\P(\M)$ be the left Bousfield localization of $\Alg_{\P}(\M)$ whose fibrant objects are the discrete simplicial $\P$-algebras. It follows that the adjunction $\Alg_{\P}(\M) \adj L_0\Alg_{\P}(\M)$ satisfies the conditions of Corollary~\ref{c:white} (note that $\Alg_{\P}(\M)$ is proper by \cite{rezk}), and so we may conclude that a map $f: A \lrar B$ of $\P$-algebras is a weak equivalence if and only if it induces an isomorphism of discrete algebras $P_0(A) \lrar P_0(B)$ and has a trivial relative cotangent complex, i.e., induces an isomorphism on Quillen cohomology with arbitrary coefficients. Since any operad in characteristic zero is homotopically sound (see \cite{hin-rectification}), the same result holds for any operad $\P$ in $\M$. 
\end{example}

\begin{example}[(Hurewicz theorem for simplicial categories)]\label{e:white-categories}
Let $\M = \Cat_{\Del}$ be the proper combinatorial model category of small simplicial categories. Given a fibrant simplicial category $\C$, let $P_n(\C) \in \Cat_\Del$ denote the ``homotopy $(n,1)$-category'' of $\C$ obtained by applying the functor $\cosk_n$ to every mapping object. Then $\C$ can be identified with the homotopy limit of the tower
$$  \cdots \lrar P_n(\C) \lrar \cdots \lrar P_2(\C) \lrar P_1(\C).$$
and the map $P_{n+1}(\C) \lrar P_{n}(\C)$ is a small extension for $n \geq 2$ (see~\cite[Proposition 3.2]{DKS86} and note Remark~\ref{r:rigid}). We remark that the above tower may differ from the Postnikov tower of the simplicial category $\C$, defined in terms of truncation. Since $\M$ is left proper and combinatorial we may consider the left Bousfield localization $L_2\M$ of $\M$ whose fibrant objects are the fibrant simplicial categories whose mapping objects are $1$-truncated (e.g., localize with respect to the maps $[1]_{S^n} \lrar [1]_{\ast}$ for $n \geq 2$, see Definition~\ref{d:1A}). Then we get that the adjunction $\M \adj L_2\M$ satisfies the conditions of Corollary~\ref{c:white}, and so we may conclude that a map $f: \C \lrar \D$ of simplicial categories is a weak equivalence if and only if it induces an isomorphism on the homotopy $(2,1)$-categories and an isomorphism on Quillen cohomology with arbitrary coefficients. 
\end{example}

\begin{rem}\label{r:spectral}
Let $\M$ be a left proper combinatorial model category and let
\begin{equation}\label{e:square-limit}
\vcenter{\xymatrix{
A \ar[r]\ar_{f}[d] & Y \ar[d] \\
B \ar@{-->}[ur] \ar^{g}[r] & X \\
}}
\end{equation}
be a square such that the map $Y \lrar X$ decomposes as a tower 
\begin{equation}\label{e:tower}
\cdots \lrar X_n \lrar \cdots \lrar X_0 = X
\end{equation}
where each $f_{n}: X_{n+1} \lrar X_n$ is a small extension with coefficients in $M_n \in \Sp(\M_{X_n//X_n})$. Then the space $Z$ of derived dotted lifts in~\ref{e:square-limit} can be written as an inverse homotopy limit $\holim_n Z_n$, where each $Z_n$ is the space of derived lifts of $A\lrar B$ against $X_n\lrar X$. The homotopy fibers of $Z_{n+1} \lrar Z_n$ are spaces of derived lifts in squares of the form
$$
\xymatrix{
A \ar[r]\ar_{f}[d] & X_{n+1} \ar[d] \\
B \ar@{-->}[ur] \ar^{g_n}[r] & X_n. \\
}$$
We may then use the Bousfield-Kan spectral sequence 
$$ E^{s,t}_r \Rightarrow \pi_{t-s}(Z) $$
to compute the homotopy groups of $Z$. The obstruction theory described above yields a description of the $E_1$-page of this spectral sequence as
$$ E^{s,t}_1 = \rH^{s-t}_Q(B,A;g_s^*M_s) .$$    
In particular, Quillen cohomology groups can be used as a computational tool to determine spaces of derived lifts (and in particular, mapping spaces in $\M$). For example, when the tower \eqref{e:tower} is one of the towers arising in Examples~\ref{e:white-spaces}, ~\ref{e:white-algebras} and~\ref{e:white-categories}, the spectrum objects $M_s$ are all Eilenberg-Maclane type spectra, constructed from the homotopy groups of the homotopy fibers of $Y\lrar X$. In that case, if the relative Quillen cohomology groups of $B$ over $A$ are non-trivial only in two consecutive degrees then the spectral sequence has no non-trivial differentials and collapses at the $E^1$-page. We will encounter such a phenomenon in Example~\ref{e:rezk}. 
\end{rem}

\section{Quillen cohomology of enriched categories}\label{s:enriched}
In this section we will turn our attention to the case where $\M$ is the model category $\Cat_{\bS}$ of categories enriched over a sufficiently nice symmetric monoidal (SM) model category $\bS$. We will begin in \S\ref{s:stabilization} where we will use the unstable comparison result of~\cite{part1} in order to identify the tangent category $\T_\C\Cat_{\bS}$ at a given fibrant enriched category $\C$ in terms of lifts against the canonical projection $\T\bS \lrar \bS$. In \S\ref{s:cotangent} we will use this identification to give explicit calculations of relative and absolute cotangent complexes. In particular, when $\bS$ is the category of chain complexes over a field, we obtain an identification of the associated Quillen cohomology in terms of the corresponding \textbf{Hochschild cohomology} of dg-categories, with a degree shift by one. Another notable example of interest is when $\bS$ is the category of simplicial sets, in which case $\Cat_{\bS}$ is a model for the theory of $\infty$-categories. We will discuss this example in detail in \S\ref{s:simp-categories}, where we will also obtain a simple description of the tangent category $\T_{\C}\Cat_{\infty}$ in terms of the \textbf{twisted arrow category} of $\C$. Examples and applications to classical problems such as detecting equivalences and splitting of homotopy idempotents are discussed at the end of that section.

\subsection{Stabilization of enriched categories}\label{s:stabilization}
Throughout this section let us fix a SM model category $\bS$ which is \textbf{excellent} in this sense of~\cite[Definition A.3.2.16]{Lur09} and such that every object in $\bS$ is cofibrant. Furthermore, we will fix the following two additional assumptions:
\begin{enumerate}
\item[(A1)]\label{i:diff}
$\bS$ is differentiable.
\item[(A2)]\label{i:small}
The unit object $1_{\bS}$ is \textbf{homotopy compact} in the sense that the functor $\pi_0\Map^{\der}_{\bS}(1_{\bS},-)$ sends filtered homotopy colimits to colimits of sets.
\end{enumerate}
By~\cite[Proposition A.3.2.4]{Lur09} (see also~\cite{BM13},~\cite{Mur15a}) there exists a combinatorial left proper model structure on the category $\Cat_{\bS}$ of $\bS$-enriched categories in which the weak equivalences are the Dwyer-Kan (DK) equivalences. Our goal in this section is to identify the tangent model category $\Sp((\Cat_{\bS})_{\C//\C})$ at a given $\bS$-enriched category $\C$.

\begin{rem}
In~\cite[Remark A.3.2.17]{Lur09} it is asserted that the axioms of an excellent model category imply that every object is cofibrant. However, it was observed elsewhere that this claim is not completely accurate, and so we have included this additional assumption explicitly. 
\end{rem}

\begin{rem}
Assumptions (A1) and (A2)
above are only used in order to establish that $\Cat_{\bS}$ is differentiable. The reader who so prefers can replace these two assumptions by a the single assumption that $\Cat_{\bS}$ is differentiable.
\end{rem}

\begin{define}\label{d:1A}
Let $\bS$ be as above. We will denote by $\ast \in \Cat_{\bS}$ the category with one object whose endomorphism object is $1_{\bS}$.
For $A \in \bS$ let $[1]_A \in \Cat_{\bS}$ denote the category with objects $0, 1$ and mapping spaces $\Map_{[1]_A}(0, 1)= A$, $\Map_{[1]_A}(1, 0)=\emptyset$ and $\Map_{[1]_A}(0, 0)=\Map_{[1]_A}(1, 1)=1_\bS$. 
\end{define}

\begin{rem}\label{r:gencofs}
The generating cofibrations for $\Cat_{\bS}$ are given by $\emptyset\lrar \ast$, as well as the maps
$
[1]_A\lrar [1]_B
$
where $A\lrar B$ is a generating cofibration of $\bS$. 
\end{rem}

\begin{define}[{cf.~\cite[1.5.4]{BM07}}]
Let $O$ be a set and let $W = O \times O$. Let $\P_{O}$ be the $W$-colored operad in $\bS$ whose algebras are the $\bS$-enriched categories with object set $O$ (also known as $(\bS,O)$-categories). More explicitly, $\P_{O}$ is the symmetrization of the non-symmetric operad in $\bS$ whose objects of $n$-ary operations are as follows: for $n \geq 1$ the object of $n$-ary operations from $(x_1,y_1),(x_2,y_2),...,(x_n,y_n)$ to $(x_\ast,y_\ast)$ is $1_{\bS}$ if $x_\ast=x_1,y_\ast=y_n$ and $y_i = x_{i+1}$ for $i=1,...,n-1$, and is initial otherwise. For $n=0$ the object of $0$-ary operations into $(x_\ast,y_\ast)$ is $1_{\bS}$ is $x_\ast=y_\ast$ and is initial otherwise.
\end{define}

\begin{rem}
Since $\P_{O}$ is the symmetrization of a non-symmetric operad it is automatically $\Sig$-cofibrant.
\end{rem}

By~\cite[Theorem 1.3]{Mur11} (see also erratum in~\cite{Mur15b}) the transferred model structure on $\Alg_{\P_O}$ exists. Considering $O$ as the discrete $\bS$-category with object set $O$ we obtain an adjunction
\begin{equation}\label{e:L}
\L:\Alg_{\P_O} \adj (\Cat_{\bS})_{O/}:\R
\end{equation}
where $\L$ interprets an $(\bS,O)$-category as an $\bS$-enriched category under $\O$, and $\R$ sends an $\bS$-enriched category $O \x{f}{\lrar} \C$ under $O$ to the $(\bS,O)$-category $\R(\C)$ whose mapping spaces are $\Map_{\R(\C)}(x, y)= \Map_{\C}(f(x),f(y))$. We now claim that $\L \dashv \R$ is a Quillen adjunction. Since $\L$ clearly preserves weak equivalences, it will suffices to check that $\L$ preserves cofibrations, or equivalently that $\R$ preserves trivial fibrations. The latter follows from the fact that a trivial fibration of enriched categories in the model structure of~\cite[Proposition A.3.2.4]{Lur09} always induces a trivial fibration on mapping objects.

\begin{rem}\label{r:left-proper}
Since $\L$ preserves and detects weak equivalences and $(\Cat_{\bS})_{O/}$ is left proper it follows that $\Alg_{\P_O}$ is left proper. In particular, $\P_O$ is stably admissible and the enveloping operad $\P_O^{\C}$ is stably admissible for any $\bS$-category with set of objects $O$.
\end{rem}
  
Since the unit of $\L \dashv \R$ is an isomorphism for every object and $\R$ preserves weak equivalences it follows that the derived unit of $\L \dashv \R$ is always a weak equivalence. In particular, the derived functor $\LL \L$ is derived fully-faithful. The following lemma identifies its essential image:

\begin{lem}
Let $f: O \lrar \C$ be an $\bS$-enriched category under $O$. Then the counit 
$ v_{\C}: \L(\R(\C)) \lrar \C $
(which is equivalent to the derived counit) is a weak equivalence if and only if $f$ is essentially surjective.
\end{lem}
\begin{proof}
It is clear that the counit $v_{\C}:\L(\R(\C)) \lrar \C$ is always fully-faithful. The result now follows from the fact that $v_{\C}$ has the same essential image as $f: O \lrar \C$.
\end{proof}

Now let $\C$ be an $\bS$-enriched category and let $O = \Ob(\C)$ be its set of objects. We may naturally consider $\C$ as a category under the discrete category $O$. The free-forgetful Quillen adjunction $(\Cat_{\bS})_{O/} \adj \Cat_{\bS}$ then induces an equivalence of categories 
\begin{equation}\label{e:over-under-2}
(\Cat_{\bS})_{\C/} \x{\simeq}{\adj} ((\Cat_{\bS})_{O/})_{\C/}
\end{equation}
which furthermore identifies the slice model structures on both sides. Let $\P_{\C} \x{\df}{=} \P^{\R(\C)}_{O}$ be the enveloping operad of $\R(\C)$. We will denote by
$$ \L^{\C}: \Alg_{\P_{\C}} \adj (\Cat_{\bS})_{\C/} :\R^{\C} $$
the adjunction induced by $\L\dashv \R$ after identifying $\Alg_{\P_{\C}} \simeq (\Alg_{\P_O})_{\R(\C)/}$ and
$(\Cat_{\bS})_{\C/} \simeq ((\Cat_{\bS})_{O/})_{\C/}$. 
We then obtain an induced Quillen adjunction
\begin{equation}\label{e:sliceScat}  \L^{\C}_{\aug}: \Alg_{\P_{\C}}^{\aug} \adj (\Cat_{\bS})_{\C//\C} :\R^{\C}_{\aug} \end{equation}
and hence an induced Quillen adjunction
$$\xymatrix{
\L^{\C}_{\Sp} := \Sp(\L^{\C}_{\aug}): \Sp(\Alg_{\P_{\C}}^{\aug}) \ar@<1ex>[r] & \Sp((\Cat_{\bS})_{\C//\C}) \ar@<1ex>[l]_-{\upvdash}}: \Sp(\R^{\C}_{\aug}) =: \R^{\C}_{\Sp}.
$$

\begin{pro}\label{p:cat}
Let $\C$ be a fibrant $\bS$-category. Then the Quillen adjunction
$$ \xymatrix{
\L^{\C}_{\Sp} : \Sp(\Alg_{\P_{\C}}^{\aug}) \ar@<1ex>^-{\simeq}[r] & \Sp((\Cat_{\bS})_{\C//\C}) \ar@<1ex>[l]_-{\upvdash}}:  \R^{\C}_{\Sp} $$
is a Quillen equivalence.
\end{pro}
The proof of Proposition~\ref{p:cat} will require knowing that $\Cat_{\bS}$ is differentiable, an issue we shall now address. Our first step is to verify the following:

\begin{lem}\label{l:finitary}
Weak equivalences in $\Cat_{\bS}$ are closed under sequential colimits.
\end{lem}
\begin{proof}
Let $\F: \NN \lrar \Cat_{\bS}$ be a functor and let $\C = \colim_{n \in \NN}\F(n)$. Then we have $\Ob(\C) \cong \colim_{n \in \NN}\Ob(\F(n))$ and for each $x,y \in \Ob(\C)$ there exists a minimal $n_0 \in \NN$ such that both $x$ and $y$ are in the image of $\F(n_0) \lrar \C$.  If we now choose $x_{n_0},y_{n_0} \in \F(n_0)$ whose images in $\C$ are $x$ and $y$ respectively then we have
\begin{equation}\label{e:map-colim}
\Map_{\C}(x,y) = \colim_{n \geq n_0} \Map_{\F(n)}(x_{n},y_{n})
\end{equation}
where for $n \geq n_0$ we have denoted by $x_{n},y_{n}$ the images of $x_{n_0},y_{n_0}$ respectively under the map $\F(n_0) \lrar \F(n)$. Let us now consider a levelwise weak equivalence $\vphi:\F \lrar \F'$ of functors $\NN \lrar \Cat_{\bS}$ and let $f:\C \lrar \C'$ be the induced map on colimits. We need to show that $f$ is a weak equivalence, i.e., essentially surjective and homotopy fully-faithful. The fact that $f$ is homotopy fully-faithful follows from formula~\ref{e:map-colim} since each $\vphi_n: \F(n) \lrar \F'(n)$ is homotopy fully-faithful and weak equivalences in $\bS$ are closed under sequential colimits (since $\bS$ is excellent). Let us now prove that $f$ is essentially surjective. Let $x' \in \C'$ be an object. Then there exists an $n \in \NN$ and an object $x'_n \in \F'(n)$ whose image in $\C$ is $x'$. Since $\vphi_n: \F(n) \lrar \F'(n)$ is essentially surjective there exists an object $x_n \in \F(n)$ such that $\vphi_n(x_n)$ is isomorphic to $x'_n$ in $\Ho(\F'(n))$. It follows that the image of $\vphi_n(x_n)$ in $\Ho(\C)$ is isomorphic to the image of $x'_n$, and hence to the image of $x'$, as desired.
\end{proof}

\begin{lem}\label{l:local-pb}
Let
\begin{equation}\label{e:square-cat} \vcenter{
\xymatrix{
\A \ar[r]\ar[d] & \B \ar^{\psi}[d] \\
\C \ar^{\vphi}[r] & \D \\
}}
\end{equation}
be a square in $\Cat_{\bS}$. Then~\ref{e:square-cat} is homotopy Cartesian if and only if it is weakly equivalent to a square which satisfies the following three properties:
\begin{enumerate}[(1)]
\item
For every $x,y \in \A$ the induced diagram on mapping spaces in homotopy Cartesian in $\bS$. 
\item
The induced square of object sets is Cartesian.
\item
Both $\Ho(\psi):\Ho(\B) \lrar \Ho(\D)$ and $\Ho(\vphi):\Ho(\C) \lrar \Ho(\D)$ are iso-fibrations.
\end{enumerate}
\end{lem}
\begin{proof}
We first prove the only if direction. Observe that any homotopy Cartesian square satisfies (1). If a square as in~\eqref{e:square-cat} is homotopy Cartesian then, up to weak equivalence, we may assume that $\vphi$ and $\psi$ are fibrations with fibrant codomain and that $\A$ is the actual pullback, in which case the square satisfies (2) and (3) by~\cite[Theorem A.3.2.24]{Lur09}.

We now prove the if direction. Every map $f: \C \lrar \D$ in $\Cat_{\bS}$ can be functorially factored as $\C \x{f'}{\lrar} \C' \x{f''}{\lrar} \D$ where $f'$ is a weak equivalence and an isomorphism on object sets and $f''$ induces fibrations on mapping objects. In order to obtain such a factorization, apply the small object argument with respect to the maps $[1]_A\lrar [1]_B$, where $A\lrar B$ is a generating trivial cofibration in $\bS$. This factors $f$ as a transfinite composition of pushouts of the trivial cofibrations $[1]_A\lrar [1]_B$ (which are identities on objects), followed by a map which induces fibrations on mapping objects. In particular, given a square as in~\eqref{e:square-cat} which satisfies (1)-(3) above we may replace it with a weakly equivalent square which satisfies (1)-(3) and such that $\D$ is furthermore fibrant in $\Cat_{\bS}$. We may therefore assume without loss of generality that $\D$ is fibrant.

We now observe that if $f: \C \lrar \D$ is a map such that $\Ho(f)$ is an iso-fibration and $\D$ is fibrant then we may functorially factor $f$ as $\C \x{f'}{\lrar} \C' \x{f''}{\lrar} \D$ where $f'$ is a weak equivalence and an isomorphism on object sets and $f''$ is a local fibration in the sense of~\cite[A.3.2.9]{Lur09} and hence a fibration in $\Cat_{\bS}$ by~\cite[Theorem A.3.2.24]{Lur09}. A square ~\ref{e:square-cat} satisfying (1)-(3) can therefore be replaced by a weakly equivalent square in which both $\vphi$ and $\psi$ are fibrations, without changing the objects of any of the categories appearing in~\eqref{e:square-cat}. The new square now satisfies (1)-(3) as well. We may thus assume without loss of generality that $\vphi$ and $\psi$ are already fibrations, in which case it will suffice to show that the induced map $\A \lrar \C \times_{\D} \B$ is a weak equivalence. But this map is homotopy fully-faithful and induces an isomorphism on objects by assumption, hence it is a weak equivalence.

\end{proof}

\begin{cor}\label{c:diff}
$\Cat_{\bS}$ is differentiable as soon as $\bS$ satisfies Assumptions (A1) and (A2) 
above.
\end{cor}
\begin{proof}
In light of Lemma~\ref{l:finitary} it will suffice to show that the (relative) functor $\colim_{\NN}: (\Cat_{\bS})^{\NN} \lrar \Cat_{\bS}$ preserves  homotopy pullbacks and homotopy terminal objects. We first note that an $\bS$-enriched category $\C$ is homotopy terminal if and only if $\Map_{\C}(x,y)$ is homotopy terminal in $\bS$ for every $x,y \in \C$. Since $\bS$ is differentiable and weak equivalences in $\bS$ are closed under sequential colimits it follows that homotopy terminal objects are closed under sequential colimits. The same statement thus holds for $\Cat_{\bS}$ in view of formula~\ref{e:map-colim}.

We now need to show that $\colim_{\NN}$ maps homotopy Cartesian diagrams to homotopy Cartesian diagrams. By Lemma~\ref{l:local-pb} it will suffice to show that $\colim_{\NN}$ preserves those squares satisfying properties (1)-(3) of Lemma~\ref{l:local-pb}. Property (1) follows from formula~\ref{e:map-colim} and our assumption that $\bS$ is differentiable. Concerning property (2) of Lemma~\ref{l:local-pb}, preservation by $\colim_\NN$ follows from the fact that the functor $\C \mapsto \Ob(\C)$ preserves colimits and that sequential colimits commute with pullbacks in the category of sets. For property (3) we first observe that the homotopy category functor $\Ho: \Cat_{\bS} \lrar \Cat$ preserves sequential colimits, since $1_{\bS}$ is assumed to be homotopy compact. It will hence suffice to show that the class of iso-fibrations in $\Cat$ is closed under sequential colimits. Now note that the property of being an iso-fibration is simply the right lifting property with respect to the map $\eta:\ast \lrar \E$, where $\E$ is the category with two objects in which all hom sets are singletons. The desired result then follows from the fact that both $\ast$ and $\E$ are compact in $\Cat$.
\end{proof} 
Having proven that $\Cat_{\bS}$ is differentiable, we will prove Proposition~\ref{p:cat} using~\cite[Corollary 2.4.9]{part0}. To apply the latter corollary we need to verify that the derived counit of the induced Quillen adjunction $\L^{\C}_{\aug} : \Alg_{\P_{\C}}^{\aug} \adj (\Cat_{\bS})_{\C//\C} : \R^{\C}_{\aug}$ becomes an equivalence after looping finitely many times.
\begin{lem}\label{l:pullback}
Let $\C$ be a fibrant $\bS$-category and let $\C \x{f}{\lrar} \D \x{g}{\lrar} \C$ be an $\bS$-category over-under $\C$ such that $g$ is a fibration in $\Cat_{\bS}$. Let $\L^{\C}(\R^{\C}(\D)) \lrar \D$ be the counit map. Then the map of derived pullbacks
$$ \C \times^h_{\L^{\C}(\R^{\C}(\D))}\C \lrar \C \times^h_{\D} \C $$
is a weak equivalence in $(\Cat_{\bS})_{\C/}$. In particular, the induced map $\Om\L^{\C}_{\aug}(\R^{\C}_{\aug}(\D)) \lrar \Om\D$ is a weak equivalence in $(\Cat_{\bS})_{\C//\C}$.
\end{lem}
\begin{proof}
Let $\C \x{i}{\lrar} \wtl{\C} \x{p}{\lrar} \D$ be a factorization of $f$ into a trivial cofibration followed by a fibration and let $f': \C =\L^{\C}(\R^{\C}(\C))\lrar \L^{\C}(\R^{\C}(\D))$ be the structure map of $\L^{\C}(\R^{\C}(\D))$. We note that $\L^{\C}(\R^{\C}(\D)) \subseteq \D$ is nothing but the full subcategory of $\D$ spanned by the image of $f$ and that $f'$ is just the map $f$ with its codomain restricted. Let $\wtl{\C}'$ be the $\bS$-category sitting in the pullback square
$$ \xymatrix{
\wtl{\C}' \ar^{p'}[r]\ar[d]\pbc & \L^{\C}(\R^{\C}(\D)) \ar^{v}[d] \\
\wtl{\C} \ar^{p}[r] & \D. \\
}$$
Since the right vertical map is (strictly) fully-faithful it follows that the left vertical map is fully-faithful, and since the essentially surjective map $i: \C \lrar \wtl{\C}$ factors as $\C \x{i'}{\lrar} \wtl{\C}' \lrar \wtl{\C}$ it follows that $\wtl{\C}' \lrar \wtl{\C}$ is essentially surjective and hence a weak equivalence of $\bS$-categories. 
Now consider the extended diagram
$$ \xymatrix{
\wtl{\C}'' \ar[d]\ar[r] \pbc & \wtl{\C}' \ar^{p'}[d]\ar^{\simeq}[r] & \wtl{\C} \ar[d] \\
\wtl{\C}' \ar^-{p'}[r]\ar^{\simeq}[d] & \L^\C(\R^\C(\D)) \ar^{v}[d] \ar@{=}[r] & \L^\C(\R^\C(\D)) \ar^{v}[d] \\
\wtl{\C} \ar^{p}[r] & \D \ar@{=}[r] & \D \\
}$$
where the upper square is a pullback square. By our assumption $\D$ is fibrant and hence by~\cite[Theorem A.3.2.24]{Lur09} we get that $\L^{\C}(\R^\C(\D))$ is fibrant as well. Since $p$ and $p'$ are fibrations between fibrant objects it now follows that the top left square and the bottom left square are also homotopy Cartesian. Since the two right squares are homotopy Cartesian we may conclude that the external square is homotopy Cartesian, and so the desired result follows.
\end{proof}

We have now gathered enough tools to establish Proposition~\ref{p:cat}:
\begin{proof}[{Proof of Proposition~\ref{p:cat}}]
Apply~\cite[Corollary 2.4.9]{part0}. The required conditions are satisfied in light of Corollary~\ref{c:diff} and Lemma~\ref{l:pullback}.
\end{proof}

Having established Proposition~\ref{p:cat} we are now in a position to use the comparison of~\cite[Theorem 4.1.1]{part1} in order to compute the stabilization of $\Cat_{\C//\C}$ via the stabilization of $\Alg^{\P_\C}_{\aug}$. For this it will be useful to recall the tensor product of $\bS$-categories $\otimes: \Cat_{\bS} \times \Cat_{\bS} \lrar \Cat_{\bS}$. By definition, if $\D$ and $\E$ are $\bS$-categories then $\D \otimes \E$ is the $\bS$-category whose object set is the Cartesian product $\Ob(\D \otimes \E) \x{\df}{=} \Ob(\D) \times \Ob(\E)$ and such that for every $d,d' \in \D$ and $e,e'\in \E$ we have
$$ \Map_{\D \otimes \E}((d,e),(d',e')) \x{\df}{=} \Map_{\D}(d,e) \otimes \Map_{\D}(d',e') $$
Since $\P_\C$ is the symmetrization of a nonsymmetric operad, it is $\Sig$-cofibrant (as well as stably admissible, by Remark \ref{r:left-proper}). We therefore obtain from~\cite[Theorem 4.1.1]{part1} that $\Sp(\Alg^{\P_{\C}}_{\aug})$ is Quillen equivalent to $\Sp(\Alg^{(\P_{\C})_{\leq 1}}_{\aug})$. 

Unwinding the definitions, we see that the $\bS$-enriched category $(\P_{\C})_1$ (i.e.\ the enveloping category of $\C$) is the tensor product $\C^{\op} \otimes \C$, and that the $\C$-module $\C^{\op} \otimes \C \lrar \bS$ associated to $(\P_{\C})_0$ is just the mapping space functor $\Map_{\C}: \C^{\op} \otimes \C \lrar \bS$. Composing adjunction \eqref{e:sliceScat} with the augmented free-forgetful adjunction associated to $(\P_\C)_{\leq 1}\lrar \P_{\C}$, we obtain an adjunction
\begin{equation}\label{e:adjredcats}\vcenter{\xymatrix{
\F^{\C}_{\aug} : \Alg_{(\P_\C)_{\leq 1}}^{\aug}=\Fun(\C^{\op} \otimes \C,\bS)_{\Map_\C//\Map_\C} \ar@<1ex>[r] & (\Cat_{\bS})_{\C//\C} \ar@<1ex>[l]_-{\upvdash}:  \G^{\C}_{\aug}
}}\end{equation}
where $\G^{\C}_{\aug}$ sends a category $\C \x{f}{\lrar} \D \x{g}{\lrar} \C$ to the functor $\G^{\C}_{\aug}(\D)(x,y) = \Map_{\D}(f(x),f(y))$. From the above considerations, we can thus conclude the following:

\begin{thm}\label{t:comp-cat}
Let $\bS$ be as above and let $\C$ be a fibrant $\bS$-enriched category. Then the adjunction
$$ \xymatrix{
\F^{\C}_{\Sp} : \T_{\Map_\C}\Fun(\C^{\op} \otimes \C,\bS) \ar@<1ex>^-{\simeq}[r] & \T_\C\Cat_{\bS} \ar@<1ex>[l]_-{\upvdash}}:  \G^{\C}_{\Sp} $$
induced by~\eqref{e:adjredcats} is a Quillen equivalence. 
\end{thm}

\begin{rem}\label{r:DK}
When $\bS$ is the category of simplicial sets endowed with the Kan-Quillen model structure, a variant of Corollary~\ref{t:comp-cat} was described by Dwyer and Kan in~\cite[Proposition 6.3]{DK88} (without an explicit proof). This variant pertained to simplicial categories with a \textbf{fixed set of objects}, and considered the corresponding category of abelian group objects, rather than stabilization. We note that when the set of objects is not fixed, the analogue of the above theorem for abelian group objects is \textbf{false}. For example, if $\C$ is a simplicial category and $A$ is an abelian group then $A \times \C$ is naturally an abelian group object in $(\Cat_{\bS})_{/\C}$, but the associated functor $\C^{\op} \times \C \lrar \bS$ is the trivial abelian group object of $\Fun(\C^{\op} \otimes \C,\bS)_{/\Map_{\C}}$. One may consider this as an additional motivation to work with spectrum objects as opposed to abelian group objects.
\end{rem}

Since every object in the enriching model category $\bS$ is cofibrant, we have that $\bS$ is tractable. By~\cite[Corollary 3.2.2]{part0} we may conclude that the tangent model category $\T\bS$ is tensored and cotensored over $\bS$. Furthermore, by Remark~\ref{r:global-tensored} we may identify $\T_{\Map_{\C}}\Fun(\C^{\op} \otimes \C,\bS)$ with the category of $\bS$-enriched lifts
$$ \xymatrix{
& \T\bS \ar^{\pi}[d] \\
\C^{\op} \otimes \C \ar@{-->}[ur] \ar_-{\Map_{\C}}[r] & \bS \\
}$$
endowed with the projective model structure. To sum up, Theorem~\ref{t:comp-cat} can also be read as follows:
\begin{cor}\label{c:comp-lifts}
Let $\bS$ be as above and let $\C$ be a fibrant $\bS$-category. Then the tangent model category $\T_{\C}\Cat_{\bS}$ is Quillen equivalent to the model category $\Fun^{\bS}_{/\bS}(\C^{\op} \otimes \C,\T\bS)$ consisting of $\bS$-enriched functors $\C^{\op} \otimes \C \lrar \T\bS$ which sit above the mapping space functor $\Map_{\C}:\C^{\op} \otimes \C \lrar \bS$.
\end{cor}

Let us now consider the special case where $\bS$ is \textbf{stable}. 
Combining Corollary~\ref{c:stable-functor-tangent} and Theorem~\ref{t:comp-cat} we obtain the following:
\begin{cor}\label{c:comp-stable}
Let $\bS$ be as above and assume in addition that $\bS$ is stable. Then for every fibrant $\bS$-enriched category $\C$ the tangent model category $\T_{\C}\Cat_{\bS}$ is naturally Quillen equivalent to the functor category $\Fun(\C^{\op} \otimes \C,\bS)$.
\end{cor}

\begin{rem}\label{r:bimodule}
A prime example of Corollary~\ref{c:comp-stable} is when $\bS$ is the category of chain complexes over a field, in which case $\Cat_{\bS}$ is the category of \textbf{dg-categories}. 
In this case we may phrase the above computation as follows: the tangent model category $\T_{\C}\Cat_{\bS}$ is Quillen equivalent to the category of \textbf{$\C$-bimodules}, also known as correspondences from $\C$ to itself.
\end{rem}

\subsection{The cotangent complex of enriched categories}\label{s:cotangent}

Throughout this section, we will assume that $\bS$ is an excellent model category with only cofibrant objects satisfying conditions (A1) and (A2) from the beginning of Section \ref{s:enriched}. Our goal in this section is to compute the cotangent complex 
$$ L_\C:=\LL\Sig^{\infty}_+(\C) \in \Sp((\Cat_{\bS})_{\C//\C}) ,$$ 
of an $\bS$-enriched category $\C$, or, more precisely, its image under the right Quillen equivalence $\G^{\C}_{\Sp}: \T_\C\Cat_{\bS} \x{\simeq}{\lrar} \T_{\Map_{\C}}\Fun(\C^{\op} \otimes \C,\bS)$ of Theorem~\ref{t:comp-cat}. 
\begin{pro}\label{p:conceptual}
Let $\C$ be a fibrant $\bS$-enriched category. Then there is a natural weak equivalence
$$ \theta_{\C}: L_{\Map_{\C}}[-1] \x{\simeq}{\lrar} \RR\G^{\C}_{\Sp}(L_{\C})  $$
in the model category $\T_{\Map_{\C}}\Fun(\C^{\op} \otimes \C,\bS)$. In other words, under the equivalence of Theorem \ref{t:comp-cat} we may identify the cotangent complex of $\C$ with the desuspension of the cotangent complex of $\Map_C \in \Fun(\C^{\op} \otimes \C,\bS)_{/\Map_{\C}}$.
\end{pro}

\begin{cor}\label{c:conceptual}
Under the equivalence of Corollary~\ref{c:comp-lifts}, the image of the cotangent complex $L_{\C} \in \T_\C\Cat_{\bS}$ in $\Fun^{\bS}_{/\bS}(\C^{\op} \otimes \C,\T\bS)$ is weakly equivalent to the desuspension of the composite functor 
$$ \C^{\op} \otimes \C \x{\Map_{\C}}{\lrar} \bS \x{\Sig^{\infty}_{\int}}{\lrar} \T\bS .$$
\end{cor}
\begin{proof}
Combine Proposition~\ref{p:conceptual} with Corollary~\ref{c:comp-lifts} and Remark~\ref{r:functor-cotangent}.
\end{proof}

Before proceeding to the proof of Proposition \ref{p:conceptual}, let us point out a few notable cases.
\begin{cor}\label{c:simp}
Let $\bS$ be the category of simplicial sets endowed with the Kan-Quillen model structure and let $\C$ be a fibrant simplicial category. Then the image of the cotangent complex $L_{\C} \in \T_{\C}\Cat_{\bS}$ in $\Fun^{\bS}_{/\bS}(\C^{\op} \times \C,\T\bS)$ is the functor $\C^{\op} \times \C \lrar \T \SS$ which associates to each $(x,y) \in \C^{\op} \times \C$ the parametrized spectrum over $\Map_{\C}(x,y)$ which is constant with value $\mathbb{S}[-1]$.
\end{cor}

When $\bS$ is stable, Corollary~\ref{c:comp-stable} identifies the tangent category $\T_{\C}\Cat_{\bS}$ with the category of functors $\C^{\op} \otimes \C \lrar \bS$. In this case, Proposition~\ref{p:conceptual} can be combined with Corollary~\ref{c:stable-functor-tangent} to give the following:
\begin{cor}\label{c:cotangent-stable}
Let $\bS$ and $\C$ be as above and assume in addition that $\bS$ is stable. Then the functor $\C^{\op} \otimes \C \lrar \bS$ associated to the cotangent complex $L_\C$ is the functor $(x,y) \mapsto \Map_{\C}(x,y)[-1]$.
\end{cor}

A special case of interest is when $\bS$ is the category of chain complexes over a field, in which case $\Cat_{\bS}$ is the model category of dg-categories:
\begin{cor}\label{c:hochschild}
The Quillen cohomology $\rH_Q^\bullet(\C, \F)$ of a dg-category $\C$ with coefficients in a bimodule $\F$ can be identified with the corresponding (shifted) \textbf{Hochschild cohomology} $\rH\rH^{\bullet+1}(\C, \F)$ (defined via the bar complex, see e.g.~\cite[5.4]{Kel}).
\end{cor}

The proof of Proposition~\ref{p:conceptual} will require a few preliminaries. Let $\C$ be a fibrant $\bS$-category. Since the adjunction $\F^{\C}_{\Sp} \dashv \G^{\C}_{\Sp}$ of Theorem~\ref{t:comp-cat} is a Quillen equivalence it will suffice to construct a weak equivalence 
\begin{equation}\label{e:other-side}
\LL\F^{\C}_{\Sp}(L_{\Map_C}) \x{\simeq}{\lrar} L_{\C}[1] .
\end{equation}
To do this, we will make use of the tensor product of $\bS$-categories $\otimes: \Cat_{\bS} \times \Cat_{\bS} \lrar \Cat_{\bS}$. Recall that while $\otimes$ is a close symmetric monoidal product on $\Cat_{\bS}$, it is \textbf{not} compatible with the model structure. However, since every object in $\bS$ is cofibrant, the functor $(-) \otimes \C: \Cat_{\bS} \lrar \Cat_{\bS}$ does preserve weak equivalences and by~\cite[Theorem A.3.5.14]{Lur09} it also preserve homotopy colimits. The unit of $\otimes$ is the category $\ast \in \Cat_{\bS}$ which has a single object whose endomorphism object is $1_{\bS}$. To avoid confusion, we warn the reader that $\ast$ is generally not the terminal object of $\Cat_{\bS}$, unless $1_{\bS}$ is terminal in $\bS$.  
For any $\bS$-category $\C$, the functor
\begin{equation}\label{e:times-sp}
(-) \otimes \C : \Sp((\Cat_{\bS})_{\ast//\ast}) \lrar \Sp((\Cat_{\bS})_{\C//\C}) 
\end{equation}
sending $\ast \lrar X_{\bullet\bullet} \lrar \ast$ to $\C \lrar X_{\bullet\bullet} \otimes \C \lrar \C$ preserves levelwise weak equivalences and suspension spectra. The functor~\ref{e:times-sp} also preserves the following slightly more general class of equivalences:
\begin{define}\label{d:suspension}
For any model category $\M$, a \textbf{strong equivalence of suspension spectra} is a map $f: X \lrar Y$ in $\Sp(\M)$ between suspension spectra such that $f_{n,n}: X_{n,n} \lrar Y_{n,n}$ is a weak equivalence for all $n>>0$.
\end{define}
Now let $\ovl{L}_{\ast} = \ovl{\Sig}^{\infty}(\ast \coprod \ast) \simeq L_\ast \in \Sp((\Cat_{\bS})_{\ast//\ast})$ be a suspension spectrum model for the cotangent complex of $\ast$ such that $(\ovl{L}_\ast)_{0,0} = \ast \coprod \ast$ (see~\cite[Corollary 2.3.3]{part0}). Since $\ovl{L}_\ast\otimes \C$ is a suspension spectrum with $\C\coprod \C$ in degree $(0, 0)$, the adjoint map $L_\C=\Sig^\infty(\C\coprod \C)\lrar \ovl{L}_\ast\otimes \C$ is a stable weak equivalence by~\cite[Lemma 2.3.2]{part0}. We may therefore use $\ovl{L}_\ast \otimes \C$ as a model for $L_\C$. 

We shall now show that the spectrum $\LL\F^{\C}_{\Sp}(L_{\Map_C})$ appearing on~\eqref{e:other-side} can also be obtained from a suitable spectrum object in $(\Cat_{\bS})_{\ast//\ast}$ by tensoring with $\C$. To do this we will first find another way to describe the functor $\F^{\C}_{\aug}: \Fun(\C^{\op} \otimes \C,\bS)_{\Map_\C//\Map_{\C}}\lrar (\Cat_\bS)_{\C//\C}$ from which $\F^{\C}_{\Sp}$ is induced upon passing to spectrum objects. Let $[1]_{\bS} := [1]_{1_{\bS}}$ be as in Definition~\ref{d:1A}. The natural map $\ast \coprod \ast \lrar [1]_{\bS}$, which can be identified with $[1]_{\emptyset_{\bS}} \lrar [1]_{1_{\bS}}$, is a cofibration in $\Cat_{\bS}$, and in particular $[1]_{\bS}$ is cofibrant. 
Consider the Quillen adjunction
$$ \lam:\Fun(\C^{\op} \otimes \C,\bS)_{\Map_{\C}/} \adj (\Cat_{\bS})_{[1]_{\bS} \otimes \C/}: \rho $$
defined as follows. If $f: \C^{\op} \otimes \C \lrar \bS$ is a functor under $\Map_{\C}$ then $\lam(f)$ is the $\bS$-enriched category whose set of objects is $\{0, 1\}\times\Ob(\C)$ and whose mapping spaces are given by 
$$ \lam(f)((i,x),(j,y)) = \left\{\begin{matrix} \Map_{\C}(x,y) & i=j \\ f(x,y) & (i,j) = (0,1) \\ \emptyset_{\bS} & (i,j) = (1,0) \\ \end{matrix}\right. $$%
Composition of morphisms is defined using the functoriality of $f$ in $\C^{\op}\otimes\C$. This construction sends $\Map_\C$ to $ [1]_{\bS} \otimes \C$ and $\lam$ therefore sends a functor under $\Map_\C$ to a functor under $[1]_{\bS} \otimes \C$. In the other direction, if $[1]_{\bS} \otimes \C \x{\iota}{\lrar} \D$ is an object of $(\Cat_{\bS})_{[1]_{\bS} \otimes \C/}$ then $\rho(\iota)$ is given by
$$ \rho(\iota)(x,y) = \Map_{\D}(\iota(0,x),\iota(1,y))$$
which admits a natural map from $\Map_{\C}$. It follows immediately from this description that $\rho$ is a right Quillen functor.

Now consider the induced adjunction on augmented objects
$$ \lam_{\aug}:\Fun(\C^{\op} \otimes \C,\bS)_{\Map_{\C}//\Map_{\C}} \adj (\Cat_{\bS})_{[1]_{\bS} \otimes \C//[1]_{\bS} \otimes \C}: \rho_{\aug}. $$
Since both $\lam$ and $\rho$ preserve initial objects, the formulas for $\lam_{\aug}$ and $\rho_{\aug}$ are the same as those for $\lam$ and $\rho$. It follows that both $\lam_{\aug}$ and $\rho_{\aug}$ preserve weak equivalences between arbitrary objects, and thus can be applied without deriving.

To get from enriched categories over-under $[1]_{\bS}\otimes\C$ to enriched categories over-under $\C$, note that the map $[1]_{\bS} \lrar \ast$ induces a map $p: [1]_{\bS} \otimes \C  \lrar \C$, associated to which is an adjunction
$$ p^{\aug}_!: (\Cat_{\bS})_{[1]_{\bS} \otimes \C//[1]_{\bS} \otimes \C} \adj (\Cat_{\bS})_{\C//\C} :p^*_{\aug} .$$
Here the left adjoint $p^{\aug}_!$ sends an $\bS$-category $\D$ over-under $[1]_{\bS} \otimes \C$ to the $\bS$-category $\C \coprod_{[1]_{\bS} \otimes \C} \D$ over-under $\C$, and the right adjoint $p^*_{\aug}$ sends an $\bS$-category $\D$ over-under $\C$ to the $\bS$-category $([1]_{\bS} \otimes \C) \times_{\C} \D$ over-under $[1]_{\bS} \otimes \C$. The category $([1]_{\bS} \otimes \C) \times_{\C} \D$ has object set $\{0,1\} \times \Ob(\D)$ and for $x,y \in \C$ whose images in $\D$ are $x',y'$ respectively we have 
\begin{align*}
\Map_{([1]_{\bS} \otimes \C) \times_{\C} \D}((0,x'),(1,y')) & = \Map_{[1]_{\bS} \otimes \C}((0,x),(1,y)) \times_{\Map_{\C}(x,y)} \Map_\D(x',y')\\
&=\Map_\D(x',y') . 
\end{align*}
In particular, the value of the composite $\rho_{\aug}(p^*_{\aug}(\C \x{\iota}{\lrar} \D \lrar \C)) \in \Fun(\C^{\op} \otimes \C,\bS)_{\Map_{\C}//\Map_{\C}}$ is naturally isomorphic to $(x,y) \mapsto \Map_{\D}(\iota(x),\iota(y))$ (as functors over-under $\Map_\C$). In other words, the diagram of right Quillen functors
$$ \xymatrix{
\Fun(\C^{\op} \otimes \C,\bS)_{\Map_{\C}//\Map_{\C}} && (\Cat_{\bS})_{\C//\C} \ar_-{\G^{\C}_{\aug}}[ll]\ar^{p^*_{\aug}}[dl] \\
& (\Cat_{\bS})_{[1]_{\bS} \otimes \C//[1]_{\bS} \otimes \C} \ar^-{\rho_{\aug}}[ul] &\\
}$$
commutes up to a natural isomorphism. It follows that the corresponding diagram of left adjoints commutes as well, i.e., we can write our functor $\F^{\C}_{\aug}$ as a composition  $\F^{\C}_{\aug} \cong p^{\aug}_!\lam_{\aug}$. In particular, for every functor $f: \C^{\op} \otimes \C \lrar \bS$ over-under $\Map_\C$ we have that $\LL\F^{\C}_{\aug}(f) \simeq \LL p^{\aug}_!\lam_{\aug}(M)$. 

Having identified the left Quillen functor $\F^{\C}_{\aug}$ in these terms, we can express the left hand term $\LL\F^{\C}_{\Sp}(L_{\Map_\C})$ of \eqref{e:other-side} as follows:
\begin{lem}
There is an equivalence in $\Sp((\Cat_{\bS})_{\C//\C})$ of the form
$$
\LL\F^{\C}_{\Sp}(L_{\Map_\C}) \simeq \Sig^{\infty}\left(\ast \coprod_{[1]_{\bS}} \A\right) \otimes \C
$$
where $\A = [1]_{1_{\bS} \coprod 1_{\bS}}$ is considered as a category over-under $[1]_{\bS}$,  and $\Sig^{\infty}\left(\ast \coprod_{[1]_{\bS}} \A\right)$ is considered as a spectrum object in $(\Cat_{\bS})_{\ast//\ast})$.
\end{lem}
\begin{proof}
Before taking suspension spectra, we may compute at the level of augmented objects:
\begin{align*}
\LL\F^{\C}_{\aug}(\Map_{\C} \coprod \Map_{\C}) & \simeq \LL p^{\aug}_!\lam_{\aug}(\Map_{\C} \coprod \Map_{\C})\\
& \simeq \LL p^{\aug}_!(\A \otimes \C) \simeq \C \coprod^h_{[1]_{\bS} \otimes \C} \left[\A \otimes \C\right] \simeq \left[\ast \coprod_{[1]_{\bS}} \A\right] \otimes \C
\end{align*}
The last equivalence follows from the preservation of homotopy colimits by $(-) \otimes \C$ and the fact that $[1]_{\bS} \lrar \A$ is a cofibration. Since left Quillen functors commute with taking suspension spectra, we conclude that
\begin{align}\label{e:times-2}
\LL\F^{\C}_{\Sp}(L_{\Map_\C}) & \simeq \LL\F^{\C}_{\Sp}(\Sig^{\infty}(\Map_{\C} \coprod \Map_{\C}))\nonumber \\
& \simeq \Sig^{\infty}(\LL\F^{\C}_{\aug}(\Map_{\C} \coprod \Map_{\C})) \simeq \Sig^{\infty}\left(\left[\ast \coprod_{[1]_{\bS}} \A\right] \otimes \C\right) \simeq \Sig^{\infty}\left(\ast \coprod_{[1]_{\bS}} \A\right) \otimes \C
\end{align}
as asserted.
\end{proof}
Rather than working with $\Sig^{\infty}\left(\ast \coprod_{[1]_{\bS}} \A\right)$, we may apply~\cite[Corollary 2.3.3]{part0} and consider a suspension spectrum model $\ovl{\Sig}^{\infty}\left(\ast \coprod_{[1]_{\bS}} \A\right) \in \Sp(\Cat_{\ast//\ast})$ whose degree $(0,0)$ object is the cofibrant object $\ast \coprod_{[1]_{\bS}} \A$. As before, such a choice of a suspension spectrum model induces a stable weak equivalence $\Sig^{\infty}\left(\left[\ast \coprod_{[1]_{\bS}} \A\right] \otimes \C\right) \lrar \ovl{\Sig}^{\infty}\left(\ast \coprod_{[1]_{\bS}} \A\right) \otimes \C$.
In particular, we have a stable equivalence
$$ \ovl{\Sig}^{\infty}\left(\ast \coprod_{[1]_{\bS}} \A\right) \otimes \C \simeq \LL\F^{\C}_{\Sp}(L_{\Map_\C}). $$
We are now in the position to prove Proposition~\ref{p:conceptual}.
\begin{proof}[{Proof of Proposition~\ref{p:conceptual}}]
Having written both sides of~\eqref{e:other-side} as the image of a spectrum object in $(\Cat_{\bS})_{\ast//\ast}$ under tensoring with $\C$, we see that in order to construct the natural equivalence of the form~\eqref{e:other-side} it will suffice to construct a natural strong equivalence of suspension spectra (see Definition~\ref{d:suspension})
$$ \ovl{\Sig}^{\infty}\left(\ast \coprod_{[1]_{\bS}} \A\right) \x{\simeq}{\lrar} \ovl{L}_\ast[1] .$$
As $\A \cong [1]_{\bS} \coprod_{\ast \coprod \ast} [1]_{\bS}$ we get that
$$ \ast \coprod_{[1]_{\bS}} \A = \ast \coprod_{[1]_{\bS}} [1]_{\bS} \coprod_{\ast \coprod \ast} [1]_{\bS} \cong \ast \coprod_{\ast \coprod \ast} [1]_{\bS} .$$
Let $[1]^\sim_{\bS} \in \Cat_{\bS}$ be the $\bS$-enriched category with objects $0$ and $1$ and all mapping spaces equal to $1_{\bS} \in \bS$. Let $[1]_{\bS} \x{\eta}{\hrar} \E \x{\simeq}{\twoheadrightarrow} [1]^\sim_{\bS}$ be a factorization of the natural map $[1]_{\bS} \lrar [1]^\sim_{\bS}$ into a cofibration followed by a trivial fibration. Since the map $\ast \coprod \ast \lrar [1]_{\bS}$ is a cofibration we get that the map $\ast \coprod \ast \lrar \E$ is a cofibration. Because $\Cat_{\bS}$ is left proper and $[1]_{\bS}^\sim\simeq \ast$, we have that
$ \Sig(\ast \coprod \ast) \simeq \ast \coprod_{\ast \coprod \ast} \E .$
The maps $\ast \coprod \ast \lrar [1]_{\bS} \x{\eta}{\lrar} \E \x{\simeq}{\lrar} \ast$ now induce a map 
\begin{equation}\label{e:nu}
\nu: \ast\coprod_{[1]_{\bS}}\A \cong \ast \coprod_{\ast \coprod \ast} [1]_{\bS} \lrar \ast \coprod_{\ast \coprod \ast} \E \simeq \Sig(\ast \coprod \ast)
\end{equation}
in $(\Cat_{\bS})_{\ast//\ast}$. We note that the $\bS$-enriched category $\ast \coprod_{\ast \coprod \ast} [1]_{\bS}$ can be identified with the free $\bS$-enriched category generated by a single object and single endomorphism of that object, while $\ast \coprod_{\ast \coprod \ast} \E$ is the free $\bS$-enriched category generated by a single object and single \textbf{self-equivalence} of that object. The map $\nu$ is the natural map between these two universal objects. Since $\ovl{L}_\ast[1] \simeq \ovl{\Sig}^{\infty}(\Sig(\ast \coprod \ast))$, it will now suffice to show that the map
$$ \ovl{\Sig}^{\infty}(\nu): \ovl{\Sig}^{\infty}\left(\ast \coprod_{\ast \coprod \ast} [1]_{\bS}\right) \lrar \ovl{\Sig}^{\infty}\left(\ast \coprod_{\ast \coprod \ast} \E\right) $$
is a strong equivalence of suspension spectra (see Definition~\ref{d:suspension}). In fact, we will show that $\nu$ becomes a weak equivalence after a single suspension, or, equivalently, that for every $\D \in (\Cat_{\bS})_{\ast//\ast}$ the induced map
\begin{equation}\label{e:map}
\nu^*: \Map^{\der}_{(\Cat_{\bS})_{\ast//\ast}}\left(\ast \coprod_{\ast \coprod \ast} \E,\D\right) \lrar \Map^{\der}_{(\Cat_\bS)_{\ast//\ast}}\left(\ast \coprod_{\ast \coprod \ast} [1]_{\bS},\D\right) 
\end{equation} 
of pointed spaces becomes a weak equivalence after looping. For this it will suffice to show that~\eqref{e:map} is a $(-1)$-truncated map of spaces (after forgetting the base point), i.e., that each of its homotopy fibers is either empty or contractible. In other words, we will show that the map $\ast \coprod_{\ast \coprod \ast} [1]_{\bS}\lrar \ast \coprod_{\ast \coprod \ast} \E$ is \textbf{$(-1)$-cotruncated}. To this end, observe that the latter map is the homotopy cobase change in $(\Cat_{\bS})_{\ast//\ast}$ of the map $[1]_{\bS} \coprod \ast \lrar \E \coprod \ast$, and so it will suffice to show that the map $\eta: [1]_{\bS} \lrar \E$ is $(-1)$-cotruncated in $(\Cat_{\bS})_{/\ast}$. Since $\eta$ is a cofibration between cofibrant objects this is equivalent to the assertion that the fold map $\E\coprod_{[1]_{\bS}} \E\lrar \E$ is a weak equivalence in $(\Cat_{\bS})_{/\ast}$, or, equivalently, that any of the two canonical maps $\E \lrar \E \coprod_{1_{\bS}} \E$ is a weak equivalence in $(\Cat_{\bS})_{/\ast}$ (or in $\Cat_{\bS}$). But this now follows from the \textbf{invertibility hypothesis} assumed on $\bS$ (see~\cite[Definition A.3.2.16]{Lur09}) since $\eta$ classifies a morphism of $\E$ which is invertible in $\Ho(\E)$ (see~\cite[Remark A.3.2.14]{Lur09}).
\end{proof}

It will be useful to record the following enhanced version of Proposition~\ref{p:conceptual}, which allows one to compute relative cotangent complexes as well. We first note that the Quillen equivalence of Theorem~\ref{t:comp-cat} is natural in $\C$. Indeed, if $f: \C \lrar \D$ is a map of $\bS$-enriched categories and $\vphi: \C^{\op} \otimes \C \lrar \D^{\op} \otimes \D$ is the induced map then we have a commutative square of Quillen adjunctions
$$ \xymatrix{
(\Cat_{\bS})_{\C//\C} \ar@<-1ex>[d]_-{f_!}\ar@<-1ex>[r]_-{\G^{\C}_{\aug}}  &
\Fun(\C^{\op} \otimes \C,\bS)_{\Map_{\C}//\Map_{\C}} \ar@<-1ex>[l]_-{\F^{\C}_{\aug}}\ar@<-5.2ex>[d]_-{\vphi_!} \\
(\Cat_{\bS})_{\D//\D}\ar@<-1ex>[u]^-{\dashv}_{f^*}\ar@<-1ex>[r]_-{\G^{\D}_{\aug}} &
\Fun(\D^{\op} \otimes \D,\bS)_{\Map_{\D}//\Map_{\D}}\ar@<-1ex>[l]_-{\F^{\D}_{\aug}}\ar@<3.2ex>[u]^-{\dashv}_{\vphi^*} 
}$$
Here $f_! \dashv f^*$, as in \S\ref{s:tangent}, is the adjunction induced on over-under objects by the identity adjunction of $\Cat_{\bS}$ (see~\eqref{e:over-under-2}), and $\vphi_! \dashv \vphi^*$ is the adjunction induced on over-under objects by the restriction-left Kan extension adjunction $\Fun(\C^{\op} \otimes \C,\bS) \adj \Fun(\D^{\op} \otimes \D,\bS)$. 
Applying the stabilization functor we obtain a commutative diagram of Quillen adjunctions
\begin{equation}\label{e:square-tangent}
\xymatrix{
\T_\C\Cat_{\bS}\ar@<-1ex>[d]_-{f^{\Sp}_!}\ar@<-1ex>[r]_-{\G^{\C}_{\Sp}}^-{\simeq}  &
\T_{\Map_\C}\Fun(\C^{\op} \otimes \C,\bS) \ar@<-1ex>[l]_-{\F^{\C}_{\Sp}}\ar@<-1ex>[d]_-{\vphi^{\Sp}_!} \\
\T_\D\Cat_{\bS}\ar@<-1ex>[u]^-{\dashv}_{f^*_{\Sp}}\ar@<-1ex>[r]_-{\G^{\D}_{\Sp}}^-{\simeq} &
\T_{\Map_\D}\Fun(\D^{\op} \otimes \D,\bS)\ar@<-1ex>[l]_-{\F^{\D}_{\Sp}}\ar@<-1ex>[u]^-{\dashv}_{\vphi^*_{\Sp}} 
}
\end{equation}
where the horizontal Quillen adjunctions are the Quillen equivalences of Theorem~\ref{t:comp-cat} associated to $\C$ and $\D$ respectively. We then have the following generalization of Proposition~\ref{p:conceptual}:
\begin{cor}\label{c:enhanced}
Let $f: \C \lrar \D$ be a map of $\bS$-enriched categories.  
Then there is a natural weak equivalence
$$ \theta_{f}:\LL\vphi^{\Sp}_!(L_{\Map_\C}[-1]) \x{\simeq}{\lrar} \G^{\D}_{\Sp}\LL f_!^{\Sp}(L_{\C}) $$
in the model category $\T_{\Map_{\D}}\Fun(\D^{\op} \otimes \D,\bS)$.
\end{cor}
\begin{proof} 
By Proposition~\ref{p:conceptual} we have a natural weak equivalence $\theta_\C:L_{\Map_\C}[-1] \x{\simeq}{\lrar} \G^{\C}_{\Sp}(L_{\C})$, and since $\F^{\C}_{\Sp} \dashv \G^{\C}_{\Sp}$ is a Quillen equivalence we may consider instead the adjoint weak equivalence $\theta^{\ad}_{\C}:\LL\F^{\C}_{\Sp}(L_{\Map_\C}) \x{\simeq}{\lrar} L_{\C}[1]$. Using the commutativity of~\eqref{e:square-tangent} we obtain a natural weak equivalence
$$\xymatrix@C=3pc{ 
\LL\F^{\D}_{\Sp}\LL\vphi_!^{\Sp}(L_{\Map_\C}) \simeq \LL f^{\Sp}_!\LL\F^{\C}_{\Sp}(L_{\Map_\C}) \ar[r]^-{\LL f^{\Sp}_!\theta^{\ad}_{\C}}_-{\simeq} & \LL f^{\Sp}_!(L_\C[1])}$$
Using the fact that $\F^{\D}_{\Sp} \dashv \G^{\D}_{\Sp}$ is a Quillen equivalence and $\G^{\D}_{\Sp}$ preserves weak equivalences we then obtain an adjoint equivalence
$$ \theta_{f} : \LL \vphi_!^{\Sp}(L_{\Map_\C}[-1]) \x{\simeq}{\lrar} \G^{\D}_{\Sp}\LL f_!^{\Sp}(L_{\C})$$
as desired.
\end{proof}

\begin{cor}\label{c:relative-2}
Let $f: \C \lrar \D$ be a map of $\bS$-enriched categories.  
Then there is a natural homotopy cofiber sequence
$$\G^{\D}_{\Sp}(L_{\D/\C})\lrar  \LL\vphi^{\Sp}_!(L_{\Map_\C}) \lrar L_{\Map_\D}$$
in the model category $\T_{\Map_\D}\Fun(\D^{\op} \otimes \D,\bS)$.
\end{cor}
\begin{proof}
By Corollary \ref{c:relative} the middle term of the above sequence can be identified with the $\G^{\D}_{\Sp}\LL \vphi_!^{\Sp}(L_{\C}[1])$, while the last term is given by $\G^{\D}_{\Sp}(L_{\D}[1])$ by Proposition \ref{p:conceptual}. This identifies the above sequence with the image of the cofiber sequence $L_{\D/\C}\lrar \LL \vphi_!^{\Sp}(L_{\C}[1])\lrar L_{\D}[1]$ under the equivalence $\G^{\D}_{\Sp}$.
\end{proof}

Let us now exploit Corollary~\ref{c:relative-2} to compute the cotangent complex of \textbf{associative algebras} in $\bS$.
The functor $\L: \Alg(\bS)\lrar (\Cat_{\bS})_{\ast/}$ of~\eqref{e:L} sends an associative algebra object $A$ to the pointed $\bS$-enriched category consisting of a single object whose endomorphism algebra is $A$. Adopting a similar notation as in Section \ref{s:groups} let us denote this pointed category by $\BB A_{\ast} = \L(A) \in (\Cat_{\bS})_{\ast/}$ and denote by $\BB A$ the underlying unpointed category of $\BB A_{\ast}$. By the commutation of $\Sig^{\infty}_+$ with left Quillen functors we may deduce that $L_A$ and $L_{\BB A_{\ast}}$ have equivalent images in $\T_{\BB A}\Cat_{\bS}$, and by Proposition~\ref{p:relative} this image can be identified with the relative cotangent complex of the base point inclusion $\ast \lrar \BB A$. 

We note that $\Fun(\BB A^{\op} \otimes \BB A,\bS) \cong \BiMod_A(\bS)$ can simply be identified with the category of $A$-bimodules. Let us denote the underlying $A$-bimodule of $A$ by ${}_A A_{A}$ and the underlying left (resp. right) $A$-module of $A$ by ${}_A A$ (resp. $A_{A}$). Using Corollary~\ref{c:comp-lifts} we may identify the tangent model category $\T_A\Alg(\bS)$ with the fiber of $\BiMod_A(\T\bS) \lrar \BiMod_A(\bS)$ over the $A$-bimodule ${}_A A_{A}$. Applying Corollary~\ref{c:relative-2} to the inclusion $g:\ast \lrar \BB A$ we now obtain the following corollary:
\begin{cor}\label{c:famous-cofiber}
Let $A$ be an associative algebra object in $\bS$. There exists a natural homotopy cofiber sequence 
\begin{equation}\label{e:cofiber-matan-2}
\G^{\BB A}_{\Sp}L_A \lrar \LL\Sig^{\infty}_{+}({}_A A \otimes A_A) \lrar \LL\Sig^{\infty}_{+}({}_A A_A)
\end{equation}
in the model category $\Sp(\BiMod_A(\bS)_{A//A})$.
\end{cor}
\begin{proof}
Let $\vphi: \ast \lrar \BB A^{\op} \otimes \BB A$ be the map induced by $\ast \lrar \BB A$. Then $\vphi_!: \bS \lrar \BiMod_A$ is just the free $A$-bimodule functor, and hence $\vphi_!(1_{\bS}) \cong {}_A A \otimes A_A$. The result is now revealed as a particular case of Corollary~\ref{c:relative-2}.
\end{proof}

\begin{rem}
When $\bS$ is \textbf{stable}, the cofiber sequence~\ref{e:cofiber-matan-2} can be written as
\begin{equation}\label{e:cofiber-matan}
L_A \lrar A^{\op} \otimes A \lrar A
\end{equation}
where we view all objects as $A$-bimodules. This is the $n=1$ case of the cofiber sequence appearing in~\cite[Theorem 7.3.5.1]{Lur14} and in~\cite[Theorem 1.1]{Fra13}. When tensored with the $A$-bimodule $A$ one obtained a long exact sequence relating the Quillen cohomology and \textbf{Hochschild cohomology} of $A$.
\end{rem}

\begin{example}
When $\bS = \C(k)$ is the category of chain complexes over a field $k$, $A$ is a discrete algebra and $M$ is a discrete $A$-bimodule, the cofiber sequence~\eqref{e:cofiber-matan} identifies the Quillen cohomology groups $\rH^n_Q(A,M)$ for $n \geq 1$ with the Hochschild cohomology group $\rH\rH^{n+1}(A,M)$. For $n=0$ we obtain instead a surjective map $f_0:\rH^0_Q(A,M) \lrar \rH\rH^1(A,M)$. Unwinding the definitions we see that $\rH^0_Q(A,M)$ is the group of derivations $A \lrar M$, $\rH\rH^1(A,M)$ is the group of derivations modulo the inner derivations, and $f_0$ is the natural map between these two types of data. 
\end{example}

\begin{example}\label{e:monoid}
Let $\bS$ be the category of simplicial sets with the Kan-Quillen model structure. Then associative algebras in $\bS$ are the same as simplicial monoids. As explained above we may identify the tangent model category $\T_A\Alg(\bS)$ at a given monoid $A$ with the fiber of $\BiMod_A(\T\bS) \lrar \BiMod_A(\bS)$ over the $A$-bimodule $\ovl{A}$. 
Unwinding the definitions we may describe such objects as parametrized spectra $\{Z_a\}_{a \in A}$ over $A$, together with a suitably compatible collections of maps $Z_a \lrar Z_{bac}$ for $b,c \in A$. In other words, we may consider objects in $\T_A\Alg(\bS)$ as \textbf{$(A^{\op} \times A)$-equivariant parametrized spectra over $A$} (in the naive sense). 
Under this identification, the right most term in~\eqref{e:cofiber-matan-2} is the constant sphere spectrum $a \mapsto \Sig^{\infty}_+(\{a\}) = \mathbb{S}$ and the middle term is the parametrized spectrum $a \mapsto \Sig^{\infty}_+(m^{-1}(a))$ where $m^{-1}(a)$ denotes the homotopy fiber of the multiplication map $m: A \times A \lrar A$ over a point $a \in A$. We may thus identify the cotangent complex of $A$ with the equivariant family $\{Z_a\}_{a \in A}$ in which $Z_a$ is given by the homotopy fiber of the map of spectra $\Sig^{\infty}_+(m^{-1}(a)) \lrar \Sig^{\infty}_+(\{a\})$.  
One can also think of $Z_a$ as the ``coreduced'' suspension spectrum of the homotopy fiber $m^{-1}(a)$. 
\end{example}

\subsection{Simplicial categories and $\infty$-categories}\label{s:simp-categories}
In this subsection we will consider in further detail the example where $\bS$ is the category of simplicial sets endowed with the Kan-Quillen model structure. In this case $\Cat_{\bS}$ is a model for the theory of $\infty$-categories, for which a well-developed theory is available, notably in the setting of Joyal's model structure on simplicial sets. We will use this theory to give a simplified description of the tangent $\infty$-category $\T_{\C}\Cat_\infty$ at a fixed $\infty$-category $\C$ in terms of functors out of its \textbf{twisted arrow category}. As an application, we show how this description can be used to give an obstruction theory for splitting homotopy idempotents.   

Recall from~\cite{Lur09} that we have a Quillen equivalence
$$ \fC: \Set_{\Del} \adj \Cat_{\bS}: \rN $$
where $\Set_{\Del}$ is the category of simplicial sets endowed with the Joyal model structure. 
If $\C$ is a fibrant simplicial category, then the counit map $\eps: \fC(\rN\C) \lrar \C$ is a weak equivalence, in which case the natural map $\eps':\fC(\rN\C^{\op} \times \rN\C) \lrar \C^{\op} \times \C$ is a weak equivalence as well. The straightening and unstraightening functors of~\cite[\S 2.2]{Lur09} then give a Quillen equivalence
$$ \St_{\eps'}: (\Set_{\Del})^{\cov}_{/\rN\C{}^{\op} \times \rN\C} \adj \Fun(\C^{\op} \times \C, \bS) : \Un_{\eps'} $$
where $(\Set_{\Del})^{\cov}_{/\rN\C^{\op} \times \rN\C}$ is the category of simplicial sets over $\rN\C^{\op} \times \rN\C$ endowed with the covariant model structure (see~\cite[\S 2]{Lur09}). Let $\Tw(\rN\C)$ be the twisted arrow category of $\rN\C$, equipped with its canonical left fibration $m:\Tw(\rN\C) \lrar \rN\C^{\op} \times \rN\C$ (see~\cite[Construction 5.2.1.1, Proposition 5.2.1.3]{Lur14}, and note that we are using the opposite convention of loc.cit.). We then have a weak equivalence $\beta: \St_{\eps'}(m) \x{\simeq}{\lrar} \Map_{\C}$ (see \cite[Proposition 5.2.1.11]{Lur14}). It follows that the induced adjunction
$$ (\St_{\eps'})_{\beta//\beta}: ((\Set_{\Del})^{\cov}_{/\rN\C^{\op} \times \rN\C})_{m//m} \adj \Fun(\C^{\op} \times \C, \bS)_{\Map_{\C}//\Map_{\C}} : (\Un_{\eps'})_{\beta//\beta} $$
is a Quillen equivalence as well. 
Theorem~\ref{t:comp-cat} now implies that the tangent model category $\T_\C\Cat_{\bS}$ is Quillen equivalent to $\Sp((\Set_{\Del}^{\cov}{}_{/\rN\C^{\op} \times \rN\C})_{m//m})$. Let us now consider the category $(\Set_{\Del})^{\cov}_{/\Tw(\rN\C)}$ of simplicial sets over $\Tw(\rN\C)$ endowed with the covariant model structure. The left Quillen functor $(\Set_{\Del})^{\cov}_{/\Tw(\rN\C)} \lrar (\Set_{\Del})^{\cov}_{/\rN\C^{\op} \times \rN\C}$ postcomposing with $m$ naturally lifts to a left Quillen functor 
$$ m_!:(\Set_{\Del})^{\cov}_{/\Tw(\rN\C)} \lrar \left((\Set_{\Del})^{\cov}_{/\rN\C^{\op} \times \rN\C}\right)_{/m} .$$  
Furthermore, $m_!$ is an equivalence on the underlying categories, which is in fact a left Quillen equivalence. Indeed, both model structures have the same cofibrations and since every object is cofibrant we see that $m_!$ preserves weak equivalences. We may therefore consider $m_!$ as a left Bousfield localization functor. Since $m_!$ detects weak equivalences between fibrant objects by~\cite[Remark 2.2.3.3]{Lur09}, it follows that this left Bousfield localization must be an equivalence. Using the sequence of Quillen equivalences
$$ (\Set_{\Del})^{\cov}{}_{/\Tw(\rN\C)} \x{\simeq}{\adj} \left((\Set_{\Del})^{\cov}_{/\rN\C^{\op} \times \rN\C}\right)_{/m} \x{\simeq}{\adj} \Fun(\C^{\op} \times \C, \bS)_{/\Map_{\C}} $$
we conclude from Theorem~\ref{t:comp-cat} that the tangent model category $\T_\C\Cat_{\bS}$ is Quillen equivalent to $\Sp\left(\left(\Set_{\Del}^{\cov}{}_{/\Tw(\rN\C)}\right)_{\ast}\right)$. Furthermore, Proposition~\ref{p:conceptual} implies that the image of the cotangent complex $L_\C$ in the model category $\Sp\left(\left(\Set_{\Del}^{\cov}{}_{/\Tw(\rN\C)}\right)_{\ast}\right)$ is weakly equivalent to the shifted cotangent complex $L_{\Tw(\rN \C)}[-1]$ of the object $\Tw(\rN \C)$, considered as a (final) object of the covariant model category $(\Set_{\Del})^{\cov}{}_{/\Tw(\rN\C)}$. Since $(\Set_{\Del}^{\cov}{}_{/\Tw(\rN\C)})_{\infty}$ is equivalent to the $\infty$-category of functors from $\Tw(\rN\C)$ to the $\infty$-category $\bS_{\infty}$ of spaces, the above considerations can be summarized by the following corollary (taking into account Corollary~\ref{c:conceptual}):
\begin{cor}\label{c:twisted-arrow}
Let $\C$ be a fibrant simplicial category. Then the underlying $\infty$-category of $\T_{\C}\Cat_{\bS}$ is equivalent to the $\infty$-category of functors
$$ \Tw(\rN\C) \lrar \Sp(\bS_{\infty}) = \Spectra $$
from the twisted arrow category of $\rN\C$ to the $\infty$-category of spectra. The cotangent complex of $\C$ is identified with the constant functor $\Tw(\rN\C) \lrar \Spectra$ on the desuspension $\mathbb{S}[-1]$ of the sphere spectrum. 
\end{cor}

\begin{cor}\label{c:twisted-arrow-quillen}
Let $\F: \Tw(\rN\C) \lrar \Spectra$ be a functor and let $M_{\F} \in \Sp((\Cat_{\bS})_{\C//\C})$ be the corresponding object under the equivalence of Corollary~\ref{c:twisted-arrow}. Then the Quillen cohomology group $\rH^n_Q(\C;M_{\F})$ is naturally isomorphic to the $(-n-1)$'th homotopy group of the spectrum $\lim\F$. In particular, if $\C$ is a discrete category and $\F$ is a diagram of Eilenberg-MacLane spectra corresponding to a functor $\F': \Tw(\C) \lrar \Ab$, then the Quillen cohomology group $\rH^n_Q(\C;M_{\F})$ is naturally isomorphic to the $(n+1)$'th derived functor $\lim^{n+1}\F'$.
\end{cor}
\begin{proof}
By definition we have $\rH^n_Q(\C;M_{\F}) = \pi_0\Map^{\der}_{\T_\C\Cat_{\bS}}(L_{\C},M_{\F}[n])$. By Corollary~\ref{c:twisted-arrow} this can be identified with
$$ \pi_0\Map^{\der}_{\Fun(\Tw(\rN\C),\Sp(\bS_{\infty}))}(\ovl{\mathbb{S}}[-1],\F[n]) \simeq $$
$$ \simeq \pi_0\Map^{\der}_{\Fun(\Tw(\rN\C),\Sp(\bS_{\infty}))}(\ovl{\mathbb{S}}[-n-1],\F) \cong \pi_{-n-1}\lim\F $$
where $\ovl{\mathbb{S}}$ denotes the constant diagram with value the sphere spectrum.
\end{proof}



\begin{rem}\label{r:beamer}
Given a diagram $M: \Tw(\C) \lrar \Sp$ the $\infty$-category $\Om^{\infty}(M)$ can be described informally as the $\infty$-category whose 
\begin{itemize}
\item
objects are pairs $(X,\eta)$ with $X \in \C$ and $\eta$ is a map $\eta:\SS[-1] \lrar M(\Id_X)$. 
\item
Maps from $(X,\eta)$ to $(X',\eta')$ are pairs $(f,H)$ where $f:X \lrar X'$ is a map in $\C$ and $H$ is a homotopy between the two resulting maps $f_*\eta,f^*\eta': \SS[-1] \lrar M(f)$. 
\end{itemize}
Similarly, given an $\alp \in \Map(\SS[-1],\lim_{\Tw(\C)}M[1]) \simeq \Map_{\C}(\C,\Om^{\infty}M[1])$, we can describe the small extension $p_\alp:\C_\alp \lrar \C$ corresponding to $\alp$ as follows: the objects of $\C_\alp$ are pairs $(X,\eta)$ with $X \in \C$ and $\eta$ is a null-homotopy of the $\Id_X$ component $\alp_{\Id_X}:\SS[-1] \lrar M(\Id_X)[1]$ of $\alp$. Maps from $(X,\eta)$ to $(X',\eta')$ are pairs $(f,H)$ where $f:X \lrar X'$ is a map in $\C$ and $H$ is a homotopy between the two resulting null homotopies $f_*\eta,f^*\eta'$ of $\alp_f: \SS[-1] \lrar M(f)[1]$.
\end{rem}

\begin{example}\label{e:joost-2}
Let $\C$ be a stable $\infty$-category. Consider the functor given informally by
$$
M:\Tw(\C\times \C) \lrar \Sp; \hspace{4pt} [(X,Y) \lrar (Z,W)] \mapsto \uline{\Map}(Y,Z)$$
where $\uline{\smash{\Map}}(-,-)$ denotes the canonical enrichment of $\C$ is spectra. Using Remark~\ref{r:beamer} we may identify $\Om^{\infty}(M) \lrar \C \times \C$ with the projection $\pi:\D \lrar \C \times \C$ where $\D$ is the $\infty$-category of fiber sequences in $\C$ and $\pi(X \lrar E \lrar Y) = (X,Y)$. 

Similarly, if we let $\E$ denote the $\infty$-category of Cartesian squares in $\C$, then the functor $\E \lrar \C^{\Del^2}$ given by restricting along the inclusion $\Del^2 \subseteq \Del^1 \times \Del^1$ is a small extension. The coefficient object of this small extension is the functor 
$$
N: \Tw(\C^{\Del^2}) \lrar \Sp; \hspace{4pt} \vcenter{\xymatrix@1@C=1.5pc@R=1.1pc{(A\ar[r]\ar[d] & B\ar[r]^f\ar[d] & C\ar[d])\\ (X\ar[r]_g & Y\ar[r] & Z)}}
\mapsto \uline{\Map}(\cof(f),\fib(g)).
$$
The class $\alp$ is given by the shifted section which sends $(A \lrar B \x{f}{\lrar} C) \lrar (X \x{g}{\lrar} Y \lrar Z)$ to the composite
$\cof(f)[-1] \simeq \fib(f) \lrar B \lrar Y \lrar \cof(g) \simeq \fib(g)[1]$.
\end{example}

\begin{rem}
When $\C$ is a simplicial category of the form $\BB A$ for a fibrant simplicial monoid $A$ the twisted arrow category $\Tw(\rN \BB A)$ can be loosely described as the $\infty$-category whose objects are the points $a \in A$ and such that morphisms from $a$ to $a'$ are given by a pair of points $b,c$ and a path from $bac$ to $a'$ in $A$. In this case we may identify functors from $\Tw(\rN \BB A)$ to spectra as $(A^{\op} \times A)$-equivariant parametrized spectra over $A$ (see Example~\ref{e:monoid}). In the special case where $A = G$ is a \textbf{simplicial group} the $\infty$-category $\rN \BB G$ is a Kan complex and the projection $\Tw(\rN\BB G) \lrar \rN\BB G$ is an equivalence. We may then identify functors $\Tw(\rN\BB G) \lrar \Spectra$ with (naive) $G$-equivariant spectra.   
Though this is consistent with the computation of \S\ref{s:groups}, we warn the reader that the equivalence between the tangent category at $G$ and $G$-spectra obtained in this way differs from the corresponding equivalence obtained in \S\ref{s:groups} by a shift. Indeed, given a spectrum object in $\sGr_{G//G}$ we may generate from it either a parametrized spectrum over the classifying space of $G$ using the functor $\ovl{W}$, or a parametrized spectrum over the underlying space of $G$, by using the forgetful functor $\sGr \lrar \bS$. Identifying the forgetful functor with the objectwise loop of $\ovl{W}$ we see that the fibers of the latter, which is used in the comparison above, are the shifts of the fibers of the former, on which the comparison of~\S\ref{s:groups} is based.
In particular, when $A=G$ is a simplicial group the cofiber sequence~\eqref{e:cofiber-matan-2} reduces to a shift of the cofiber sequence~\eqref{e:cofiber-3} (see also Remark~\ref{r:central}).
\end{rem}

Now let $f: \C \lrar \D$ be a map of simplicial categories and let $\gam: \Tw(\rN \C) \lrar \Tw(\rN\D)$ and $\vphi: \C^{\op} \times \C \lrar \D^{\op} \times \D$ be the induced maps. Then we obtain a commutative diagram of left Quillen functors
$$ \xymatrix{
\Fun(\fC(\Tw(\rN\C)), \bS)\ar[d]_{\fC(\gamma)_!} & (\Set_{\Del})^{\cov}_{/\Tw(\rN\C)} \ar[l]^-\simeq_-{\St}\ar_{\gam_!}[d]\ar^-{\simeq}[r] & \Fun(\C^{\op} \times \C,\bS)_{/\Map_{\C}} \ar_-{\vphi_!}[d] \\
\Fun(\fC(\Tw(\rN\D)), \bS) & (\Set_{\Del})^{\cov}_{/\Tw(\rN\D)} \ar^-{\simeq}[r]\ar[l]_-\simeq^-{\St} & \Fun(\D^{\op} \times \D,\bS)_{/\Map_{\D}} \\
}$$
Here $\fC(\gam)_!$ is the left Kan extension functor, $\vphi_!$ is the functor induced by left Kan extension on over objects and $\gam_!$ is given by post-composing with $\gam$. In particular, $\LL\vphi_!\Map_\C \in \Fun(\D^{\op} \times \D,\bS)_{/\Map_{\D}}$ is weakly equivalent to the image of $\gam \in (\Set_{\Del})^{\cov}_{/\Tw(\rN\D)}$ under the bottom right horizontal equivalence, while $\Map_\D$ is weakly equivalent to the image of $\Id_{\Tw(\N\D)}$. Using this, the cofiber sequence of Corollary~\ref{c:relative-2} can be identified with the cofiber sequence in $\Sp((\Set_{\Del})^{\cov}_{/\Tw(\rN\D)})$ of the form
\begin{equation}\label{d:cofseq}
L'_{\D/\C} \lrar \Sig^{\infty}_+(\gam) \lrar \Sig^{\infty}_+(\Id_{\Tw(\N\D)}) 
\end{equation}
where $L'_{\D/\C}$ is the image of $L_{\D/\C}$ under the equivalence of Corollary~\ref{c:twisted-arrow}. We may therefore conclude the following:
\begin{cor}\label{c:coinitial}
Let $f: \C \lrar \D$ be a map of fibrant simplicial categories such that the induced map $\gam:\Tw(\rN\C) \lrar\Tw(\rN\D)$ is coinitial. Then the relative cotangent complex of $f$ vanishes.
\end{cor}

\begin{rem}\label{r:coinitial}
Recall that a map $p: X \lrar Y$ of simplicial sets is said to be \textbf{coinitial} if $p^{\op}$ is cofinal, i.e., if $p$ is equivalent to the terminal object in $(\Set_{\Del})^{\cov}_{/Y}$ (cf.~\cite[Definition 4.1.1.1]{Lur09}). This notion appears in the literature under various names, including \textit{right cofinal}, and \textit{initial}. By the $\infty$-categorical Quillen theorem A (see, e.g., \cite[Theorem 4.1.3.1]{Lur09}) a map $p: X \lrar Y$ where $Y$ is an $\infty$-category is coinitial if and only if for every object $y \in Y$ the simplicial set $X \times_{Y} Y_{/y}$ is weakly contractible.
\end{rem}

\begin{rem}
The cofiber sequence of~\eqref{d:cofseq} can also be straightened to obtain a cofiber sequence of functors $\fC(\Tw(\rN\D))\lrar \Sp(\bS_\ast)$ of the form
\begin{equation}\label{d:cofseq-2}
L''_{\D/\C} \lrar \fC(\gamma)_!(\ovl{\mathbb{S}})\lrar \ovl{\mathbb{S}}
\end{equation}
where $L''_{\D/\C}$ is the straightening of the object $L'_{\D/\C}$ appearing in~\eqref{d:cofseq} and $\ovl{\mathbb{S}}$ is the constant diagram on the sphere spectrum. Corollary~\ref{c:coinitial} can be seen in this context by using the straightened $\infty$-categorical Quillen theorem A (see Remark~\ref{r:coinitial}), namely, the fact that a map is coinitial if and only if the left Kan extension of the constant diagram is weakly constant.
\end{rem}
 
\begin{example}[Detecting equivalences]\label{e:joost}
Let $[1]_{\bS} = [1]_{\Del^0}$ (see Definition~\ref{d:1A}), let $[1]_{\bS}^{\sim}$ be the simplicial category with two objects $0,1$ and such that all mapping spaces are $\Del^0$ and let $[1]_{\bS} \lrar \E \lrar [1]_{\bS}^{\sim}$ be a factorization of the natural map $[1]_{\bS} \lrar [1]_{\bS}^{\sim}$ into a cofibration followed by a trivial fibration. Then the twisted arrow category of $\rN([1]_{\bS}) = \Del^1$ is the ``cospan category'' $\ast \lrar \ast \llar \ast$ (and is hence weakly contractible) and the twisted arrow category of $\rN\E$ is categorically equivalent to $\Del^0$ (and is hence ``strongly'' contractible). It then follows from the $\infty$-categorical Quillen theorem A (see Remark~\ref{r:coinitial}) that the induced map $\Tw(\Del^1) \lrar \Tw(\rN\E)$ is coinitial, and hence the map $[1]_{\bS} \lrar \E$ has a trivial relative cotangent complex by Corollary~\ref{c:coinitial}. Note that functors $[1]_{\bS} \lrar \C$ correspond to morphisms in $\C$, and that such a functor extends to $\E$ up-to-homotopy if and only if the corresponding morphism is invertible. As in Example~\ref{e:white-categories} let $P_2(\C)$ be the homotopy $(2,1)$-category of $\C$, so that the map $\C \lrar P_2(\C)$ can be decomposed as a tower of small extensions. Given a commutative square
\begin{equation}\label{e:square-5}
\vcenter{\xymatrix{
[1]_{\bS} \ar[r]\ar[d] & \C \ar[d] \\
\E \ar[r]\ar@{-->}[ur] & P_2(\C) \\
}}\end{equation}
Corollary~\ref{c:etale} implies that~\eqref{e:square-5} has a contractible space of derived lifts. In particular, this yields an obstruction theoretic proof of the (well-known) fact that a morphism in $\C$ is invertible if and only if it is invertible in the homotopy $(2,1)$-category of $\C$. We expect that a similar result can be obtained concerning the question of when a morphism in an $(\infty,2)$-category admits an adjoint. 
\end{example}

\begin{rem}
Example~\ref{e:joost} implies, in particular, that any localization map
$$ \C \lrar \C[W^{-1}] $$
of simplicial categories (or $\infty$-categories) has a trivial relative cotangent complex. Indeed, such a map can be obtained as an iterated pushout of the map $[1]_{\bS} \lrar \E$.
\end{rem}

The following example is inspired by ideas of Charles Rezk (see~\cite{rezk-mo}):
\begin{example}[Splitting of homotopy idempotents]\label{e:rezk}
Let $\Idem$ be the category with one object $x_0 \in \Idem$ and one non-identity morphism $f: x_0 \lrar x_0$ such that $f \circ f = f$. If $\C$ is a simplicial category then a derived map $\Idem \lrar \C$ corresponds to an object $x \in \C$ equipped with a \textbf{homotopy coherent idempotent} $x \lrar x$ (see~\cite[\S 4.4.5]{Lur09}). The twisted arrow category $\Tw(\Idem)$ admits the following explicit description: if we denote by $M = \End_{\Idem}(x_0) = \{1,f\}$, then $\Tw(\Idem)$ is the category with two objects $\ovl{1},\ovl{f} \in \Tw(\Idem)$ (corresponding to the morphisms $1,f$ of $\Idem$), and such that $\End_{\Tw(\Idem)}(\ovl{f}) = M \times M$, $\End_{\Tw(\Idem)}(\ovl{1}) = 1$ and $\Hom_{\Tw(\Idem)}(\ovl{1},\ovl{f}) = M \times M \setminus \{(1,1)\}$. If $\F: \Tw(\Idem) \lrar \Ab$ is a functor from the twisted arrow category to abelian groups then $\F(\ovl{f})$ is an $M \times M$-module and we may canonically decompose it as 
$$\F(\ovl{f}) = A_{0,0} \oplus A_{0,1} \oplus A_{1,0} \oplus A_{1,1}$$
such that the action of $(f,1)$ and $(1,f)$ on $A_{\lam,\mu}$ is given by multiplication by $\lam,\mu \in \{0,1\}$ respectively. Given an element $b \in \F(\ovl{1})$ the maps $(f,1),(1,f),(f,f):\ovl{1} \lrar \ovl{f}$ 
send $b$ into $A_{1,0} \oplus A_{1,1}$, $A_{0,1} \oplus A_{1,1}$ and $A_{1,1}$, respectively, and the composition rule of $\Tw(\Idem)$  translates into the condition that the projection of the image of $b$ to $A_{1,1}$ should be the same in all three cases. In particular, we may identify the category $\Fun(\Tw(\Idem),\Ab)$ with the category of tuples $(B,A_{0,0},A_{0,1},A_{1,0}, A_{1,1},g_{0,1},g_{1,0},g_{1,1})$ where $g_{i,j}$ is a map from $B$ to $A_{i,j}$. Under this equivalence the functor $\lim: \Fun(\Tw(\Idem),\Ab) \lrar \Ab$ takes a tuple as above to 
$$\Ker[(g_{0,1},g_{1,0}): B \lrar A_{0,1} \oplus A_{1,0}] .$$ 
It follows that the total derived functor of $\lim$ takes a tuple as above to the \textbf{chain complex} $B \lrar A_{0,1} \oplus A_{1,0}$ where $B$ sits in degree $0$ and $A_{0,1} \oplus A_{1,0}$ sits in degree $-1$. In particular, the higher derived functors $\lim^n$ vanish for $n \neq 0,1$. By Corollary~\ref{c:twisted-arrow-quillen} we may conclude that if $M \in \Sp((\Cat_{\bS})_{\Idem//\Idem})$ is an object corresponding to a functor $\Tw(\Idem) \lrar \Ab$ then the Quillen cohomology groups $\rH^n_Q(\Idem;M)$ vanish for $n \neq -1,0$. This observation has the following concrete consequence: recall that if $\C$ is a simpicial category and $P_2(\C)$ is the homotopy $(2,1)$-category of $\C$ then a derived map $\Idem \lrar P_2(\C)$ corresponds to an object $x \in P_2(\C)$ together with a map $f': x \lrar x$ and a homtopy $h: f' \circ f' \Rightarrow f'$ such that the diagram 
\begin{equation}\label{e:coherent}
\vcenter{\xymatrix{
f' \circ f' \circ f' \ar^{h \circ f'}[r]\ar_{f' \circ h}[d] & f' \circ f' \ar^{h}[d] \\
f' \circ f' \ar^{h}[r] & f' \\
}}
\end{equation}
commutes in the groupoid $\Map_{P_2(\C)}(x,x)$. Given such a ``partially coherent'' homotopy idempotent $\F: \Idem \lrar P_2(\C)$, the \textbf{derived space of lifts} $Z = \Map^{\der}_{/P_2(\C)}(\Idem,\C)$ in the diagram
$$ \xymatrix{
& \C \ar[d] \\
\Idem \ar@{-->}[ur] \ar[r] & P_2(\C) \\
}$$
can be considered as the space of fully coherent structures one can put on $\F$, extending the partial coherent structure we started with. When $\rN\C$ is idempotent complete (see~\cite[\S 4.4.5]{Lur09}) we can also identify $Z$ with the space of \textbf{splittings} of the homotopy idempotent $\F$ (i.e., retract diagrams in $\rN\C$ whose induced idempotent is $\F$). Factoring the map $\pi: \C \lrar P_2(\C)$ as a sequence of small extensions as in Example~\ref{e:white-categories} one may use the spectral sequence of Remark~\ref{r:spectral} to compute the homotopy groups of $Z$. By the above computation it follows that if $\F$ is a functor from $\Tw(\Idem)$ to spectra such that $\F(\ovl{1})$ and $\F(\ovl{f})$ are spectra which have trivial homotopy groups in dimension $\neq n$ then $\rH^k_Q(\Idem,\F)$ is trivial for $k \neq -n,-n-1$. Since the coefficient spectra appearing in the factorization of $\pi$ are exactly of this form this implies that the spectral sequence cannot support any non-trivial differentials for degree reasons, and hence collapses at the $E^1$-page. 

The coefficient spectra appearing in the $E_1$-term are all $2$-connective, so that all terms in $E^1$ only contribute to $\pi_n(Z)$ for $n \geq 2$, with only two terms contributing to each $\pi_n$. We conclude that the spectral sequence converges and that $Z$ is simply connected. The higher homotopy groups of $Z$ can be written down explicitly in terms of the Quillen cohomology groups of $\Idem$ with coefficients in the homotopy groups of the homotopy fibers of $\pi_*:\Map_{\C}(x,x) \lrar \Map_{P_2(\C)}(x,x)$ over the images of $1,f \in \End_{\Idem}(x_0)$ in $\Map_{P_2(\C)}(x,x)$ (and these Quillen cohomology groups themselves admit a very simple description, see above). In particular, every such partially coherent idempotent can be made fully coherent and  every two fully coherent refinements are equivalent (even more, since $Z$ is simply-connected every self-equivalence of a fully coherent refinement which induces the identity on the partial coherencies is equivalent to the identity). This also means that a homotopy idempotent $\Idem \lrar \Ho(\C) = P_1(\C)$ can be made coherent if and only if it lifts to $P_2(\C)$, i.e., if and only if $h$ can be chosen so that~\eqref{e:coherent} commutes.
\end{example}

\end{document}